\documentclass[a4paper,11pt]{amsart}

\input{Dynkin.sty}

\usepackage{pstricks,pst-plot}
\usepackage{amssymb,rotating}
\usepackage{xypic,amsfonts,amsmath}
\usepackage[french]{babel}
\usepackage{xypic}
\usepackage[all]{xy}
\usepackage{latexsym}
\usepackage{epsf}
\usepackage{textcomp}
\usepackage[latin1]{inputenc}
\usepackage[T1]{fontenc}
\usepackage{rotating}

\usepackage{mathrsfs}

\usepackage{hyperref}

\newcommand{\Rho}[1]{\hbox{\scriptsize{$\rho_{i,\lambda}^{(\varepsilon)}#1$}}}
\newcommand{\Rhop}[1]{\hbox{\scriptsize{$\rho_{i,\lambda}^{(\varepsilon')}#1$}}}

\newtheorem{theorem}{Th\'eor\`eme}[section]
\newtheorem{definition}[theorem]{D\'efinition}
\newtheorem{proposition}[theorem]{Proposition}
\newtheorem{question}[theorem]{Question}

\newtheorem{corollary}[theorem]{Corollaire}
\newtheorem{lemma}[theorem]{Lemme}
\newtheorem{remark}[theorem]{Remarque}
\newtheorem{example}[theorem]{Exemple}
\newtheorem{examples}[theorem]{Exemples}
\usepackage{pict2e}

\newtheorem{theo_alpha}{Théorème}
\def\theo_alpha{\Alph{theo_alpha}}

\newtheorem{defin_alpha}{Définition}
\def\defin_alpha{\Alph{defi_alpha}}

\newtheorem{conje_alpha}{Conjecture}
\def\conje_alpha{\Alph{conje_alpha}}

\newtheorem{quest_alpha}{Question}
\def\quest_alpha{\Alph{quest_alpha}}

\newtheorem{rem_alpha}{Remarque}
\def\rem_alpha{\Alph{rem_alpha}}

\def\C{\mathbb{C}} 
\def\Z{\mathbb{Z}} 
\def\R{\mathbb{R}} 
\def\N{\mathbb{N}} 
\def\Q{\mathbb{Q}}

\def\la{\langle}
\def\ra{\rangle}

\def\peq{\preccurlyeq}
\def\seq{\succcurlyeq}

\def\G{\mathbf{G}}    
 
\def\E{\mathbf{E}}
\def\P{\mathbf{P}}
\def\V{\mathbf{V}}
\def\W{\mathbf{W}}

\def\a{\alpha}
\def\b{\beta}
\def\ga{\gamma}
\def\Ga{\Gamma}

\def\de{\delta}

\def\veps{\varepsilon}

\def\l{\lambda}

\def\SS{\mathcal{S}}  
\def\I{\mathcal{I}}

\def\rk{{\rm rk}\hskip .1em}

\def\CC{\mathcal{C}}

\def\PA{\mathit{PA}}
\def\GA{\mathit{GA}}

\def\g{\mathfrak{g}}

\def\ll{\mathfrak{l}}

\def\m{\mathfrak{m}}
\def\n{\mathfrak{n}}
\def\ss{\mathfrak{s}}           
 
\def\lsl{\mathfrak{sl}}
\def\gl{\mathfrak{gl}}       
\def\so{\mathfrak{so}}
\def\sp{\mathfrak{sp}}
\def\Liea{\mathfrak{a}}

\def\gs{\mathfrak{g}^\mathfrak{s}}

\def\ad{\mathrm{ad}\,}

\def\gr{\mathop{\rm gr}\nolimits}

\def\id{\mathop{\rm id}\nolimits}

\def\diag{\mathop{\rm diag}\nolimits}
\def\dim{\mathop{\rm dim}\nolimits}
\def\codim{\mathop{\rm codim}\nolimits}
\def\Hom{\mathop{\rm Hom}\nolimits}
\def\End{\mathop{\rm End}\nolimits}

\def\Vect{\mathop{\rm Vect}\nolimits}

\def\le{\leqslant}
\def\ge{\geqslant}

\def\d{\mathrm{d}}
\def\T{\mathrm{T}}

\def\F{\mathcal{F}}
\def\FF{\mathscr{F}}

\def\GaD{\Ga_{{\rm Dyn}}^{h}}

\def\Wh{{\rm Wh}}

\begin{document}

\title[]{Paires admissibles d'une algèbre de Lie simple complexe et $W$-algèbres finies.}

\author[]{Guilnard Sadaka$^*$}
\thanks{$^*$La majorité de ce travail a été financée par le Centre National de la Recherche Scientifique Libanais.}
\address{Guilnard Sadaka, Laboratoire de Mathématiques et Applications,
Boulevard Marie et Pierre Curie,
86962 Futuroscope Chasseneuil Cedex,
France}

\email{guilnard.sadaka@math.univ-poitiers.fr}

\address{\hspace{3cm}Laboratoire de Mathématiques de Versailles,
45 Avenue des Etas-Unis,
78035 Versailles Cedex,
France}
\email{guilnard.sadaka@uvsq.fr}


\maketitle

\begin{abstract}

Soient $\g$ une algèbre de Lie simple complexe et
$e$ un élément nilpotent de $\g$. Dans cet article, on s'intéresse au problème,
soulevé par Premet, d'isomorphisme entre les $W$-algèbres finies construites à partir de
certaines sous-algèbres nilpotentes de $\g$, dites {\em $e$-admissibles}.
On considère une version graduée de ce problème et on introduit les notions de paire
et graduation $e$-admissibles. On montre que la $W$-algèbre associée à une paire
$e$-admissible possède des propriétés similaires à celle introduite par Gan et Ginzburg.
De plus, on définit une relation d'équivalence sur l'ensemble des paires admissibles et
montre que si deux paires sont équivalentes, alors les $W$-algèbres
associées sont isomorphes. En introduisant la notion de la connexité pour les
graduations $e$-admissibles, on réduit le problème d'isomorphisme à
l'étude de l'équivalence des paires admissibles pour une graduation admissible
fixée. Ceci nous permet de montrer ensuite que les paires admissibles relativement aux 
graduations {\it $b$-optimales} sont équivalentes. On retrouve comme cas particulier
un résultat de Brundan et Goodwin. Dans la dernière
partie, on utilise nos résultats pour résoudre complètement
le problème d'isomorphisme dans quelques cas particuliers.

\smallskip

\noindent {\tiny ABSTRACT.} Let $\g$ be a complex simple Lie algebra and $e$ a nilpotent element of $\g$.
We are interested in the isomorphism question (raised by Premet)
between the finite $W$-algebras constructed from some nilpotent subalgebras
of $\g$ called {\it $e$-admissible}.
We introduce the concept of {\it $e$-admissible pair} and {\it $e$-admissible grading}. 
We show that the $W$-algebra associated to an $e$-admissible pair admits
similar properties to the ones introduced by Gan and Ginzburg.
Moreover, we define an equivalence relation on the set of admissible pairs
and we show that if two admissible pairs are equivalent, it follows that
the associated $W$-algebras are isomorphic. By introducing the notion of connectivity of admissible
gradings, we reduce the isomorphism question to the study of the equivalence of admissible pairs
for a fixed admissible grading. This allows us to prove that admissible pairs relative to {\it $b$-optimal} gradings are equivalent, hence the corresponding $W$-algebras are isomorphic.
We recover as a special case a result of Brundan and Goodwin.
In the final part, we use our results to find a complete answer to the isomorphism question
in some particular cases.

\end{abstract}
\section{Introduction}
\subsection{}Soient $\Bbbk$ un corps algébriquement clos de caractéristique nulle 
et $\g$ une algèbre de Lie simple 
de dimension finie définie sur $\Bbbk$ et de groupe 
adjoint $G$. On note $\la.\,,.\ra$ la forme de Killing de $\g$. Comme celle-ci 
est non dégénérée sur $\g\times\g$, elle induit un isomorphisme 
$$\kappa:\;\g \longrightarrow \g^*,\qquad x\longmapsto \la x,.\ra$$
de $\g$ sur son dual $\g^*$.  
Soient $e$ un élément nilpotent de $\g$ et $\chi$ la forme linéaire
$\kappa(e)$. D'après le Théorème de Jacobson-Morosov 
(cf. e.g.~\cite[Theorem 32.1.5]{TY}), 
$e$ est contenu dans un $\lsl_2$-triplet $(e,h,f)$ de $\g$. 
On rappelle qu'un tel triplet vérifie les relations:
$$
  [h,e]=2e,\qquad [e,f]=h,\qquad [h,f]=-2f.
$$
De plus, $e$ et $f$ sont dans la même $G$-orbite et $\ad h$ est un élément semi-simple dont les 
valeurs propres sont entières. 
Ceci définit sur $\g$ une $\Z$-graduation appelée une {\em graduation de Dynkin associée à $e$}.

On pose
$$
  \mathfrak{p}_+:=\bigoplus_{j \ge 0} \g(j)
$$
où $\g(j):=\{x \in \g \; ; \; [h,x] = jx \}$ pour tout $j \in \Z$. 
Alors $\mathfrak{p}_+$ est une sous-algèbre parabolique de $\g$ 
qui contient $\g^{e}$, le centralisateur de $e$ dans $\g$. 
Il en résulte que la forme bilinéaire 
$$
  \g(-1)\times\g(-1)\longrightarrow \Bbbk,\qquad (x,y)\longmapsto \la e,[x,y]\ra
$$
est non dégénérée. 
Soit $\ll$ un sous-espace lagrangien de $\g(-1)$ relativement à cette forme, 
c'est-à-dire un sous-espace totalement isotrope de dimension maximale, et posons:
\begin{eqnarray}\label{eq:mle}
  \m_{\ll,e} := \ll \oplus \bigoplus_{j \le -2} \g(j).
\end{eqnarray}
Alors $\m_{\ll,e}$ est une sous-algèbre $\ad$-nilpotente de $\g$ qui vérifie
$\la e,[\m_{\ll,e},\m_{\ll,e}]\ra=\{0\}$. 
De plus, $\m_{\ll,e}\cap\g^{e}=\{0\}$ et $\dim\m_{\ll,e}=\frac{1}{2}\dim G \cdot e$ 
si $G \cdot e$ désigne l'orbite de $e$ dans $\g$ sous l'opération adjointe de $G$.

À l'élément nilpotent $e$ et la sous-algèbre $\ad$-nilpotente $\m_{\ll,e}$,
on associe une algèbre d'endomorphismes $H_{\ll,e}$, appelée
{\em $W$-algèbre finie associée à $e$}, dont on rappelle une construction
ci-dessous. Les $W$-algèbres finies ont été introduites par Premet~\cite{Pr1}.
Dans le cas des éléments nilpotents pairs (i.e. $\g(-1) = \g(1) = \{0\}$), elles furent
introduites par Lynch~\cite{Ly}, généralisant ainsi la construction de Kostant~\cite{Ko}
correspondant au cas des éléments nilpotents réguliers. L'étude des {\em $W$-algèbres finies} a connu un 
essor particulièrement intense, notamment en raison 
de leur importance dans la théorie des représentations comme l'illustre 
l'{\em équivalence de Skryabin},~\cite{S}.

On note $U(\g)$ et $U(\m_{\ll,e})$ les algèbres enveloppantes de $\g$ et $\m_{\ll,e}$ respectivement.
Soient $\Bbbk_e$ le $U(\m_{\ll,e})$-module à gauche correspondant au caractère
$\chi|_{\m_{\ll,e}}$ de $\m_{\ll,e}$ et $I_{\ll,e}$ l'idéal à gauche de $U(\g)$
engendré par les éléments $x - \chi(x)$ avec $x \in \m_{\ll,e}$.
Soient 
$$
    Q_{\ll,e} := U(\g) \otimes_{U(\m_{\ll,e})} \Bbbk_e \simeq U(\g)/I_{\ll,e}
$$
la {\em représentation de Gelfand-Graev généralisée}, et $H_{\ll,e}$ l'algèbre d'endomorphismes 
\begin{eqnarray}\label{eq:Hintro}
  H_{\ll,e} := Q_{\ll,e}^{\ad \m_{\ll,e}} \simeq \End_\g(Q_{\ll,e})^{\rm op}.
\end{eqnarray}

On peut généraliser les constructions de $Q_{\ll,e}$ et $H_{\ll,e}$ comme suit.
À toute sous-algèbre ${\rm ad}$-nilpotente $\m$ de $\g$ vérifiant les conditions 
suivantes,

\smallskip

($\chi$1)\; $\chi([\m,\m])=\{0\}$;

($\chi$2)\; $\m\cap\g^{e}=\{ 0\}$;

($\chi$3)\; $\dim\m=(\dim G \cdot e)/2$,  

\smallskip

\noindent
on définit une algèbre d'endomorphismes $H(\m,e)$ suivant la construction (\ref{eq:Hintro}). 
Précisément, on pose
$$
  H(\m,e):= Q(\m,e)^{\ad \m} \simeq \End_\g(Q(\m,e))^{\rm op},
$$
où $Q(\m,e)$ est le quotient de $U(\g)$ par l'idéal à gauche de $U(\g)$ 
engendré par les éléments $x-\chi(x)$ avec $x \in \m$. 
On dira que $H(\m,e)$ est la {\em $W$-algèbre finie associée à $\m$}. 
Une question naturelle, soulevée par A.~Premet, est la suivante:

\begin{quest_alpha}\label{q:introPr}
Si $\m$ est une sous-algèbre $\ad$-nilpotente de $\g$ vérifiant les conditions {\rm ($\chi$1)}, {\rm ($\chi$2)}
et {\rm ($\chi$3)}, 
les $W$-algèbres $H(\m,e)$ et $H_{\ll,e} = H(\m_{\ll,e})$ sont-elles isomorphes?
\end{quest_alpha}

Cette question est d'autant plus naturelle que la réponse est affirmative, d'après un résultat 
de Premet~\cite{Pr1}, dans le cas où le corps de base est de caractéristique positive.  
La $W$-algèbre $H(\m,e)$ dépend uniquement de l'orbite de $e$ dans $\g$, à conjugaison près.
Comme $e$ sera fixé dans la suite, nous noterons désormais plus simplement $H(\m)$ et $Q(\m)$
l'algèbre $H(\m,e)$ et le quotient $Q(\m,e)$ respectivement. 

\subsection{}\label{S:intronulle}
Pour le cas où le corps de base est de caractéristique nulle, 
les résultats, partiels, obtenus jusqu'ici sont valables sur le corps $\C$ des 
nombres complexes. Ils sont essentiellement dûs à W. Gan et V. Ginzburg d'une part (cf.~\cite{GG}) 
et à J. Brundan et S. Goodwin d'autre part (cf.~\cite{BG}). On les résume ci-dessous. 
On suppose désormais que $\Bbbk=\C$.

Tout d'abord, W. Gan et V. Ginzburg montrent que l'algèbre $H_{\ll,e}$ ne dépend pas, à un
isomorphisme près, du choix du sous-espace lagrangien $\ll$ dans 
$\g(-1)$,~\cite[Theorem 4.1]{GG}. 

On notera désormais plus simplement $H_e$ l'algèbre $H_{\ll,e}$ (définie à un isomorphisme 
près). 
Plus récemment, J. Brundan et S. Goodwin étendent le résultat de Gan et Ginzburg comme suit.  
Soit $\Ga$ une {\em $2d$-bonne $\Z$-graduation pour $e$}, 
c'est-à-dire une $\Z$-graduation $\Ga:\g=\bigoplus_{j \in \Z} \g_j$ telle que $e \in \g_{2d}$ 
et telle que l'application linéaire $\ad e: \g_j \to \g_{j+2d}$ est injective pour tout $j \le -d$ 
et surjective pour tout $j\ge -d$ (cf.~\cite[\S5]{BG})\footnote{Dans \cite{BG}, 
les auteurs considèrent des bonnes $\R$-graduations qui se définissent de manière analogue.}. 
La restriction de la forme bilinéaire $(x,y)\mapsto \la e,[x,y]\ra$ à 
$\g_{-d}\times\g_{-d}$ est non dégénérée. Si $\g_{-d}^0$ est un sous-espace lagrangien de $\g_{-d}$, 
alors 
\begin{eqnarray}\label{eq:mchibon}
  \m := \g_{-d}^0 \oplus \bigoplus_{j < -d} \g_j
\end{eqnarray}
est une sous-algèbre ${\rm ad}$-nilpotente $\Ga$-graduée de $\g$ vérifiant les conditions ($\chi$1), ($\chi$2) et ($\chi$3) 
ci-dessus. 
Le résultat principal de~\cite{BG} assure que les algèbres $H(\m)$ et $H_{e}$ sont isomorphes. 
En particulier, l'algèbre $H(\m)$ ne dépend pas, à un isomorphisme près, du choix du sous-espace 
lagrangien $\g_{-d}^0$ dans $\g_{-d}$. 

\subsection{} La définition suivante est centrale dans cet article
(cf.~Définition \ref{d:adm}):

\medskip

\begin{defin_alpha}\label{d:introadm}
Soient $\m$ et $\n$ deux sous-algèbres de $\g$. On dit que la paire 
$(\m,\n)$ est {\em {\bf admissible pour $e$}} (ou {\em {\bf $e$-admissible}}) s'il existe 
une $\Z$-graduation $\Gamma:\g=\bigoplus_{j \in \Z}\g_j$ et un entier $a>1$ tels que:

\smallskip

\begin{description}
    \item[{\rm (A1)}] $e \in \g_a;$

    \item[{\rm (A2)}] $\m$ et $\n$ sont $\Gamma$-graduées et $\bigoplus_{j \le -a}\g_{j}\subseteq \m \subseteq \n
\subseteq \bigoplus_{j < 0} \g_{j}\,;$

	\item[{\rm (A3)}] $\m^{\perp} \cap [\g,e] = [\n,e];$

	\item[{\rm (A4)}] $\n \cap \g^e = \{0\};$

	\item[{\rm (A5)}] $[\n,\m] \subseteq  \m;$

	\item[{\rm (A6)}] $\dim \m + \dim \n = \dim \g - \dim \g^e$.
\end{description}

\smallskip

\noindent
On dit qu'une $\Z$-graduation $\Gamma: \g=\bigoplus_{j \in \Z} \g_j$ est {\em {\bf admissible 
pour $e$}} s'il existe un entier $a>1$ tel que $e\in\g_{a}$ et s'il existe une paire $e$-admissible 
relativement à $\Gamma$. 
\end{defin_alpha}

La {\em $W$-algèbre associée à une paire $e$-admissible $(\m,\n)$} est définie 
par 
$$
    H(\m,\n):=Q(\m)^{\ad \n}
$$
où $Q(\m)$ est le quotient de $U(\g)$ par l'idéal à gauche de $U(\g)$ 
engendré par les éléments $x-\chi(x)$ avec $x \in \m$.  
Les $2d$-bonnes $\Z$-graduations pour $e$ sont des cas particuliers de $\Z$-graduations admissibles pour $e$ 
mais il existe des $\Z$-graduations admissibles pour $e$ qui ne sont pas $2d$-bonnes pour $e$ 
(voir Exemple \ref{ex:admsl}).

\subsection{} On s'intéresse dans cet article au problème d'isomorphisme suivant.
\begin{quest_alpha}\label{q:intro}
Si $(\m,\n)$ et $(\m',\n')$ sont deux paires $e$-admissibles de $\g$,
les $W$-algèbres $H(\m,\n)$ et $H(\m',\n')$ sont-elles isomorphes?
\end{quest_alpha}

Dans un premier temps, on obtient une caractérisation des graduations $e$-admissibles.
\begin{theo_alpha}[Théorème \ref{t:carac}] \label{t:introch1}
Une graduation $\Ga:\g = \bigoplus_{j \in \Z} \g_j$ 
est admissible pour $e$, avec $e\in\g_a$ et $a>1$, si et seulement si $\g_{\le-a}\cap\g^{e}=\{0\}$.
\end{theo_alpha}

On s'intéresse ensuite aux propriétés de l'algèbre $H(\m,\n)$.
On montre qu'il existe une variété affine $\SS$ transverse aux orbites coadjointes de $\g^*$.
De plus, on construit une {\em filtration de Kazhdan généralisée} $\F$ sur $H(\m,\n)$
et on obtient le résultat suivant qui étend \cite[Theorem 4.1]{GG}.

\begin{theo_alpha}[Théorèmes~\ref{t:gr} et \ref{t:Skr}] \label{t:introch2}
L'algèbre graduée $\gr_\F H(\m,\n)$ est isomorphe à $\C[\SS]$, 
l'algèbre des fonctions régulières définies sur $\SS$. 
De plus, si $\m=\n$, on a l'équivalence de Skryabin pour $H(\m,\n)$. 
\end{theo_alpha} 

On définit une relation d'équivalence sur l'ensemble
des paires $e$-admissibles pour réduire le problème d'isomorphisme de la Question \ref{q:intro}.
On dit que deux paires $e$-admissibles $(\m,\n)$ et $(\m',\n')$ sont {\bf {\em comparables}} 
(cf.~Définition~\ref{d:comp})
s'il existe une $\Z$-graduation $\Gamma$ telle que $\m,\n,\m',\n'$ soient $\Gamma$-graduées 
et si, 
$$
\m \subseteq \m'\subseteq \n'\subseteq\n \ \text{ ou } \  
\m' \subseteq \m\subseteq \n\subseteq\n'.
$$
On dit que deux paires $e$-admissibles $(\m,\n)$ et $(\m',\n')$ sont {\bf {\em équivalentes}} 
(cf.~Définition~\ref{d:equi}) 
s'il existe un entier $s\ge 1$ et une suite $(\m_1,\n_1),\ldots,(\m_s,\n_s)$ de paires $e$-admissibles 
tels que:

\smallskip

$\ast$ $(\m_1,\n_1)=(\m,\n)$;

\smallskip

$\ast$ $(\m_s,\n_s)=(\m',\n')$;

\smallskip

$\ast$ $(\m_i,\n_i)$ et $(\m_{i+1},\n_{i+1})$ soient comparables pour tout $i\in\{1,\ldots,s-1\}$. 

\smallskip

\noindent
On montre alors:

\begin{theo_alpha}[Théorème~\ref{t:equi}]\label{t:introch3}
Si $(\m,\n)$ et $(\m',\n')$ sont deux paires $e$-admissibles équivalentes, alors les algèbres 
$H(\m,\n)$ et $H(\m',\n')$ sont isomorphes. 
\end{theo_alpha}
En conséquence, pour traiter le problème d'isomorphisme de la Question \ref{q:intro}, il suffit
d'étudier la relation d'équivalence sur l'ensemble des paires $e$-admissibles.
On introduit ensuite les notions de graduations $e$-admissibles 
adjacentes et connexes.
Deux graduations $\Ga, \Ga'$ admissibles pour $e$ sont dites
{\rm \bf adjacentes} si elles ont une paire $e$-admissible en commun.
Deux graduations $\Gamma, \Gamma'$ admissibles pour $e$ sont dites {\rm \bf connexes} 
s'il existe une suite $(\Gamma_i)_{i \in \{1,\cdots,s\}}$
de graduations admissibles pour $e$ telle que
\begin{enumerate}
 \item[{\rm (1)}] $\Gamma = \Gamma_1$;
 \item[{\rm (2)}] les graduations $\Gamma_i$ et $\Gamma_{i+1}$ sont adjacentes pour tout $1 \le i \le s-1$;
 \item[{\rm (3)}] $\Gamma' = \Gamma_s$.
\end{enumerate}

On montre alors le résultat suivant.

\begin{theo_alpha}[Théorèmes~\ref{t:CONN} et \ref{t:DYN}] \label{t:introconnch3}
Les graduations admissibles pour $e$ sont connexes à une graduation de Dynkin. 
En particulier, elles sont connexes entre elles.
\end{theo_alpha}

Grâce à ce dernier résultat, le problème d'isomorphisme se réduit
à l'étude de la relation d'équivalence sur l'ensemble des paires $e$-admissibles pour une graduation
$e$-admissible donnée.
Dans ce contexte, on définit la notion de graduation $b$-optimale.
Pour $b >0$, on dit qu'une graduation $\Ga: \g = \bigoplus_{j \in \Z} \g_j$ est {\bf {\em $b$-optimale pour $e$}}
si $\g_{< - b} \cap \g^e = \{0\}$ et s'il existe $a \in \N$, avec $a \ge 2$,
tel que $e\in \g_a$ et $a \ge 2b$.

On montre les résulats suivants.

\begin{theo_alpha}[Théorèmes~\ref{t:equiopt} et \ref{c:equidyn}] \label{t:intro2vp}
Soit $\Ga = \bigoplus_{j \in \Z} \g_j$ une graduation $b$-optimale de $\g$.
Alors les paires $e$-admissibles relativement à $\Ga$ sont équivalentes entre elles.
De plus, elles sont équivalentes à celles issues d'une graduation de Dynkin.
\end{theo_alpha}

En particulier, les paires $e$-admissibles relatives à une $2d$-bonne graduation
et construites dans \cite{BG} sont équivalentes à celles issues d'une
graduation de Dynkin. On retrouve ainsi \cite[Theorem 1]{BG}.

On traite enfin quelques cas particuliers.
Le premier est le suivant.
\begin{theo_alpha}[Théorème~\ref{t:2vp}] \label{t:intro2vp}
Soit $\Ga = \bigoplus_{j \in \Z} \g_j$ une graduation $e$-admissible de $\g$ telle que $e \in \g_a$
et $\g_{<0} = \g_{\le -a} \oplus (\g_{b-a} + \g_{-b})$, avec, $b \in ]0, \frac{a}{2}]$.
Alors les paires $e$-admissibles relativement à $\Ga$ sont équivalentes entre elles.
\end{theo_alpha}
Soit $\ss$ la sous-algèbre engendrée par le $\lsl_2$-triplet $(e,h,f)$.  
Son centralisateur $\gs$ dans $\g$ 
est une algèbre réductive égale à $\g^{e} \cap \g^{h}$.  
Lorsque $e$ est distingué, $\rk \gs = 0$ et les graduations admissibles sont
connexes (et même adjacentes) à la graduation de Dynkin, cf. Propositions \ref{p:dist} et \ref{p:lambdagamQ}.
La réponse à la Question \ref{q:introPr} est donc positive.
Il semble ensuite naturel de considérer le cas où $\rk \gs =1$.
On montre le résultat suivant.
\begin{theo_alpha}[Théorèmes~\ref{t:equisl}, \ref{t:equiso} et \ref{t:equisp}, §\ref{S:excep}] \label{t:introrg1}
Supposons que $\g$ soit ou bien de type classique, ou bien de type exceptionnel
$\g = \G_2$, $\mathbf{F}_4$ ou $\E_6$. Si $\rk \gs =1$, alors
les paires $e$-admissibles sont équivalentes entre elles.
En particulier, les $W$-algèbres associées sont isomorphes.
\end{theo_alpha}
Compte tenu de ces résultats, on formule la conjecture suivante.
\begin{conje_alpha}
Si $\rk \gs =1$, alors
les paires $e$-admissibles sont équivalentes entre elles.
En particulier, les $W$-algèbres associées sont isomorphes.
\end{conje_alpha}
Résoudre cette conjecture serait une première étape dans la résolution de la conjecture plus générale suivante.
\begin{conje_alpha}
Les paires $e$-admissibles sont équivalentes entre elles.
En particulier, les $W$-algèbres associées sont isomorphes.
\end{conje_alpha}

\smallskip

\subsection*{Remerciements}
Une grande partie de ce travail s'inscrit dans le cadre de ma thèse.
Je remercie Rupert Yu et Anne Moreau, mes deux directeurs de thèse, pour leur patience
et pour avoir orienté ce travail et partagé leurs idées durant
des discussions enrichissantes. 
Je remercie également Simon Goodwin
pour son intérêt envers mes travaux et ses remarques fructueuses.

\section{Graduations et paires admissibles} \label{ch:padm}
Dans cette section, on conserve les notations de l'introduction.
Si $x \in \g$, on note $\g^x$ son
centralisateur dans $\g$.
\subsection{} \label{S:not}
Si $U$ est un sous-espace de $\g$, on note $U^\perp$ son orthogonal 
par rapport à la forme de Killing. Si deux sous-espaces $U$ et $V$ de $\g$
sont en couplage par rapport à la forme de Killing, alors $\dim U = \dim V$.
\begin{lemma} \label{l:eltss}
Soit $\Gamma: \g = \bigoplus \limits_{j \in \Q} \g_j$ une $\Q$-graduation de $\g$.
Il existe un élément semisimple $h_{\Gamma} \in \g$ tel que
$$
	\g_j = \{x \in \g \, ; \; [h_\Gamma,x] = jx\}.
$$
\end{lemma}
\begin{proof}
L'opérateur $\partial: \g \rightarrow  \g$ qui à $x \in \g_j$ associe $jx$ 
est une dérivation sur l'algèbre de Lie simple $\g$.
Par suite, c'est une dérivation intérieure de $\g$ donnée par $\ad h_\Gamma$
pour un élément semisimple $h_\Gamma$ de $\g$ (cf. e.g. \cite[Proposition 20.1.5]{TY}). 
\end{proof}
Soient $\Gamma: \g = \bigoplus \limits_{j \in \Q} \g_j$ 
une graduation de $\g$ et $h_\Gamma$ l'élément semisimple
définissant $\Gamma$.
Comme la forme de Killing est $\g$-invariante,
donc $\ad h_\Gamma$-invariante, 
les sous-espaces $\g_i$ et $\g_{-i}$ sont en couplage
par rapport à la forme de Killing pour tout $i$. En particulier,
ils sont de même dimension.

Pour une
$\Q$-graduation $\g = \bigoplus \limits_{j \in \Q} \g_j$ et
pour tout $k \in \Q$, on désigne par 
$\g_{\le k}$, $\g_{< k}$, $\g_{\ge k}$, $\g_{> k}$
les sommes $\bigoplus \limits_{j \le k} \g_j$,
$\bigoplus \limits_{j < k} \g_j$, $\bigoplus \limits_{j \ge k} \g_j$, $\bigoplus \limits_{j > k} \g_j$ respectivement.

\subsection{} \label{s:gradadm}
Rappelons que $\kappa:\g \rightarrow \g^*$ désigne l'isomorphisme de Killing.
On fixe pour la suite un élément nilpotent $e$ de $\g$ et on pose
$\chi:=\kappa(e)$. Rappelons la définition centrale de cet article.
\begin{definition} \label{d:adm}
Soient $\m$ et $\n$ deux sous-algèbres de $\g$. 
On dit que la paire $(\m,\n)$ est {\rm \bf admissible pour $e$} (ou {\rm \bf $e$-admissible}) s'il existe une
$\Z$-graduation $\g = \bigoplus \limits_{j \in \Z} \g_j$ et 
un entier $a > 1$ tels que:
\begin{description}
    \item[{\rm (A1)}] $e \in \g_a;$

    \item[{\rm (A2)}] $\m$ et $\n$ sont graduées et vérifient $\g_{\le -a} \subseteq \m \subseteq  \n \subseteq  \g_{<0};$

	\item[{\rm (A3)}] $\m^{\perp} \cap [\g,e] = [\n,e];$

	\item[{\rm (A4)}] $\n \cap \g^e = \{0\};$

	\item[{\rm (A5)}] $[\n,\m] \subseteq  \m;$

	\item[{\rm (A6)}] $\dim \m + \dim \n = \dim \g - \dim \g^e$.
\end{description}
\smallskip
On dit 
qu'une $\Z$-graduation $\g = \bigoplus \limits_{j \in \Z} \g_j$ est {\rm \bf admissible pour $e$}
s'il existe un entier $a > 1$ tel que $e \in \g_a$ et s'il existe une paire 
admissible pour $e$ par rapport à cette graduation.

Dans le cas particulier où $\m$ est une sous-algèbre de $\g$ 
telle que la paire $(\m,\m)$ soit admissible pour $e$, on dit que
la sous-algèbre $\m$ est {\rm \bf admissible pour $e$}.
\end{definition}
%
%
%
%
\begin{remark} \label{r:morth}
Soit $(\m,\n)$ est une paire admissible pour $e$.
D'après la Définition \ref{d:adm}, on a les propriétés suivantes:
\begin{enumerate}
 \item[{\rm (1)}] les sous-algèbres $\m$ et $\n$ sont ${\rm ad}$-nilpotentes;
\item[{\rm (2)}] $\m^\perp \subseteq  \g_{\le a-1};$
\item[{\rm (3)}] $\chi([\n,\m]) = \{0\};$
\item[{\rm (4)}] $\dim \m + \dim \n$ et $\dim \m - \dim \n$ sont pairs.
\end{enumerate}
\end{remark}
\begin{definition} \label{d:max}
On dit qu'une paire $(\m,\n)$ admissible pour $e$ relativement à une 
$\Z$-graduation $\Gamma: \g = \bigoplus_{j\in \Z} \g_j$ est {\rm \bf optimale}
si $\m = \g_{\le -a}$. 
\end{definition}
On désigne désormais par $\PA(e)$ l'ensemble des paires admissibles pour $e$
et par $\GA(e)$ l'ensemble des graduations admissibles pour $e$.
Pour $\Gamma \in \GA(e)$, on note $\PA(e,\Gamma)$ l'ensemble
des paires admissibles pour $e$ relativement à $\Gamma$.

\begin{example} \label{ex:dyn}
Soit $(e,h,f)$ un $\lsl_2$-triplet de $\g$.
On désigne par $\GaD$ la graduation de Dynkin associée à $h$,
$$
	\GaD:\g  = \bigoplus_{j \in \Z} \g_j,
$$
où $\g_j = \{x \in \g \, ; \; \ad h (x) = jx\}$. En particulier, $e \in \g_2$.

Posons
$$
    \m = \bigoplus_{j \le -2} \g_j \ \text{ et } \  \n = \bigoplus_{j < 0} \g_j.
$$
On vérifie sans peine que
$(\m,\n) \in \PA(e,\GaD)$. La graduation $\GaD$ est donc $e$-admissible.
\end{example}

%
\begin{example} \label{ex:hom}
Si $\Gamma: \g = \bigoplus \limits_{j \in \Z} \g_j$
est une graduation admissible pour $e$ définie par l'élément semisimple $h_\Gamma$,
alors la graduation $\Gamma': \g = \bigoplus \limits_{j \in \Z} \g'_j$
définie par l'élément
$$
    h_{\Gamma'} := k h_\Gamma \  \text{ où } \  k \in \N^* 
$$
est admissible pour $e$. On notera $k \Ga$ la graduation $\Ga'$.
\end{example}
%
\begin{example} \label{ex:bonnondyn}
Une $2d$-bonne graduation pour $e$ de $\g$ est $e$-admissible.
En effet, soit $\Ga:\g=\bigoplus_{j \in \Z} \g_j$ une $2d$-bonne graduation pour $e$.
On vérifie sans peine que pour
$ \m = \g_{-d}^0 \oplus \g_{< -d}$ de $\g$ (cf. §\ref{S:intronulle} de l'Introduction),
 $(\m,\m) \in \PA(e,\Ga)$ donc $\Ga \in \GA(e)$.

Prenons l'exemple où $\g=\lsl_3(\C)$ et 
$e = E_{1,3}$.
On considère la graduation $\Gamma: \g = \bigoplus_{j \in \Z} \g_j$
définie par l'élément semisimple
$h_\Gamma:=\frac{1}{3}\diag(2,2,-4)$. 
Les matrices élémentaires $E_{i,j}$ sont homogènes et leurs degrés
sont donnés par la matrice suivante:
$$\left(
\begin{array}{ccc}
0 & 0 & 2 \\
0 & 0 & 2 \\
-2 & -2 & 0 
\end{array} \right).
$$
La sous-algèbre $\g_{-2}$ est une sous-algèbre admissible
pour $e$ relativement à $\Gamma$.
De plus, la graduation $\Ga$ est une graduation non Dynkin mais bonne pour $e$.

\end{example}
Il est important de remarquer qu'une graduation admissible pour $e$ n'est pas
toujours bonne, comme l'illustre l'exemple suivant.
\begin{example} \label{ex:admsl}
On suppose que $\g=\lsl_4(\C)$ et $e := E_{1,3} + E_{2,4}$.
On considère la graduation $\Gamma: \g = \bigoplus \limits_{j \in \Z} \g_j$ 
définie par l'élément semisimple
$$\frac{1}{2}\diag(3,1,-1,-3).$$ En particulier, $e \in \g_2$.
Les matrices élémentaires $E_{i,j}$ sont homogènes pour cette graduation
et leurs degrés sont donnés sur la matrice suivante
$$
\left(
\begin{array}{cccc}
0 & 1 & 2 & 3 \\
-1 & 0 & 1 & 2 \\
-2 & -1 & 0 & 1 \\
-3 & -2 & -1 & 0
\end{array}
\right).
$$ 
Posons
$$
	\m= \g_{\le -2} \  \text{ et } \  \n = \m \oplus \C E_{2,1} \oplus \C E_{3,2}.
$$
On montre par calcul direct que $(\m,\n) \in \PA(e,\Ga)$ avec $a=2$.
Comme $\ad(e):\g_{-1} \rightarrow \g_1$ n'est pas injective, ceci
fournit un exemple d'une graduation admissible pour $e$ 
qui n'est pas bonne pour $e$.
\end{example}
%
\subsection{} \label{s:prtadm}
On fixe une $\Z$-graduation $\Gamma$ de $\g$,
\begin{equation} \label{eq:grad}
	\g = \bigoplus_{j \in \Z} \g_j,
\end{equation}
et soit $h_\Gamma$ l'élément semisimple de $\g$ définissant $\Gamma$.
On suppose que $e \in \g_a$, avec $a >1$.
%
%
%
\begin{proposition} \label{p:dist}
Si $e$ est un élément nilpotent distingué de $\g$,
alors les graduations de Dynkin sont les seules graduations de $\g$
admissibles pour $e$, à homothéties près (cf. Exemple \ref{ex:hom}).
\end{proposition}
\begin{proof}
Supposons que $e$ soit un élément nilpotent distingué de $\g$
et que la graduation (\ref{eq:grad}) soit admissible pour $e$.
Soit $(e,h,f)$ un $\lsl_2$-triplet tel que $h \in \g_0$ et $f \in \g_{-a}$, cf. \cite[Proposition 32.1.7]{TY}.
L'élément $t:= \frac{a}{2} h - h_\Gamma$ centralise $e$ et il est semisimple.
Comme $e$ est distingué, on a $t=0$ et donc $h_\Gamma = \frac{a}{2} h $,
d'où la proposition.
\end{proof}
%
%
%
%
%
%
%
De façon générale, pour $k \in \Z$, on désigne par $\g_k^e$ l'intersection de $\g_k$
avec $\g^e$.
\begin{lemma} \label{l:couplage}
Soit $k \in \Z$ tel que l'application $\ad e: \g_k \rightarrow \g_{k+a}$ soit injective.
Alors l'application $\ad e: \g_{-(k+a)} \rightarrow \g_{-k}$ est surjective.
En particulier, on a $\dim \g_k = \dim \g_{-(k+a)} - \dim \g_{-(k+a)}^{e}$.
\end{lemma}
\begin{proof}
On a $[e,\g_k] \subset \g_{k+a}$ et $\dim [e,\g_k] = \dim \g_k$.
Comme $\g_{k+a}$ et $\g_{-(k+a)}$ sont en couplage par rapport à la forme de Killing,
il existe $V \subset \g_{-(k+a)}$ tel qu'on ait un couplage
$[e,\g_k] \times V$. En particulier, on a $\dim V = \dim [e,\g_k] = \dim \g_k$.
Comme la forme de Killing est $\ad$-invariante,
on a un couplage entre $\g_k$ et $[e,V]$.
En particulier, $\dim [e,V] = \dim \g_k = \dim V$.
On en déduit que l'application $\ad e: \g_{-(k+a)} \rightarrow \g_{-k}$ est surjective.
Le lemme s'ensuit.
\end{proof}
\begin{corollary} \label{c:adsurj}
Si $\Gamma \in \GA(e)$,
alors l'application $\ad e: \g_{\ge 0} \rightarrow  \g_{\ge a}$
est surjective.
\end{corollary}
%
%
%
%
\begin{proposition} \label{p:supp}
Supposons que $\Gamma \in \GA(e)$.
Soit $\n$ un supplémentaire gradué de $\g_{< 0} \cap \g^e$ dans $\g_{<0}$
contenant $\g_{\le -a}$. Alors la paire $(\g_{\le -a},\n)$ est admissible pour $e$
si et seulement si $\n$ est une sous-algèbre de $\g$.
En particulier, si
$\g_{<0} \cap \g^e =\{0\}$, alors la paire
$(\g_{\le -a}, \g_{<0})$ est l'unique paire optimale admissible pour $e$.
\end{proposition}
\begin{proof}
Tout d'abord, si $(\g_{\le -a},\n) \in \PA(e,\Ga)$,
alors $\n$ est une sous-algèbre de $\g$.
Réciproquement, supposons que $\n$ soit une sous-algèbre graduée de $\g$.
La condition (A1) est clairement vérifiée.
Par construction, les conditions (A2) et (A4) sont satisfaites.
En outre, on a 
$$
	[\g_{\le -a} , \n] \subseteq  \g_{\le -a-1} \subseteq \g_{\le -a},
$$
d'où (A5). D'après la Remarque \ref{r:morth}(3),
$[\n,e]$ est inclus dans $(\g_{\le -a})^\perp$. Par suite,
$[\n,e] \subset (\g_{\le -a})^\perp \cap [\g,e]$.
Soit $X = [Y,e]$ un élément de $(\g_{\le -a})^\perp \cap [\g,e]$
avec $Y = \sum \limits_{i} Y_i$ où $Y_i \in \g_i$ pour tout $i$.
Comme $(\g_{\le -a})^\perp = \g_{< a}$, on peut supposer $Y \in \g_{< 0}$.
Écrivons alors $Y =Y' + Y''$ avec $Y' \in \n$ et $Y'' \in \g_{< 0} \cap \g^e$.
Ainsi, 
$$
	X = [Y, e]= [Y',e] \in [\n ,e],
$$
ce qui implique (A3).
D'autre part, on a
\begin{eqnarray*}
 \dim \g - \dim \g^e  & = & \dim [\g,e] = \dim [\g_{ \ge 0},e] + \dim [\g_{< 0},e] \\
		      & = & \dim \g_{\ge a} + \dim \n = \dim \g_{\le -a} + \dim \n
\end{eqnarray*}
où la dernière égalité provient du Corollaire \ref{c:adsurj},
de la construction de $\n$ et de (A4).
En conclusion, $(\g_{\le -a},\n) \in \PA(e,\Ga)$.
\end{proof}
Il n'existe pas toujours de paires optimales admissibles pour $e$
comme l'illustre l'exemple suivant:
\begin{example} \label{ex:nonexisopt}
On suppose que $\g=\lsl_{11}(\C)$
et $e := \sum \limits_{\underset{i \notin \{6,9\}}{1 \le i \le 10}} E_{i,i+1}$.
On considère la graduation $\Ga: \g = \bigoplus_{j \in \Z} \g_j$
définie par l'élément semisimple $$\frac{1}{11}\diag(73,40,7,-26,-59,-92,29,-4,-37,51,18).$$
Les matrices élémentaires $E_{i,j}$ sont homogènes pour cette graduation et $e\in \g_3$.
Montrons que, relativement à cette graduation, une paire admissible
pour $e$ optimale n'existe pas.
Supposons que $(\g_{\le -3},\n)$ soit une paire optimale.
Comme $\n$ est graduée, on peut écrire
$$
	\n = \g_{\le -3} \oplus \n_{-2} \oplus \n_{-1} \ \text{ où } \  \n_{-2} \subset \g_{-2} \text{ et } 
						\n_{-1} \subset \g_{-1}.
$$
On a $\dim \g^e = 25$ et, d'après (A6),
$\dim (\n_{-2} \oplus \n_{-1}) = 12 = \dim \n_{-2} + \dim \n_{-1}$
avec $\dim \g_{-2} \oplus \g_{-1} = 13$.
Sachant que $\g_{<0} \cap \g^e = \C(E_{7,10} + E_{8,11}) \subset \g_{-2}$ et $\dim \g_{<0} \cap \g^e =1$,
et comme $\g^e \cap \n_{-2} =\{0\} $, on a $\n_{-2} \not= \g_{-2}$.
Par conséquent, on a $\dim \n_{-1} = \dim \g_{-1}$ et $\n_{-1} = \g_{-1}$.
Puisque $\n$ est une sous-algèbre graduée de $\g$ on a $[\n_{-1},\n_{-1}] \subset \n_{-2}$.
Or $\{0\} \neq \g_{<0} \cap \g^e \subset [\n_{-1},\n_{-1}]$.
En particulier,
$\n \cap \g^e \neq \{0\}$ ce qui contredit le fait
que $(\g_{\le -3},\n)$ soit une paire admissible pour $e$. 
Par suite, une paire optimale relativement à cette graduation n'existe pas.
\end{example}
\subsection{} \label{s:carac}
Dans ce paragraphe, $\Gamma: \g = \bigoplus \limits_{j \in \Z} \g_j$
désigne une $\Z$-graduation de $\g$
vérifiant $e \in \g_a$ où $a >1$. L'objectif de ce paragraphe est de montrer le théorème suivant:
\begin{theorem} \label{t:carac}
La graduation $\Gamma$ est admissible pour $e$
si et seulement si $\g_{\le -a} \cap \g^e = \{0\}$.
\end{theorem}
L'implication directe est claire d'après (A2) et (A4).
On explique ici la stratégie pour montrer l'autre implication.
Soient $h_\Gamma$ l'élément semisimple de $\g$
définissant la graduation $\Gamma$ et $(e,h,f)$ un $\lsl_2$-triplet
de $\g$ tel que $h\in \g_0$ et $f \in \g_{-a}$.
On pose 
$$
	\ss := \Vect(e,h,f)
$$
et 
$$
	t:= h_\Gamma -\frac{a}{2} h.
$$
L'élément $t$ de $\g$ est semisimple
et appartient à $\g^e \cap \g^h = \g^\ss$.
De plus, les valeurs propres de $\ad t$ sont rationnelles.
On suppose désormais que $\g_{\le -a} \cap \g^e = \{0\}$.
On va construire une paire $(\m,\n)$ admissible pour $e$ relativement à la 
graduation $\Gamma$ suivant le Lemme \ref{l:stratgen}.
D'après ce lemme, la paire $(\m,\n)$ vérifie les conditions 
(A2), (A3), (A4) et (A6). Pour conclure qu'elle est $e$-admissible,
il restera à montrer que $\m$ et $\n$ sont des sous-algèbres de $\g$
vérifiant (A5).
\begin{lemma} \label{l:stratgen}
Soit $\g = \P_1 \oplus \cdots \oplus \P_s$ une décomposition de $\g$
en sous-espaces de $\g$ stables par $\ss$ et $t$, et deux à deux orthogonaux
par rapport à la forme de Killing. 
Pour tout $i \in \{1,\cdots,s\}$, soient $\m_i$ et $\n_i$ deux sous-espaces
gradués dans $\P_i$ vérifiant les conditions suivantes:
\begin{description}
 \item[{\rm (C1)}] $\P_i \cap \g_{\le -a} \subset \m_i \subset \n_i \subset \P_i \cap \g_{<0}$;
 \item[{\rm (C2)}] $\m_i^\perp \cap [e,\P_i] = [e,\n_i]$;
 \item[{\rm (C3)}] $\n_i \cap \g^e = \{0\}$;
 \item[{\rm (C4)}] $\dim \m_i + \dim \n_i = \dim \P_i - \dim (\P_i \cap \g^e).$
\end{description}
On pose
$$
  \m= \bigoplus_{i=1}^s \m_i \qquad \text{et} \qquad \n= \bigoplus_{i=1}^s \n_i
$$
Alors la paire $(\m,\n)$ vérifie les conditions
(A2), (A3), (A4) et (A6) de la Définition \ref{d:adm}.
\end{lemma}
\begin{proof}
La condition (A2) est vérifiée car  
$$
      \g_{\le -a}= \bigoplus \limits_{i=1}^s (\P_i \cap \g_{\le -a}) 
\subset \bigoplus \limits_{i=1}^s \m_i \subset \bigoplus \limits_{i=1}^s \n_i 
\subset \bigoplus \limits_{i=1}^s (\P_i \cap \g_{<0})= \g_{<0}.
$$
En outre, comme
$$
	\m^\perp \cap [e,\g] = \bigoplus \limits_{i=1}^s (\m_{i}^\perp \cap [e, \P_i])
							= \bigoplus_{i=1}^s  [e,\n_{i}] = [e,\n],
$$
alors (A3) est vérifiée.
De plus, puisque $\n \cap \g^e = \bigoplus \limits_{i=1}^s  (\n_{i} \cap \g^e)$, alors 
			$\n \cap \g^e =\{0\}$, d'où la condition (A4).
Enfin un simple calcul de dimension montre que
$$ 
      \dim \m + \dim \n = \sum_{i=1}^s (\dim \m_i+\dim \n_i) = \sum_{i=1}^s (\dim \P_i - \dim (\P_i \cap \g^e))
		=\dim \g - \dim \g^e.
$$
La condition (A6) s'ensuit.
\end{proof}
\begin{proof}[Démonstration du Théorème \ref{t:carac}]
On considère la décomposition 
de $\g$ en composantes isotypiques
de $\ss$-modules simples,
$$
	\g = \E_1 \oplus \cdots \oplus \E_r.
$$
D'après le lemme de Schur, cette décomposition est orthogonale relativement
à la forme de Killing. 
Comme $t$ commute avec $\ss$, toujours d'après le lemme de Schur,
chaque composante isotypique $\E_i$ est stable sous l'action
adjointe de $t$. Par suite, pour tout $i$, la composante isotypique $\E_i$ se décompose 
en espaces propres pour $\ad t$,
\begin{equation*}
	\E_i = \bigoplus_{\lambda \in \Q} \E_{i,\lambda},
\end{equation*}
tels que
\begin{equation} \label{eq:coupt}
	\la.,.\ra _{|\E_{i,\lambda}\times \E_{i,\mu}} = 0 \,  \text{ si } \,  \lambda+\mu \not= 0
\ \text{ et }\
	\la.,.\ra _{|\E_{i,\lambda}\times \E_{i,-\lambda}} 
	 \, \text{  non dégénérée.}
\end{equation}

La dernière assertion signifie que $\E_{i,\lambda}$ et $\E_{i,-\lambda}$
sont en couplage. 
Soient $i \in \{1, \cdots, r\}$, $\lambda \in \Q$
et $d_i$ la dimension d'un $\ss$-module simple de $\E_i$. 
L'ensemble des valeurs propres de $\ad h$ sur $\E_i$ est donné par:
$$
	\{-(d_i -1), -(d_i-3), \cdots, d_i-3, d_i-1\}.
$$
Il s'ensuit que le plus petit poids de $\ad h_\Gamma$ sur $\E_{i,\lambda}$ (resp. sur $\E_{i,-\lambda}$) est égal à
$$
	\rho_{i,\lambda} := - \frac{a}{2}(d_i-1) + \lambda \quad (\text{resp. }
							\rho_{i,-\lambda} := - \frac{a}{2}(d_i-1) - \lambda). 
$$
On en déduit que l'ensemble des valeurs propres de $\ad h_\Gamma$ sur $\E_{i,\lambda}$ est donné par
$$
	\Xi_{i,\lambda} := \{\rho_{i, \lambda} + l a\, ; \; 0 \le l \le d_i-1\} \subset \Q.
$$
En particulier, on a $\rho_{i, -\lambda} = -\rho_{i,\lambda} - (d_i-1)a$
et $\Xi_{i,-\lambda}=-\Xi_{i,\lambda} := \{-\mu\, ; \; \mu \in \Xi_{i,\lambda}\}$.
On a la décomposition 
$$
    \E_{i,\lambda} = \bigoplus_{l=0}^{d_i-1} \E_{i,\lambda}^l,
$$
où $\E_{i,\lambda}^l$ est le sous-espace propre de $\E_{i,\lambda}$ pour $\ad h_\Ga$
associé à la valeur propre $\rho_{i, \lambda} + l a$.
De façon analogue,
$$
    \E_{i,-\lambda} = \bigoplus_{l=0}^{d_i-1} \E_{i,-\lambda}^l,
$$
où $\E_{i,-\lambda}^l$ est le sous-espace propre de $\E_{i,-\lambda}$ pour $\ad h_\Ga$
associé à la valeur propre $\rho_{i, -\lambda} + l a$.
De plus, pour $l,l' \in \{0, 1, \cdots, d_i-1\}$, le couplage (\ref{eq:coupt}) est décrit comme suit:
\begin{equation} \label{eq:couphgamma}
	\la \E_{i,\lambda}^l, \E_{i,-\lambda}^{l'} \ra
					= \left\{ \begin{array}{ll}
								\{0\} &\text{ si } l + l' \neq d_i -1 ;\\
								\C &\text{ si } l+l'=d_i-1.
							   \end{array}\right.
\end{equation}
En particulier, les sous-espaces propres $\E_{i,\lambda}^l$ et $ \E_{i,-\lambda}^{d_i -1 -l}$
sont en couplage par rapport à la forme de Killing.
Remarquons que $\g_{\le -a} \cap \g^e =\{0\}$ si et seulement si
$$
  	\frac{a}{2}(d_i-1) + \lambda > -a \qquad \text{et} \qquad \frac{a}{2}(d_i-1) - \lambda > -a,
$$
i.e., si et seulement si
\begin{equation} \label{eq:condlambda}
	-\frac{a}{2}(d_i+1) < \lambda <  \frac{a}{2}(d_i+1).
\end{equation}
D'après notre hypothèse, les inégalités de (\ref{eq:condlambda}) sont donc satisfaites.
On pose 
$$
    m_{i,\lambda} := \frac{\dim \E_{i,\lambda}}{d_i}.
$$ 
On a alors
$$
    \dim \E_{i,\lambda}^l = m_{i,\lambda}, \qquad  l \in \{0, 1, \cdots, d_i-1\}.
$$
Pour $\lambda \ge 0$, on pose
$$
	\V_{i,\lambda} = \E_{i,\lambda} + \E_{i,-\lambda}.
$$
On a alors la décomposition orthogonale par rapport à la forme de Killing en sous-espaces stables
par $\ss$ et $t$:
$$
	\g = \bigoplus_{i =1}^r \bigoplus_{\lambda \in \Q^+} \V_{i, \lambda}.
$$
On cherche à appliquer le Lemme \ref{l:stratgen} à cette décomposition
pour construire une paire $e$-admissible.
On remarque que 
\begin{equation} \label{eq:dimvi0}
	\dim \V_{i,0} =  m_{i,0} d_i, \quad  \dim \V_{i,0} \cap \g^e = m_{i,0}
\end{equation}
et
\begin{equation} \label{eq:dimvil}
	\dim \V_{i,\lambda} = 2 m_{i,\lambda} d_i, \quad 
		\dim \V_{i,\lambda} \cap \g^e = 2 m_{i,\lambda} \  \text{ si } \ 
		    \lambda \neq 0.
\end{equation}
On distingue deux cas: 

\smallskip

\noindent
{\bf I.} Il existe $k \in \{0, 1,  \cdots, d_i-1\}$ tel que $\rho_{i,\lambda} + k a = 0$.

\smallskip

\noindent
{\bf II.} Il existe $k \in \{-1, 0, 1,  \cdots, d_i-1\}$ tel que
$\rho_{i,\lambda} + k a < 0 < \rho_{i,\lambda} + (k+1) a $.

\smallskip

\noindent
Ce sont les seules possibilités d'après la formule (\ref{eq:condlambda})
car $\rho_{i,\lambda}-a < 0 < \rho_{i,\lambda} + d_ia$.

\paragraph{Cas I.} Dans ce cas, $\E_{i,\lambda} \cap \g_{<0} = \bigoplus \limits_{l =0}^{k-1}\E_{i,\lambda}^l$
et $\E_{i,-\lambda} \cap \g_{<0} = \bigoplus \limits_{l =0}^{d_i-2-k}\E_{i,-\lambda}^l$. 
En particulier, $(\V_{i,\lambda} \cap \g_{<0} ) \cap \g^e = \{0\}$. On pose
$$
    \m_{i,\lambda} = \n_{i,\lambda} := \V_{i,\lambda} \cap \g_{<0} = \V_{i,\lambda} \cap \g_{\le -a}.
$$
Par construction, on a
$$
    \n_{i,\lambda} \cap \g^e = \{0\}.
$$
De plus, 
$$
    \m_{i,\lambda}^\perp \cap [\V_{i,\lambda},e] = 
	\bigoplus \limits_{l =1}^{k}\E_{i,\lambda}^l + \bigoplus \limits_{l =1}^{d_i-1-k}\E_{i,-\lambda}^l
	      = [\n_{i,\lambda},e ].
$$
En outre, si $\lambda = 0$, alors
$$
  \dim  \m_{i,0} + \dim \n_{i,0} = (d_i-1) m_{i,0} = \dim \V_{i,0}- \dim (\V_{i,0} \cap \g^e)
$$
d'après (\ref{eq:dimvi0}).
Si $\lambda \neq 0$, alors
$$
  \dim  \m_{i,\lambda} + \dim \n_{i,\lambda} = 2  [(d_i-1) m_{i,\lambda}] = 
\dim \V_{i,\lambda}- \dim (\V_{i,\lambda} \cap \g^e)
$$
d'après (\ref{eq:dimvil}).
Par conséquent, les conditions (C1), (C2), (C3) et (C4) sont vérifiées.
%
\paragraph{Cas II.} Dans ce cas, $\E_{i,\lambda} \cap \g_{<0} = \bigoplus \limits_{l =0}^{k}\E_{i,\lambda}^l$
et $\E_{i,-\lambda} \cap \g_{<0} = \bigoplus \limits_{l =0}^{d_i-2-k}\E_{i,-\lambda}^l$. 
On présente dans la Table \ref{carac:CasII} les choix de $\m_{i,\lambda}$ et $\n_{i,\lambda}$
dans chacun des sous-cas suivants:
\begin{enumerate}
 \item[(a)] $k=-1$;
 \item[(b)] $k=d_i-1$;
 \item[(c)] $-1 < k < d_i-1$ et $\lambda = 0 $;
 \item[(d)] $-1 < k < d_i-1$, $\lambda \neq 0$ et $\rho_{i,\lambda} + k a < -\rho_{i,\lambda} - (k+1) a;$
\item[(e)] $-1 < k < d_i-1$, $\lambda \neq 0$ et $-\rho_{i,\lambda} - (k+1) a< \rho_{i,\lambda} + k a $.
\end{enumerate}
\begin{table}[h]
\centering
\renewcommand{\arraystretch}{2}
\setlength{\tabcolsep}{0.5cm}
\begin{tabular}{|c|c|c|}
 \hline
  {\bf Cas II}& $\m_{i,\lambda}$ & $\n_{i,\lambda}$ \\
\hline
 (a) &  \multicolumn{2}{c|}{$\bigoplus \limits_{l =0}^{d_i-2}\E_{i,-\lambda}^l$} \\[0.2cm]
 \hline
 (b) &  \multicolumn{2}{c|}{$\bigoplus \limits_{l =0}^{d_i-2}\E_{i,\lambda}^l$} \\[0.2cm]
 \hline
 (c) & $\bigoplus \limits_{l =0}^{k-1}\E_{i,0}^l$ 
& $V_{i,[0]} \cap \g_{<0}$ \\[0.2cm]
 \hline
 (d)  & 
     \multicolumn{2}{c|}{$\bigoplus \limits_{l =0}^{k}\E_{i,\lambda}^l +\bigoplus \limits_{l =0}^{d_i-3-k}\E_{i,-\lambda}^l$}
	\\[0.2cm] 
\hline
(e) &
	      \multicolumn{2}{c|}{$\bigoplus \limits_{l =0}^{k-1}\E_{i,\lambda}^l +
	      \bigoplus \limits_{l =0}^{d_i-2-k}\E_{i,-\lambda}^l$} \\[0.2cm]
 \hline
\end{tabular}

\smallskip

\caption{{\bf Cas II.}} \label{carac:CasII}
\end{table}
Dans chacun de ces sous-cas, par construction (cf. Table \ref{carac:CasII}), 
$$
    \V_{i,\lambda} \cap \g_{\le -a} \subset \m_{i,\lambda} \subset \n_{i,\lambda} 
			\subset \V_{i,\lambda} \cap \g_{<0}
\quad \text{ et } \quad
    \n_{i,\lambda} \cap \g^e = \{0\}.
$$

\smallskip

\noindent {\bf Sous-cas (a).} On a
$$
    \m_{i,\lambda}^\perp \cap [\V_{i,\lambda},e] = 
	 \bigoplus \limits_{l =1}^{d_i-1}\E_{i,-\lambda}^l
	      = [\n_{i,\lambda},e ],
$$
et, d'après (\ref{eq:dimvil}),
$$
  \dim  \m_{i,\lambda} + \dim \n_{i,\lambda} = 2  [(d_i-1) m_{i,\lambda}] = 
\dim \V_{i,\lambda}- \dim (\V_{i,\lambda} \cap \g^e).
$$ 

\smallskip

\noindent {\bf Sous-cas (b).} On a
$$
    \m_{i,\lambda}^\perp \cap [\V_{i,\lambda},e] = 
	 \bigoplus \limits_{l =1}^{d_i-1}\E_{i,\lambda}^l
	      = [\n_{i,\lambda},e ],
$$
et, d'après (\ref{eq:dimvil}),
$$
  \dim  \m_{i,\lambda} + \dim \n_{i,\lambda} = 2  [(d_i-1) m_{i,\lambda}] = 
\dim \V_{i,\lambda}- \dim (\V_{i,\lambda} \cap \g^e).
$$ 

\smallskip

\noindent {\bf Sous-cas (c).} On a 
$$
    \m_{i,0}^\perp \cap [e,\V_{i,0}] = 
	 \bigoplus \limits_{l =1}^{k+1}\E_{i,0}^l
	      = [e,\n_{i,0} ],
$$
et, d'après (\ref{eq:dimvi0}),
$$
  \dim  \m_{i,[0]} + \dim \n_{i,[0]} = (d_i-1) m_{i,0} = \dim V_{i,[0]}- \dim (V_{i,[0]} \cap \g^e).
$$ 

\smallskip

\noindent {\bf Sous-cas (d).}
Si $\rho_{i,\lambda} + ka < -\rho_{i,\lambda} - (k+1) a$ on a alors
$$
    \m_{i,\lambda}^\perp \cap [\V_{i,\lambda},e] = 
	 \bigoplus \limits_{l =1}^{k+1}\E_{i,\lambda}^l \oplus \bigoplus \limits_{l =1}^{d_i-2-k}\E_{i,-\lambda}^l
	      = [\n_{i,\lambda},e ].
$$

\smallskip

\noindent {\bf Sous-cas (e).} On procède de manière analogue au sous-cas (d).
On a en particulier d'après (\ref{eq:dimvil}),
$$
  \dim  \m_{i,\lambda} + \dim \n_{i,\lambda} = 2  [(d_i-1) m_{i,\lambda}] = 
\dim \V_{i,\lambda}- \dim (\V_{i,\lambda} \cap \g^e).
$$ 

Par conséquent, les conditions (C1), (C2), (C3) et (C4) sont vérifiées dans tous les sous-cas.
De plus, $\m_{i,\lambda} \subset \n_{i,\lambda} \subset \V_{i,\lambda} \cap \g_{\le -\frac{a}{2}} $.

\paragraph{Conclusion.} Posons
$$
	\m= \bigoplus_{i,\lambda} \m_{i,\lambda} \ \text{ et } \ 
				\n= \bigoplus_{i,\lambda} \n_{i,\lambda}.
$$
Les conditions (A2), (A3), (A4) et (A6) sont vérifiées d'après 
le Lemme \ref{l:stratgen}. Il reste à vérifier que 
$\m$ et $\n$ sont des sous-algèbres de $\g$ qui vérifient la condition (A5).
On a $\g_{\le -a} \subset \m \subset \n \subset \g_{\le -\frac{a}{2}} \subset \g_{<0}$. 
Ainsi,
$$
  [\m,\m ] \subset [\n,\m ] \subset [\n,\n ] \subset [\g_{\le -\frac{a}{2}},\g_{\le -\frac{a}{2}} ]
	\subset \g_{\le -a} \subset \m \subset \n.
$$
Il s'ensuit que $\m$ et $\n$ sont des sous-algèbres de $\g$ telles que $[\n,\m] \subset \m$.
Par conséquent, la paire $(\m,\n)$ est admissible pour $e$.
Ceci achève la démonstration du théorème. 
\end{proof}
\begin{remark}
Comme la graduation $\Gamma$ est entière, il en résulte 
que $\lambda \in \frac{1}{2} \Z$. Pour $a$ fixé le nombre de graduations admissibles pour $e$
est fini d'après les inégalités (\ref{eq:condlambda}).
\end{remark}
\section{$W$-algèbres finies associées aux paires admissibles} \label{ch:walg}
Dans cette section, on conserve les notations
de la section précédente. Plus précisément, $e$ est un élément
nilpotent de $\g$ et $\chi:= \kappa(e)$ où $\kappa$ est l'isomorphisme 
de Killing. On fixe une paire
$e$-admissible $(\m,\n)$ de $\g$ relativement à une $\Z$-graduation
$\Gamma: \g = \bigoplus \limits_{j \in \Z} \g_j$ avec $e \in \g_a$ pour $a >1$.
Le groupe adjoint $G$ opère dans $\g$ et $\g^*$
via l'opération adjointe et l'opération coadjointe respectivement.
Pour $g \in G$, $x \in \g$ et $\xi \in \g^*$,
on note $g(x)$ et $g(\xi)$ les images de $x$ et $\xi$
par $g$ pour ces opérations respectives.
\subsection{} \label{S:const}
Si $\Liea$ une sous-algèbre de $\g$, on note $U(\Liea)$ l'algèbre enveloppante de $\Liea$.
D'après la Remarque \ref{r:morth}(3), la restriction à $\m$
de $\chi$ est un caractère de $\m$.
Ce dernier s'étend en une représentation $\chi : U(\m) \to \C$ de $U(\m)$ 
et on désigne par $\C_{\chi}$ le $U(\m)$-module à gauche correspondant. 
La multiplication à droite par un élément de $\m$ 
induit une structure de $U(\m)$-module à droite sur $U(\g)$.
Soit $I(\m)$ l'idéal à gauche de $U(\g)$ engendré par les 
éléments $x - \chi(x)$, pour $x \in \m$.
On pose 
$$
	Q(\m)  :=  U(\g) \otimes_{U(\m)} \C_{\chi} .  
$$ 
%
%
Alors $Q(\m)$ est isomorphe à $U(\g)/ {I(\m)}$ en tant que $U(\g)$-modules.
L'opération adjointe de $\n$ dans $\g$ s'étend de façon unique en une opération 
$\theta$ de $\n$ dans $U(\g)$. 
L'idéal $I(\m)$ est stable sous cette action de $\n$.
En particulier, ceci induit une structure de $\n$-module sur $Q(\m)$ donnée par
$$
    \theta(x)(u+I(\m)) = \theta(x)(u)+I(\m),
$$
pour $x \in \n$ et $u \in U(\g)$.

On pose
\begin{equation}\label{eq:W}
 \begin{array}{rcl}   
 H(\m,\n) & = & \{ u + I(\m)  \in  Q(\m) \, ; \ \theta(x)(u)  \in I(\m) 
					\hbox{ pour tout } \, x \in \n \} \\
	  & = & \{ u + I(\m)  \in  Q(\m) \, ; \ I(\m) u  \subset I(\m) \}.
\end{array} 
\end{equation}
Autrement dit, $H(\m,\n)$ est le sous-espace de $Q(\m)$ formé des éléments invariants par $\n$.
Pour $u+I(\m), v+I(\m) \in H(\m,\n)$, on a $uv+I(\m) \in H(\m,\n)$. 
En particulier, $(u+I(\m))  (v+I(\m)) = uv + I(\m)$
définit une structure d'algèbre sur $H(\m,\n)$.
\begin{remark}
Lorsque $\m=\n$, d'après (\ref{eq:W}), l'application
$$
	\Phi : H(\m,\n) \  
		\rightarrow \ {\rm End}_{U(\g)}  Q(\m)^{{\rm op}}
$$
donnée par $\Phi(u + I(\m))(v+I(\m)) = vu+I(\m)$ où
$u+I(\m) \in H(\m,\n)$ et $v+I(\m) \in Q(\m)$ est bien définie
et c'est un isomorphisme d'algèbres.
\end{remark}
%
%
%
%
\subsection{}
Soient $(e,h,f)$ un $\lsl_2$-triplet de $\g$ et $\ss$ un sous-espace gradué de $\g$ 
supplémentaire de $[\n,e]$ dans $\m^{\perp}$.
\begin{lemma} \label{l:supps}
On a
$$
    \g = [\g,e] \oplus \ss.
$$
\end{lemma}
\begin{proof}
D'après la condition (A3) de la Définition \ref{d:adm}, $\ss \cap [\g,e] = \{0\}$.
D'autre part, d'après la condition (A4), 
$\dim \n= \dim [\n,e]$. Ainsi, on obtient
$$
    \dim \ss + \dim [\g,e] = \dim \m^{\perp} - \dim \n + \dim \g - \dim \g^e = \dim \g,
$$
grâce à la condition (A6). Le lemme s'ensuit.
\end{proof}
\begin{remark}
Lorsque $\Ga$ est la graduation de Dynkin
associée à $h$, on montre en reprenant les arguments de \cite[Paragraph 2.3]{GG} que $\g^f$ est
un supplémentaire gradué de $[\n,e]$ dans $\m^\perp$, donc $\ss = \g^f$
convient dans ce cas.
\end{remark}
On pose 
$$
	\SS : = \chi+ \kappa(\ss)  \quad \subset \quad \g^* .
$$
Rappelons que $h_\Ga$ est l'élément semisimple définissant $\Ga$.
Soit $\gamma: \C^* \to G$ le sous-groupe à un paramètre associé à
$\ad h_\Gamma$.
Pour tout $j \in \Z$,
$$
	\g_j = 
		\{x \in \g \ ; \ \ga(t)(x) = t^j x, \ \text{ pour tout } \  t \in \C^* \} . 
$$ 
On définit une opération $\rho$ du groupe $\C^*$ dans $\g$ en posant
pour tous $t \in \C^*$ et $x \in \g$,
$$ 
	\rho(t) (x) = t^{a} \ga(t^{-1}) (x) . 
$$
Pour $x \in \g_j$ et $t \in \C^{*}$,
on a $\rho(t) (x) = t^{-j+a} x$. 
En particulier, comme $e \in \g_a$,
$\rho(t)(e) = e$. 
\begin{lemma}   \label{l:rho}
L'opération $\rho$ stabilise $e + \ss$ et 
$e + \m^\perp$.
De plus, elle est contractante dans ces deux variétés, 
i.e., $\lim \limits_{t \to 0} \rho(t)(e + x) =e$ pour tout $x \in \m^\perp$. 
\end{lemma}
\begin{proof}
Comme $\ss$ et $\m^\perp$ sont des sous-espaces gradués de $\g$, ils sont $\ad h_\Gamma$-stables
donc $\rho$-stables.
D'autre part, $\rho(t)(e)=e$.
On en déduit que $\rho$ stabilise $e + \ss$ et $e + \m^\perp$.
De plus, d'après la Remarque \ref{r:morth}(2), 
on a $\lim \limits_{t \to 0} \rho(t)(e+x) =e$ pour tout $x \in \m^\perp$. 
\end{proof}
\begin{theorem} \label{t:trans}
La variété affine $\SS = \chi + \kappa(\ss)$ 
$($resp. $\chi + \kappa(\m^\perp))$ est transverse aux orbites coadjointes de $\g^*$.
Précisément, pour tout 
$\xi \in \SS$ $($resp. $\xi \in \chi+\kappa(\m^\perp))$, 
on a $T_{\xi}(G.\xi)+ T_{\xi}(\SS) = \g^{*}$ 
$($resp. $T_{\xi}(G.\xi)+ T_{\xi}(\chi+\kappa(\m^\perp)) = \g^{*})$.
\end{theorem}

\begin{proof}
Montrons tout d'abord le théorème pour $\SS$. 

\smallskip

\noindent
On identifie $\g^{*}$ à $\g$ via l'isomorphisme $\kappa$. 
Pour tout $x \in e + \ss$, on a ${\rm T}_{x}(G.x) =  [\g,x]$ et ${\rm T}_{x}(e + \ss)= \ss$. 
Il suffit donc de montrer que pour tout $x \in e + \ss$ on a $[\g,x] + \ss = \g$. 
Soient alors $x \in e + \ss$ et 
$$\eta : G \times (e+\ss) \rightarrow \g$$
l'application donnée par l'opération adjointe.
Pour tout $(g,X) \in G \times (e+\ss)$, $v \in {\rm T}_gG$ et $w \in \ss$, 
l'application différentielle de $\eta$ en $(g,X)$ est donnée par 
(cf. e.g. \cite[Proposition~29.1.4]{TY}):
\begin{eqnarray} \label{eq:diff}
   {\rm d} \eta_{(g,X)} (v,w)  
                        =  g([v,X]) + g(w).
\end{eqnarray}
Ainsi ${\rm d} \eta_{(\id,e)} (v,w) = [v,e] + w$. 
On déduit que l'application ${\rm d} \eta_{(\id,e)}$ est surjective 
car $[\g,e] + \ss = \g$ (cf. Lemme \ref{l:supps}). 
Par conséquent, ${\rm d} \eta_{(\id,X)}$ est surjective pour tout $X$ 
dans un voisinage ouvert $V$ de $e$ dans $e + \ss$. 
Comme le morphisme $\eta$ est $G$-équivariant pour l'action donnée par 
$g.(g',x)=(gg',x)$, on déduit que 
l'application ${\rm d} \eta_{(g,X)}$ est surjective pour tous $X \in V$ et
$g \in G$. 
D'après (\ref{eq:diff}), il vient que 
$$
	\g =  [\g,X] + \ss 
$$
pour tout $X \in V$. 
On pose $Y := \{\rho(t)(x) \ ; \ t \in \C^{*} \}$ et 
on note $Z$ l'adhérence de $Y$ dans $e+\ss$. 
Comme l'opération de $\C^{*}$ dans $e+\ss$ est contractante, 
$e$ appartient à $Z$. 
L'intersection $V \cap Z$ est donc une partie 
ouverte non vide (elle contient $e$) de $Z$. 
Or $Z$ est une variété irréductible de $e + \ss$ 
et $Y$ est un ensemble constructible dense de $Z$.
Ainsi, $V \cap Y \neq \varnothing$.

Soient $X \in V$ et $t \in \C^{*}$ tels que $X=\rho(t)(x)$. 
Comme $\ga(t^{-1})$ est un automorphisme de Lie de $\g$, 
on a
\[
    \hspace{-2cm}  [\g, X]  =  [\g, \rho(t)(x)]  
                   =  [\g, t^{a}  \ga(t^{-1})(x)]  
                   =  t^{a}[ \g, \ga(t^{-1})(x)]
\]    
\[
               \hspace{2cm} =  t^{a}[ \ga(t^{-1}) (\g), \ga(t^{-1})(x)] 
                   =  t^{a}  \ga(t^{-1}) ([\g, x]) 
                   =  \rho(t)( [\g, x] ).
\]
On a ainsi obtenu
$\g = [\g, X] + \ss = \rho(t) ([\g, x] + \ss)$ 
car $\rho(t)(\ss) = \ss$. Par suite on a $\g = [\g, x] + \ss$ 
ce qui complète la démonstration du théorème pour $\SS$. 
Les mêmes arguments s'appliquent pour $\chi + \kappa(\m^\perp)$ 
d'après la Remarque \ref{r:morth} et le Lemme \ref{l:rho}. 
\end{proof}
Soit $N$ le sous-groupe unipotent de $G$ d'algèbre de Lie $\n$.  
\begin{lemma}            \label{l:alpha1}
L'image de l'application adjointe
$N \times (e + \m^\perp) \rightarrow \g$  
est contenue dans $e + \m^{\perp}$.
\end{lemma}
\begin{proof} 
Comme $\n$ est une sous-algèbre $\ad$-nilpotente de $\g$, 
le groupe $N$ est engendré par les éléments
$\exp (\ad x)$ où $x$ parcourt $\n$. 
Il suffit de montrer que pour tout $x \in \n$ et tout $y \in \m^{\perp}$, 
$\exp (\ad x) (e+y)$ appartient à $e+ \m^{\perp}$. 
Soient $x \in \n$ et $y \in \m^{\perp}$. 
On a 
$$
	\exp (\ad x) (e+y) = e + y + [x,e+y] + \cdots + \frac{1}{k !} (\ad x)^{k}(e+y) 
$$
pour $k$ suffisamment grand car $\ad x$ est nilpotent.  
Pour tout $i \in \N$, $(\ad x)^{i} e \in \m^{\perp}$ 
d'après la Remarque \ref{r:morth}.  
D'autre part, comme $y \in \m^\perp$, 
on a $\langle m, [x,y] \rangle  = \langle [m, x],y \rangle =0 $ pour tout $m \in \m$
d'après la condition (A5) de la Définition \ref{d:adm},  
d'où $[x,y] \in \m^\perp$. Ainsi, $ (\ad x)^{i} y \in \m^{\perp}$ pour tout $i \in \N$
et $\exp (\ad x) (e+y) \in e + \m^{\perp}$.
Le lemme s'ensuit. 
\end{proof}
\noindent
Grâce au Lemme \ref{l:alpha1}, on définit par restriction à $N \times (e + \ss)$
de l'application adjointe $N \times (e + \m^\perp) \rightarrow \g$,
l'application
$$
	\alpha  \, : \, N \times (e + \ss) \longrightarrow e + \m^\perp . 
$$ 
On définit une opération de $\C^{*}$ dans $N \times (e + \ss)$ en posant:
$$ 
	t . (g,x) := \big( \gamma (t^{-1}) g \gamma (t), \rho(t)(x) \big) , 
$$
pour tous  $t \in \C^{*}$, $g \in N$ et $x \in e+\ss$. 
L'opération est bien définie. En effet,
$\gamma(t^{-1})(\exp \ad x) \gamma(t) = \exp \ad(\ga(t^{-1})(x))$ 
appartient à $N$ pour tout $x \in \n$, $\n$ étant gradué. 
\begin{lemma}            \label{l:alpha2}

{\rm (i)} Pour tout $(g,x) \in N \times (e + \ss)$, on a: 
$\lim \limits_{t \to 0} t.(g,x) =({\bf 1}_G,e)$.

{\rm (ii)} Le morphisme $\alpha$ 
est $\C^{*}$-équivariant où l'opération de $\C^*$
dans $e + \m^\perp$ est donnée par $\rho$.
\end{lemma}
\begin{proof}
(i) %
Comme l'opération $\rho$ de $\C^*$ dans $e + \ss$ est contractante 
(cf. Lemme \ref{l:rho}) et que $\n$ est nilpotente, 
il suffit de montrer que $\gamma (t^{-1}) (\exp \ad x) \gamma(t) 
\xrightarrow[t \to 0]{} {\bf 1}_G$ pour tout $x \in \n$. 
%
Soit $x \in \n$. 
D'après la troisième inclusion de (A2),  
on a $\gamma(t^{-1}) (x) \xrightarrow[t \to 0]{} 0$, d'où
$$
	\gamma(t^{-1}) (\exp \ad x) \gamma(t) = \exp \ad (\gamma(t^{-1}) (x)) 
		\xrightarrow[t \to 0]{} \exp (0) = {\bf 1}_G .
$$
(ii) Pour $t \in \C^{*}$, $g \in G$ et $x \in e+\m^\perp$, on a:
\begin{eqnarray*}
    \alpha ( t.(g,x) ) & = & \alpha \big( \gamma(t^{-1}) g \gamma(t), \rho(t)(x) \big) 
                 \ = \  \gamma( t^{-1} ) g \gamma(t)  ( \rho(t)(x) ) \\
                & = & \gamma(t^{-1}) g \gamma(t) \big( t^a \gamma(t^{-1})(x) \big)
                 \ = \  t^a \gamma(t^{-1}) (g(x)) \\
                 & = & \rho(t) ( g(x) ) 
                 \ = \  \rho(t) \alpha( g,x ).
\end{eqnarray*}
\end{proof}
\begin{theorem}        \label{t:alpha}
L'application  
$$
	\alpha  \, : \, N \times (e + \ss) \longrightarrow e + \m^\perp
$$
est un isomorphisme de variétés affines.
\end{theorem}
\begin{proof}
Rappelons l'énoncé général suivant formulé dans \cite[Proof of Lemma 2.1]{GG}: 

\smallskip

{\em Un morphisme équivariant $\beta : X_1 \to X_2$ entre deux 
$\C^*$-variétés affines lisses munies d'opérations contractantes de $\C^*$ qui induit 
un isomorphisme entre les espaces tangents des points fixes par $\C^*$ 
est un isomorphisme.} 

\smallskip

D'après le Lemme \ref{l:alpha2}, 
il suffit de montrer que la différentielle de
$\alpha$ au point $({\bf 1}_G,e)$ induit un isomorphisme
entre l'espace tangent ${\rm T}_{({\bf 1}_G, e)}(N \times (e + \ss)) = \n \times \ss$  
de $N \times (e + \ss)$ en $({\bf 1}_G,e)$ 
et l'espace tangent ${\rm T}_e (e+\m^\perp) = \m^\perp$ 
de $e+\m^\perp$ en $e$. 
En effet, les ensembles des points fixes par $\C^*$ 
de $N \times (e + \ss)$ et $e + \m^\perp$ 
sont $\{({\bf 1}_G,e)\}$ 
et $\{e\}$ respectivement.

\smallskip

On a ${\rm d}  \alpha_{({\bf 1}_G,e)}(\n \times \ss) = [\n,e]+\ss$ 
par un calcul similaire à celui de la démonstration du Théorème \ref{t:trans}.
Comme $[\n,e]+\ss = \m^{\perp}$, 
l'application ${\rm d} \alpha_{({\bf 1}_G,e)}$ est surjective. 
Par conséquent, c'est un isomorphisme entre $ \n \times \ss$ et $\m^\perp$ 
pour des raisons de dimension. Le théorème s'ensuit.
\end{proof}
En plus de cette propriété importante de transversalité,
la variété $\SS$ admet 
une structure de Poisson, tout comme la tranche de Slodowy.
Rappelons 
tout d'abord que $\g^*$ admet une structure de Poisson canonique
donnée par
$$
	\{ F ,G \}(\xi) := \xi ([\d_\xi F , \d_\xi G ])
$$
pour tous $F,G \in \C[\g^*]$ et $\xi \in \g^*$ où $\d_\xi F \in (\g^*)^* \simeq \g$.
Toute orbite
coadjointe dans $\g^*$ admet une structure naturelle de variété 
symplectique \cite[Proposition 1.1.5]{CG}.
\begin{proposition}    \label{p:Poiss}
La variété $\SS \subset \g^*$ hérite de la structure de Poisson sur $\g^*$. 
\end{proposition}
\begin{proof}
D'après la Proposition 3.10 et la Remarque 3.11 de \cite{V}
il suffit de vérifier les conditions suivantes:

(i) $\SS$ est transverse aux orbites coadjointes de $\g^*$.

(ii) Pour tout $\xi \in \SS$, on a
$$
	 \#_\xi  {\rm Ann} (\T_\xi \SS)  \, \cap \, \T_\xi(\SS) = \{ 0\} ,
$$ 
où Ann$(\T_\xi \SS)$ est l'annulateur de $\T_\xi \SS \simeq \kappa(\ss)$ 
dans $(\T_\xi \g^*)^* \simeq (\g^*)^* \simeq \g$ 
et
$$
	\#_\xi \, : \, (\T_\xi \g^*)^*\simeq \g \longrightarrow \T_\xi \g^* \simeq \g^* , \; 
		 	\alpha \longmapsto 
							\xi([\alpha,\cdot]). 
$$
Tout d'abord, la condition (i)
est satisfaite grâce au Théorème \ref{t:trans}.
Il reste à montrer la condition (ii).
On a 
\begin{eqnarray*}
	{\rm Ann}(\T_\xi \SS)  = \{ x \in \g \ ; \ \eta (x) = 0, \text{ pour tout } \eta \in\kappa(\ss) \} 
	 = \ss^\perp . 
\end{eqnarray*}
Par conséquent, 
\begin{eqnarray*}
	\#_\xi{\rm Ann}(\T_\xi \SS) 
	= \la\, \kappa^{-1}(\xi ) , [ \ss^\perp ,\cdot ] \,\ra 
	=  \la\, [ \kappa^{-1}(\xi ) , \ss^\perp ] , \cdot \,\ra = \kappa \big( [ \kappa^{-1}(\xi ) , \ss^\perp ] \big) .
\end{eqnarray*}
On est alors amené à vérifier que l'intersection, 
$$
	\kappa ( [ \kappa^{-1}(\xi ) , \ss^\perp ]) \cap \T_\xi(\SS) 
	= \kappa ( [ \kappa^{-1}(\xi ) , \ss^\perp] ) \cap \kappa(\ss) 
	=\kappa \big( [ \kappa^{-1}(\xi ) , \ss^\perp ]  \cap \ss \big)
$$
est nulle. 
La proposition s'ensuit grâce au Lemme \ref{l:int} ci-dessous. 
\begin{lemma}    \label{l:int}
Soit $\xi \in \SS$. 
Alors $[\kappa^{-1}(\xi), \ss^{\perp}] \cap \ss= \{ 0 \}$. 
\end{lemma}
\begin{proof}
Soit $Y \subset e +\ss$ l'ensemble des
$y \in e + \ss$ vérifiant $[y, \ss^{\perp}] \cap \ss \not= \{ 0 \}$. 
Comme $\ss$ et $\ss^\perp$ 
sont $\ad h_\Gamma$-stables, on a pour tout $t \in \C^*$, 
$$\gamma(t^{-1}) ([y, \ss^{\perp}]) \cap \ss 
	= [\gamma(t^{-1}) y,  \ss^{\perp}] \cap \ss,$$ 
d'où 
$\rho(t) \big( [y, \ss^{\perp}] \cap \ss \big)
= [\rho(t) y,  \ss^{\perp}] \cap \ss$.  
Par suite, $\rho$ stabilise $Y$. 
D'autre part, d'après le Lemme \ref{l:supps}, 
$e$ appartient à $(e + \ss) \smallsetminus Y$. 
Ainsi, en tout point $y'$ dans un voisinage
ouvert $V$ de $e$ dans $e + \ss$, 
on a $y' \in (e + \ss) \smallsetminus Y$. 

Supposons que $Y \not=\varnothing$ 
et soit $y \in Y$. 
Comme $\rho$ stabilise $Y$, on a $\rho(t)y \in Y$
pour tout $t \in \C^*$. 
Or, pour $t$ suffisament petit, 
$\rho(t)y$ appartient au voisinage ouvert $V$ 
de $e$ d'après le Lemme \ref{l:rho}, d'où la contradiction.
\end{proof}
\end{proof}
\subsection{} \label{S:Kaz}
Soient $\ss$ un sous-espace gradué de $\m^{\perp}$
supplémentaire de $[\n,e]$ et $\SS := \chi + \kappa(\ss)$.
On note plus simplement $H$ l'algèbre $H(\m,\n)$ définie par (\ref{eq:W}) (cf. §\ref{S:const}). 
De la même manière, on désigne par $I$ et $Q$ 
l'idéal $I(\m)$ et le quotient $Q(\m)$ respectivement. 

Soit $\C[\g^*]$ l'algèbre des fonctions régulières sur $\g^*$.
L'isomorphisme canonique entre l'algèbre symétrique
$S(\g)$ de $\g$ et $\C[\g^*]$ envoie 
un monôme $x_1 \cdots x_j$ de $S(\g)$ sur l'élément
$F_{x_1 \cdots x_j}$ défini par 
$F_{x_1 \cdots x_j}(\xi) = \xi(x_1) \cdots \xi(x_j)$
pour tout $\xi \in \g^*$. On identifie désormais $S(\g)$ et $\C[\g^*]$
via cet isomorphisme.
L'opération adjointe de $\g$ dans $\g$ 
induit une opération, encore notée $\ad$, de $\g$ dans $S(\g)$
qui se transporte en une opération de $\g$ dans $\C[\g^*]$.
 
On définit une opération $\rho^{\sharp}$ du groupe $\C^*$ dans $\g^*$ en posant
pour tous  $t \in \C^*, \, \xi \in \g^*$:  
$$
	\rho^{\sharp} (t) (\xi) 
		:= t^{-a} \gamma(t)(\xi).
$$ 
Remarquons que $\rho^{\sharp}$ est l'opération contragrédiente de $\rho$.
Ceci induit une opération de $\C^*$ dans $\C[\g^*]$ donnée par:
$$ 
	\rho^{\sharp}(t)(F)(\xi) := F(\rho^{\sharp}(t^{-1})(\xi)) ,  
$$
pour tous $t \in \C^{*}$, 
$F \in \C[\g^*]$ et $\xi \in \g^*$.
On pose, pour $k \in \Z$, 
$$
	\C[\g^{*}](k) :=  
		\{ F \in \C[\g^{*}] \  | \ \rho^{\sharp}(t)(F) = t^k F, \text{ pour tout } t \in \C^{*} \} .
$$
\begin{lemma}        \label{l:Sg}
Pour tout $k \in \Z$, le sous-espace $\C[\g^{*}](k)$ 
de $\C[\g^*]$ est engendré par les monômes de la forme
$x = x_1 \ldots x_j$ vérifiant $(\ad h_\Gamma) x = i x$
et $i +aj=k$.
\end{lemma} 

\begin{proof}
Soit $x = x_1 \ldots x_j$ un monôme de $\C[\g^*]$ vérifiant $(\ad h_\Gamma) x = i x$
et $i +aj=k$.
Pour $\xi \in \g^{*}$ et $t \in \C^{*}$, on a:
\begin{eqnarray*}
   \rho^{\sharp}(t)(x)(\xi)  =  
			        \displaystyle \prod_{l=1}^{j} (\rho^{\sharp}(t^{-1})(\xi))(x_{l}) 
			        =  \displaystyle \prod_{l=1}^{j} (t^{a}\ga(t^{-1})(\xi))(x_{l}) 
			        =  t^{aj} \xi(\ga(t)(x)).
\end{eqnarray*}
Comme $(\ad h_\Gamma) (x) = i x$, on a  $\gamma(t)(x) = t^{i} x$. 
On en déduit que $\rho^{\sharp}(t)(x) = t^{i+a j} x$, 
i.e., $x \in \C[\g^{*}](i+a j) = \C[\g^*](k)$. 

D'autre part, un élément $x$ de $\C[\g^*](k)$
s'écrit $x = \sum \limits_{l} x_l$
où $x_l \in \C[\g^*]$ est tel que $(\ad h_\Gamma) (x_l) = i_l x_l$
et $x_l = x_{1,l} \ldots x_{j_l,l}$ avec $x_{t,l} \in \g$ et $i_l + a j_l = k$ 
pour tout $l$. Le lemme s'ensuit.
\end{proof}
\begin{lemma}   \label{l:phi}

{\rm (i)} Pour $x \in \g_j$ et $t \in \C^*$, on a $\rho^{\sharp}(t) \kappa(x) = t^{j - a}\kappa(x)$. 
En particulier, comme $e \in \g_a$, on a $\rho^{\sharp}(t) \chi = \chi$.

{\rm (ii)} Les sous-espaces $\kappa(\ss)$ et $\kappa(\m^\perp)$ sont stables par $\rho^{\sharp}$ 
et les poids de $\rho^{\sharp}$ sur $\kappa(\m^\perp) \supset \kappa(\ss)$ sont des entiers strictement négatifs.

\end{lemma}
\begin{proof}
(i) Pour $x \in \g_j$, on a
$\rho^{\sharp}(t)\kappa(x)  =  t^{-a} \ga(t)\kappa(x) 
                            =  t^{-a} \kappa(\ga(t)(x)) 
                            =  t^{j-a} \kappa(x).$
%

\smallskip

(ii) 
Les sous-espaces $\kappa(\ss)$ et $\kappa(\m^\perp)$ sont stables par $\rho^\sharp$
car $\ss$ et $\m^\perp$ sont stables par $\ad h_\Gamma$. 
Comme $\m^\perp \subset \bigoplus_{j \le a-1} \g_j$ 
(cf. Remarque \ref{r:morth}), 
les poids sur $\kappa(\m^\perp)$ 
sont les entiers $j - a$, pour $j \le a-1$. 
Ce sont donc des entiers strictement négatifs. 
\end{proof}
D'après le Lemme \ref{l:phi},  
l'application $\rho^{\sharp}$ induit une opération de $\C^*$ dans $\C[\SS] = \C[\chi + \kappa(\ss)]$ 
et dans $\C[\chi + \kappa(\m^\perp)]$. Ceci définit une structure d'algèbres graduées sur ces deux algèbres.
\begin{lemma}  \label{l:TT}
La graduation sur $\C[\SS]$ 
(resp. sur $\C[\chi + \kappa(\m^\perp)]$) est positive au sens où
$\C[\SS](k) = 0$ (resp. $\C[\chi + \kappa(\m^\perp)](k) = 0$) pour tout $k < 0$ 
et $\C[\SS](0) \simeq \C$ $($resp. $\C[\chi + \kappa(\m^\perp)](0) \simeq \C)$. 
\end{lemma}
\begin{proof}
Commençons par démontrer le lemme pour $\C[\SS]$. 
On fixe une base $(z_{1}, \cdots, z_{s})$ de $\ss$ où
$z_{i} \in \g_{d_{i}}$ avec $d_i \in \Z$. 
D'après la Remarque \ref{r:morth}, on a $d_i \le a-1$.  
Les éléments $(e,z_{1}, \cdots, z_{s})$ sont alors linéairement indépendants. 
On pose $z_{0} := e$ et on complète cet ensemble 
par des éléments homogènes relativement à la graduation $\Gamma$ 
en une base $(z_{0},z_{1}, \cdots, z_{s}, z_{s+1}, \cdots z_{r})$ de $\g$. 
Soit $(\varphi_{0}, \cdots, \varphi_{r})$ la base duale de
$(\kappa(z_{0}), \cdots, \kappa(z_{r}))$. 
Alors $\C[\g^{*}]= \C[ \varphi_{0}, \cdots, \varphi_{r}]$ 
et, pour tout $i \in \{0, \cdots, r\}$, on a 
$$
	\rho^{\sharp}(t)(\varphi_{i}) (\sum \limits_{j = 0}^{r} c_{j} \kappa(z_{j}))
			  = \varphi_{i}(t^{-d_i+a}c_{i} \kappa(z_{i})) , \text{ avec } c_i \in \C.
$$
Par suite, $\rho^{\sharp}(t)(\varphi_{i})=t^{-d_i+a} \varphi_{i}$. 
Soit $\I(\SS) \subset \C[\g^{*}]$ l'idéal des fonctions polynomiales
s'annulant sur $\SS$. 
Alors
\begin{eqnarray*}
    \I(\SS) 
	    & = & \{ F \in \C[\varphi_{0}, \cdots, \varphi_{r}] \ ; \ F(\kappa(e+x)) = 0, 
	    \; \forall \, x \in \ss \}.
\end{eqnarray*}
Par suite, on constate que $\I(\SS)$ est l'idéal de $\C[\g^{*}]$ 
engendré par les fonctions
$\varphi_{0}-1,\varphi_{s+1} \cdots, \varphi_{r}$. 
Alors, $\C[\SS] = \C[\psi_{1}, \cdots, \psi_{s}]$ où $\psi_{i}$ 
désigne la classe de $\varphi_{i}$ modulo $\I(\SS)$ pour tout $i$. 
Remarquons que $\I(\SS)$ est stable par $\rho^{\sharp}$. 
En effet, pour $F \in \I(\SS)$, 
$\rho^{\sharp}(t)(F)(\xi)= F(\rho^{\sharp}(t^{-1})(\xi)) = 0$ 
car $\rho^{\sharp}(t^{-1})(\xi) \in \SS$. 
Comme pour $i \in \{1, \cdots, s\}$, 
$\psi_i$ est de poids $- d_i + a \ge 1$ relativement à l'opération $\rho^{\sharp}$ 
car $d_i \le a - 1$.
On en déduit que la graduation sur $\C[\SS]$ est positive. 
De plus, $F \in \C[\SS]$ est de poids zéro si et seulement si $F$ est constante.

\smallskip

Les mêmes arguments s'appliquent pour $\chi + \kappa(\m^\perp)$ 
car $\m^\perp \subset \bigoplus_{j \le a-1} \g_j$ 
(cf. Remarque \ref{r:morth}). 
\end{proof}
Soit $(U^{j}(\g))_j$ la filtration standard de $U(\g)$.
L'opération adjointe de $h_\Gamma$ dans $\g$ 
s'étend de façon unique en une dérivation sur 
$U(\g)$.
On pose pour tout $i \in \Z$, 
$U_{i}(\g) := \{ x \in U(\g) \ ; \ (\ad h_\Gamma)(x) = ix \}$.
Soit $\F$ la filtration croissante de $U(\g)$ définie par, 
$$
 	\F_{k}U(\g) := \sum \limits_{i+aj \le k} U_{i}(\g) \cap U^{j}(\g), \quad  (k \in \Z). 
$$
Pour $(r,s) \in \Z^2$,  $x \in U^{r}(\g)$ (resp. $x \in U_{r}(\g)$) 
et $y \in U^{s}(\g)$ (resp. $y \in U_{s}(\g)$), 
on remarque que $[x,y] \in U^{r+s-1}(\g)$ 
(resp. $[x,y] \in U_{r+s}(\g)$).
Si $x \in \F_{r}U(\g)$ et $y \in \F_{s}U(\g)$, 
alors on a $[x,y] \in \F_{r+s-a}U(\g)$. 
On désigne par $\gr_{\F} U(\g)$ 
l'algèbre graduée de $U(\g)$ par rapport à la filtration
$\F$, i.e., $\gr_{\F} U(\g) = \bigoplus_{k} \gr_{\F,k} U(\g)$  
où $\gr_{\F,k} U(\g) := \F_{k}U(\g)/ \F_{k-1}U(\g)$.
D'après le Lemme \ref{l:Sg}, on a un isomorphisme d'algèbres graduées $\g$-équivariant
entre $\gr_{\F} U(\g)$ et $\C[\g^*]$ qui
envoie $\gr_{\F,k} U(\g)$ sur $\C[\g^*](k)$.
\begin{remark}   \label{rk:Poiss}
Comme l'algèbre graduée  $\gr_\F U(\g)$ est commutative, elle admet une structure de Poisson
donnée comme suit.
Pour $u_1 \in \F_k U(\g) / \F_{k-1} U(\g)$ et $u_2 \in \F_l U(\g) / \F_{l-1} U(\g)$,
soit $v_1$ (resp. $v_2$) un représentant de $u_1$ dans $\F_k U(\g)$ 
(resp. $u_2$ dans $\F_l U(\g)$). Alors 
$$
    \{u_1,u_2\} := v_1v_2 - v_2v_1 \mod \F_{k+l-a-2} U(\g).
$$
En particulier, l'isomorphisme entre $\gr_{\F} U(\g)$ et $\C[\g^*]$
est un morphisme d'algèbres de Poisson.
\end{remark}
Rappelons que $Q$ est le quotient $U(\g)/I$. 
Soit $\pi : U(\g) \rightarrow  Q$ l'application quotient 
et posons
$$
	\F_{k} Q := \pi (\F_{k} U(\g)), \quad k  \in \Z. 
$$ 
Ceci définit une structure de $U(\g)$-module filtré sur $Q$.
De plus, pour tout $k < 0$, on a $\F_k Q =\{0\}$.

Soit $\gr(\pi) : \gr_{\F} U(\g) \rightarrow \gr_{\F} Q$ le morphisme gradué surjectif associé à $\pi$,
i.e., pour $u \in \F_k U(\g)$, on a $\gr(\pi)(u + \F_{k-1} U(\g)) = \pi(u) + \F_{k-1} Q$.
On a la suite exacte de $\gr_{\F} U(\g)$-modules
$$
	0 \rightarrow  \gr_{\F} I \rightarrow \gr_{\F} U(\g) \rightarrow \gr_{\F} Q \rightarrow 0. 
$$
En particulier, $\gr_{\F} I$ est un idéal de $\gr_{\F} U(\g)$.
On en déduit que  $\gr_{\F} Q$ admet une structure d'algèbre
avec $\gr_{\F} Q \simeq \gr_{\F} U(\g) / \gr_{\F} I $ (cf. e.g. 
\cite[Proposition 7.5.3]{TY}).
De plus, $\gr_{\F} I$ est le noyau de $\gr( \pi )$.

Soit $\I(\chi + \kappa(\m^\perp))$ l'idéal de $\C[\g^*]$
formé des fonctions polynomiales sur $\g^{*}$ 
qui s'annulent sur $\chi + \kappa(\m^\perp)$. 
\begin{lemma}     \label{l:grp}
L'image du noyau $\gr_{\F} I$ de $\gr(\pi)$ par l'isomorphisme $\gr_\F U(\g) \rightarrow \C[\g^*]$
est $\I(\chi + \kappa(\m^\perp))$. 
\end{lemma}
\begin{proof}
Comme $\m$ est $\ad h_\Gamma$-stable, il existe une base
$$
 (y_{1}, \cdots,y_{m},y_{m+1}, \cdots, y_{r}) 
$$
de $\g$, avec $y_i \in \g_{d_i}$ ($d_i \in \Z$), telle que
$(y_{1}, \cdots, y_{m})$ soit une base de $\m$.
Soit $i \in \{1, \cdots, r\}$. 
On pose, 
$$
	\tilde{y_i} := y_i-\chi(y_i) .
$$
D'après la condition (A2), on a 
$\g_{\le  - a} \subset \m$. 
Par suite, pour $i > m$, on a $d_{i} \ge -a+1$ et $\chi(y_{i}) = \la e, y_{i} \ra = 0$ 
car $e \in \g_a$.
Par conséquent, $\tilde{y_i} = y_i \in U^{1}(\g) \cap U_{d_{i}}(\g) 
\subset \F_{a + d_{i}} U(\g)$. Soit $(y_{1}^{*}, \cdots, y_{r}^{*})$ 
la base duale de $(y_1,\ldots,y_r)$.  
On a
\begin{eqnarray}   \label{eq:chi}
	\chi = \sum \limits_{j=1}^{m} \chi(y_j) y_{j}^{*} + \sum \limits_{j=m+1}^{r} \chi(y_j) y_{j}^{*}
		    = \sum \limits_{j=1}^{m} \chi(y_j) y_{j}^{*} .
\end{eqnarray}  
En outre, $\ker \gr(\pi) = \gr_{\F} I$ est engendré par les éléments
$\tilde{y_i} + \F_{a+d_{i}-1} U(\g)$ 
pour $i \in \{1, \cdots, m\}$. 
Soit $J$ l'image de $\gr_{\F} I$ par l'isomorphisme $\gr_{\F} U(\g) \rightarrow \C[\g^*]$.
Alors $J$ est engendré par les éléments $\tilde{y_i}$ pour $i \in \{1, \cdots, m\}$, et
$$
Z(J) = \{ \sum \limits_{j=1}^{m} \chi(y_j)y_{j}^{*}+\sum \limits_{j=m+1}^{r} \lambda_j y_{j}^{*}
 \in \g^*; \  \lambda_{m+1},\ldots, \lambda_r \in \C \}
$$ 
est l'ensemble des zéros communs de $J$ dans $\C[\g^*]$.
Comme $(y_{m+1}^{*}, \cdots, y_{r}^{*})$ est une base de $\kappa(\m^\perp)$, 
$Z (J)= \chi + \kappa(\m^\perp)$ d'après (\ref{eq:chi}).  
Comme $J$ et $\gr_\F I$ sont engendrés par des fonctions affines,
ils sont radiciels, d'où $Z(J) = \gr_\F I$.
\end{proof}
\begin{proposition}    \label{p:grQ}
On a un isomorphisme $\n$-équivariant, 
$$
	\vartheta \, : \, \gr_{\F} Q \to \C[\chi + \kappa(\m^\perp) ], 
$$
entre les algèbres graduées $\gr_{\F} Q$ et $\C[\chi + \kappa(\m^\perp) ]$. 
\end{proposition}
\begin{proof}
Rappelons que l'on a un isomorphisme d'algèbres graduées $\g$-équivariant entre
$\gr_\F U(\g)$ et $\C[\g^*]$. 
D'après le Lemme \ref{l:grp}, on en déduit les isomorphismes suivants: 
$$
	\gr_\F Q \simeq \gr_\F U(\g)/{\gr_\F I} 
		\simeq \C[\g^*]/{\I (\chi + \kappa(\m^\perp))} \simeq \C[\chi+\kappa(\m^\perp)] . 
$$
Il s'ensuit qu'on a un isomorphisme d'algèbres graduées
 $\vartheta :  \gr_{\F} Q \to \C[\chi + \kappa(\m^\perp)]$. 
Comme $\gr_\F I$ est $\n$-stable, l'isomorphisme 
$\gr_\F I \simeq \I (\chi + \kappa(\m^\perp))$ est $\n$-équivariant.
Par conséquent, $\vartheta$ est $\n$-équivariant.
\end{proof}
La filtration $(\F_k Q)_k$ induit une filtration $(\F_k H := H \cap \F_k Q)_k$ sur $H$.
Par conséquent, on a une application injective graduée $\iota: \, \gr_{\F} H \hookrightarrow \gr_{\F} Q$.
C'est un morphisme d'algèbres comme on le vérifie aisément.
En résumé, 
\begin{proposition} \label{p:iota}
L'application
$\iota: \, \gr_{\F} H \hookrightarrow \gr_{\F} Q$ est 
un morphisme d'algèbres graduées injectif.
\end{proposition}
Soit $\mu: \C[\chi+\kappa(\m^\perp)] \rightarrow \C[\SS]$ 
le comorphisme correspondant à l'inclusion
$\SS \hookrightarrow \chi+\kappa(\m^\perp)$. 
D'après la Proposition \ref{p:grQ}, 
on en déduit un morphisme d'algèbres graduées
$
	\nu: \gr_{\F} H \rightarrow \C[\SS] .
$
\begin{theorem}   \label{t:gr}
Le morphisme
$$ 
	\nu : \gr_{\F} H \rightarrow \C[\SS] 
$$ 
est un isomorphisme d'algèbres de Poisson graduées.
\end{theorem}
Afin de montrer ce résultat, on suit la stratégie de \cite[Section 5]{GG} qui 
s'applique dans notre cas comme on le vérifiera. On détaille ici la démonstration
qui repose sur la théorie des suites spectrales.
Notre principale référence pour les suites spectrales est \cite{CaE}.

\medskip

Rappelons que l'opération adjointe de $\n$ dans $U(\g)$ 
induit une opération dans $Q$ (car $\n$ stabilise $I$). 
Considérons le complexe des cochaînes de Chevalley-Eilenberg du $\n$-module $Q$:
$$
	0 \to C^0 \to C^1 \to \ldots \to  C^i \to \ldots 
$$ 
où $C^{i} : = \Hom (\bigwedge\nolimits^{i} \n , Q) \simeq (\bigwedge\nolimits^{i} \n)^* \otimes Q 
	\simeq \bigwedge\nolimits^{i} \n^* \otimes Q$, 
et soit $\partial : C^{\bullet} \to C^{\bullet+1}$ la différentielle correspondante.
Pour tout $j \in \Z$, 
on pose $\n^*(j) := \{ \xi \in \n^*; \, (\ad^* h_\Gamma) \xi = j \xi \}$
où $\ad^* h_\Gamma$ désigne l'opération coadjointe de $h_\Gamma$ dans $\g^*$.
On constate que
$\n^*(j) \simeq (\g_{-j} \cap \n)^*$.   
%
Comme $\n \subset \bigoplus \limits_{j \le -1} \g_j$, 
on a $\n^* = \bigoplus \limits_{j \ge 1} \n^*(j) $.
Soit $i \in \N$. On a 
$$ 
	\bigwedge\nolimits^{i} \n^* = \bigoplus_{q \ge 1} (\bigwedge\nolimits^{i} \n^*)_q \  
		\text{ où } \ (\bigwedge\nolimits^{i} \n^*)_q
			   : = \bigoplus_{j_1+ \cdots + j_i=q} \n^*(j_1) 
			    	\bigwedge\nolimits \cdots \bigwedge\nolimits \n^*(j_i). 
$$
On définit une filtration croissante sur $C^i$ en posant: 
$$ 
	\FF_k C^i := \sum_{q+j \le k} (\bigwedge\nolimits^{i} \n^*)_q \otimes \F_j Q, 
	\quad (k \in \Z).
$$
Remarquons que $C^0= Q$ et
que $\FF_k C^0 = \F_k Q$ pour tout $k$. 
De plus, $\FF_k C^i=0$ pour $k$ (négatif) suffisament petit.
\begin{lemma}   \label{l:C}
Pour tout $i$, on a $\partial(\FF_k C^i) \subset \FF_k C^{i+1}$ .
\end{lemma}
\begin{proof}
Soit $f  \in  \FF_{k} C^{i}$. 
Montrons que $\partial f \in \FF_k C^{i+1}$. 
On peut supposer que $f = \varphi \otimes v$ 
où $\varphi \in (\bigwedge\nolimits^{i} \n^{*})_{q}$  
et $v \in \F_{j} Q$ tels que $q+j \le k$. 
Soient $Y_{1}, \cdots, Y_{i+1}$ des éléments de $\n$ 
tels que pour tout $r \in \{1, \cdots, i+1\}$, 
$Y_{r} \in \g_{\lambda_{r}}$ avec $\lambda_r \le -1$. 
Alors
\begin{eqnarray*}
 &&  \hspace{-1cm} \partial f(Y_{1}\wedge \cdots \wedge Y_{i+1}) \hspace{7cm}\\
&& = \ \sum \limits_{l=1}^{i+1} (-1)^{l} Y_{l} \cdot 
							f(Y_{1}\wedge \cdots \wedge \widehat{Y}_{l} \wedge\cdots 
								  \wedge Y_{i+1}) \\
                                       &&\  + \  \sum \limits_{1 \le l < m \le i+1} (-1)^{l+m} f([Y_{l},Y_{m}] \wedge
						 Y_{1}\wedge \cdots \wedge\widehat{Y}_{l} \wedge \cdots \wedge 
						 	\widehat{Y}_{m} \wedge \cdots \wedge Y_{i+1}) \\
   &&= \ \sum \limits_{l=1}^{i+1} (-1)^{l} Y_{l} \cdot 
							      (\varphi (Y_{1}\wedge \cdots \wedge \widehat{Y}_{l}\wedge \cdots 
							       \wedge Y_{i+1})v) \\
                                     && \ + \ \sum \limits_{1 \le l < m \le i+1} (-1)^{l+m} \varphi([Y_{l},Y_{m}] \wedge 
						    Y_{1} \wedge \cdots \wedge \widehat{Y}_{l} \wedge \cdots \wedge \widehat{Y}_{m}
						      \wedge \cdots \wedge Y_{i+1})v. \\
\end{eqnarray*}
Si $\sum \limits_{r \neq l} \l_r \not= -q$, 
alors $\varphi(Y_{1}\wedge \cdots \wedge \widehat{Y}_{l} \wedge \cdots \wedge Y_{i+1})$ 
est nul. 
D'autre part, $Y_l v \in \F_{j + \l_l} Q$. 
De même, si $\sum \limits_{r} \lambda_r \not= -q$, alors
$\varphi([Y_{l},Y_{m}]\wedge Y_{1}\wedge \cdots \wedge \widehat{Y}_{l} \wedge \cdots 
	\wedge \widehat{Y}_{m} \wedge \cdots \wedge Y_{i+1})$ est nul. 
Ainsi, 
$ 
	\partial f \in \sum \limits_{r=1}^{i+1} 
		(\bigwedge\nolimits^{i} \n^{*})_{q-\l_r} \otimes \F_{j+\l_r} Q  
						+ (\bigwedge\nolimits^{i} \n^*)_{q} \otimes \F_{j} Q
$, 
et $\partial f \in \FF_k C^{i+1}.$
\end{proof}
Rappelons que la structure de $\n$-module sur $Q$ induit
une structure de $\n$-module sur $\gr_\F Q$. 
Considérons le complexe,
$$
	G^{i} : = \Hom (\bigwedge\nolimits^{i} \n , \gr_{\F} Q) \simeq (\bigwedge\nolimits^{i} \n)^* \otimes \gr_{\F} Q 
	 \simeq \bigwedge\nolimits^{i} \n^* \otimes \gr_{\F} Q ,
$$
associé à $\gr_{\F} Q$, et notons $\delta$ la différentielle correspondante.  
On a une graduation sur $G^{i}$ donnée par 
$G^{i} = \bigoplus \limits_{k} G^{i}_{k}$ où
$G^{i}_{k} := \bigoplus \limits_{q+j=k} (\bigwedge\nolimits^{i} \n^*)_{q} 
		\otimes \F^{j}Q/\F^{j+1}Q,$ 
avec $\F^{j} Q := \F_{-j}Q.$
Pour $k \in \Z$ et $i \in \N$, on pose 
$$ 
	\FF^{k} C^i:=\FF_{-k} C^i. 
$$
La filtration $(\FF^{k} C^i)_k$ est décroissante
et $\FF^{k} C^i=0$ pour $k$ assez grand. 
\begin{remark}      \label{rk:Gi}
Soit $(k,i) \in \Z \times \N$. 

{\rm 1)} L'application $G^{i}_{k} \to \FF^{k}C^{i} /\FF^{k+1}C^{i}$ qui à  
$f_q \otimes \bar{v_j} \in (\wedge^{i} \n^*)_{q} \otimes \F^{j}Q/\F^{j+1}Q$,
avec $q+j=k$, associe la classe de $f_q \otimes v_j$, où $v_j$ est
un représentant de $\bar{v_j}$, est un isomorphisme.
On identifie désormais $G^{i}_{k}$ à $\FF^{k}C^{i} /\FF^{k+1}C^{i}$.

{\rm 2)} On a $\de(G^{i}_{k}) \subset G^{i+1}_{k}$ par le Lemme \ref{l:C}. 
\end{remark}

L'inclusion $\FF^{k}(C^i) \hookrightarrow C^i$ 
induit un morphisme $H^{i}(\FF^{k}(C^\bullet)) \to H^{i}(C^\bullet)$. 
Notons $\FF^{k}H^{i}(\n,Q)$ l'image de $H^{i}(\FF^{k}(C^\bullet))$ dans $H^{i}(C^\bullet)$. 
Ceci définit une filtration sur $H^{i} (C^\bullet) = H^{i}(\n,Q)$ et,
$$
	\FF^{k}H^{i}(\n,Q) \simeq \displaystyle{\frac{ \ker \partial  \cap \FF^k C^{i} }
					{ {\rm im} \, \partial \cap \FF^k C^{i} }} .
$$
On peut alors considérer l'algèbre graduée associée:
$$ 
	\gr_{\FF} H^{i} (\n, Q) = \bigoplus_k \gr_{\FF,k} H^{i}(\n,Q),
$$
où $\gr_{\FF,k} H^{i}(\n,Q) = \FF^{k} H^{i} (\n, Q) / \FF^{k+1} H^{i} (\n, Q)$.
\begin{remark}  
Pour $i=0$, on a $\gr_{\FF} H^{0} (\n, Q) = \gr_{\F} Q^{\n} = \gr_{\F} H$.
\end{remark}
\begin{proposition}      \label{p:H0}
{\rm (i)} On a un isomorphisme d'algèbres entre $H^{0}(\n, \gr_{\F} Q)$ et $\C[\SS].$

{\rm (ii)} Pour tout $i>0$, on a $H^{i}(\n, \gr_{\F} Q) =0.$
\end{proposition}
\begin{proof}
(i) L'isomorphisme  $\alpha : N \times \SS \to \chi + \kappa(\m^\perp)$ 
du Théorème \ref{t:alpha} induit un isomorphisme
$$
	\alpha^{*} : \C[\chi + \kappa(\m^\perp)]  \to \C[N] \otimes \C[\SS] .
$$
Le groupe $N$ opère dans $N \times \SS$ 
par $x.(y,s)=(xy,s)$, avec $x,y \in N$ et $s \in \SS$, 
et dans $\chi + \kappa(\m^\perp)$ par l'opération coadjointe.   
Ceci induit des opérations de $N$ dans $\C[N] \otimes \C[\SS]$ et 
dans $\C[\chi + \kappa(\m^\perp)]$. 
On vérifie sans peine que $\alpha$ et $\alpha^{*}$ sont $N$-équivariants 
pour ces opérations.
Ainsi,
\begin{equation*}    \label{eq:H0}
 \C[\chi + \kappa(\m^\perp)]^N \simeq (\C[N] \otimes \C[\SS] )^N 
	\simeq \C[N]^N \otimes \C[\SS] \simeq \C[\SS] 
\end{equation*}
car $\C[N]^N = \C$.  
D'autre part, d'après la Proposition \ref{p:grQ}, on a
\begin{equation*}    \label{eq:H0-2}
 \C[\chi + \kappa(\m^\perp)]^N = 
	 \C[\chi + \kappa(\m^\perp)]^{\n}
	= H^{0}(\n, \C[\chi + \kappa(\m^\perp)]) 
	\simeq H^{0} (\n, \gr_{\F} Q ). 
\end{equation*}
On obtient ainsi l'isomorphisme souhaité: 
$$
	H^0 (\n, \gr_{\F} Q ) \simeq \C[\SS] .
$$

\smallskip

(ii) Les isomorphismes $\vartheta : \gr_{\F} Q \to \C[\chi + \kappa(\m^\perp) ]$ et
$\alpha^{*} : \C[\chi + \kappa(\m^\perp)]  \to \C[N] \otimes \C[\SS]$ nous donnent, 
pour tout $i \ge 0$,
$$
	H^{i} (\n, \gr_{\F} Q ) \simeq H^{i}(\n, \C[\chi + \kappa(\m^\perp)]) 
		\simeq H^{i} (\n, \C[N] \otimes \C[\SS] ) .
$$
\noindent
L'opération de $N$ dans $N \times \SS$
implique $H^{i} (\n, \C[N] \otimes \C[\SS] ) = H^{i} (\n, \C[N] ) \otimes \C[\SS]$. 
D'autre part, d'après \cite[Theorem 10.1]{ChE} (ou \cite[Lemma 5.1]{Ho} 
ou \cite{Gr}), pour $i > 0$, $H^{i} (\n, \C[N] )$ est égal au $i$-ème groupe de cohomologie 
du complexe de De Rham pour $N$. 
Ce dernier est nul d'après \cite[Ch.\,III, Theorem 3.7]{Ha}. 
En conclusion, 
$H^{i} (\n, \C[N] \otimes \C[\SS] ) = H^{i} (\n, \C[N] ) \otimes \C[\SS] = 0$, 
ce qui achève la démonstration de (ii).
\end{proof}

Compte tenu de la Proposition \ref{p:H0}(i),  
on souhaite montrer que $\gr_\F H^0(\n,Q) \simeq H^0(\n,\gr_\F Q)$. 
Comme dans \cite[Section 5]{GG}, nous y parvenons à l'aide 
de suites spectrales. 
%
D'après le Lemme \ref{l:C}, on associe à la filtration $(\FF^k C^\bullet)_k$ 
la suite spectrale donnée par: 
$$
\begin{array}{lcl}
	 E_{m}^{k,l}  & :=  &
		\displaystyle{\frac{\FF^kC^{k+l} \cap \partial^{-1} (\FF^{k+m} C^{k+l+1}) 
			+ \FF^{k+1} C^{k+l}}
				{\partial(\FF^{k-m+1}C^{k+l-1} \cap \partial^{-1} (\FF^{k} C^{k+l}))
			+ \FF^{k+1} C^{k+l}}} ; \\
\\
	E_{\infty}^{k,l} & := & \displaystyle{ \frac{\ker \partial \cap \FF^{k} C^{k+l} 
		+ \FF^{k+1} C^{k+l}}{{\rm im} \, \partial \cap \FF^{k} C^{k+l} 
						      + \FF^{k+1} C^{k+l}}} ;  \\
\\
	E^{i} & := & H^{i}(C^\bullet) = H^{i} (\n, Q). 
\end{array}
$$
Il résulte de \cite[Ch.\,XV, Section 4]{CaE} que
\begin{eqnarray}    \label{eq:infty}
	E_{\infty}^{k,l} \simeq \gr_{\FF,k} E^{k+l} . 
\end{eqnarray}

\noindent
Compte tenu de la Remarque~\ref{rk:Gi} (2),
on peut poser pour $k \in \Z$ et $i > 0$,
$$ 
	H^{i}(\n, \gr_{\F} Q)_{k} := \displaystyle \frac{\ker \de 
		\cap G^{i}_{k}}{\de(G^{i-1}_{k})} .
$$

\begin{lemma}   \label{l:E1}

{\rm (i)} Pour tous $k,l \in \Z$, on a 
$E_{1}^{k,l} \simeq H^{k+l}(\n, \gr_{\F} Q)_{k}.$ 

{\rm (ii)} On a 
$ 
	\gr_{\FF} H^{0} (\n, Q) 
		\simeq H^{0} (\n, \gr_{\F} Q) . 
$ 

\end{lemma}
\begin{proof}

(i) On a
$$
    E_{1}^{k,l} = \frac{\FF^{k} C^{k+l} \cap \partial^{-1} (\FF^{k+1} C^{k+l+1}) 
					+ \FF^{k+1} C^{k+l}}
					{\partial(\FF^{k}C^{k+l-1} \cap \partial^{-1} (\FF^{k} C^{k+l}))
					+ \FF^{k+1} C^{k+l}}.
$$
De la Remarque \ref{rk:Gi} et du diagramme commutatif suivant, 

\begin{equation*} 
\xymatrix{
	G^{i-1}_{k} \ar[r]^{\de} \ar[d]^{\sim} & G^{i}_{k} \ar[r]^{\de} \ar[d]^{\sim} & G^{i+1}_{k} \ar[d]^{\sim} \\
	\FF^{k} C^{i-1}/\FF^{k+1} C^{i-1} \ar[r]^{\hspace{.3cm}\partial} & \FF^{k} C^{i}/\FF^{k+1} C^{i}  
						      \ar[r]^{\hspace{-.5cm}\partial} & \FF^{k} C^{i+1}/\FF^{k+1} C^{i+1}}
\end{equation*}

\noindent
on tire, 
$$
	H^{i}(\n, \gr_{\FF} Q)_{k} \simeq 
			\frac{\ker (\de) \cap \FF^{k}C^{i} /\FF^{k+1}C^{i}}{\de(\FF^{k}C^{i-1} /\FF^{k+1}C^{i-1})} 
		\simeq 
			\frac{\FF^{k}C^{i}\cap \partial^{-1} (\FF^{k+1} C^{i+1})
			+ \FF^{k+1} C^{i}}{\partial(\FF^{k} C^{i-1})+\FF^{k+1} C^{i}} .
$$
En particulier, lorsque $i=k+l$, on obtient (i).

\smallskip

(ii) D'après la Proposition \ref{p:H0}(ii), 
pour tout $(k,l) \in \Z^2$ tel que
$k+l>0$, on a  $E_{1}^{k,l} \simeq H^{k+l}(\n, \gr_{\FF} Q)_{k}=0$. 
Comme $C^{i} = 0$ pour tout $i<0$, pour tout $(k,l) \in \Z^2$ 
tel que $k+l<0$, on a  $E_{1}^{k,l}=0$.

Soit $1 < s < \infty$. 
D'après ce qui précède,  
il résulte de \cite[Ch.\,XV, Propositions 5.2, 5.2a]{CaE} que
$$ 
	E_{1}^{k,l} \simeq E_{s}^{k,l} . 
$$
La suite spectrale $(E_s^{k,l})_s$ est donc stationnaire. 
Par conséquent, $E_{1}^{k,l} \simeq E_{\infty}^{k,l}$. 
D'après (i), pour $k,l \in \Z$ vérifiant $k+l=0$ on a
$E_{1}^{k,l} \simeq H^{k+l}(\n, \gr_{\F} Q)_{k} 
= H^{0}(\n, \gr_{\F} Q)_{k}$. 
D'autre part, d'après (\ref{eq:infty}), 
on a $E_{\infty}^{k,l} \simeq \gr_{\FF,k} H^{k+l}(\n,Q) = \gr_{\FF,k} H^{0}(\n,Q)$, 
et (ii) s'ensuit. 
\end{proof}

Nous sommes désormais en mesure de démontrer 
le Théorème \ref{t:gr}: 

\begin{proof}[Démonstration du Théorème \ref{t:gr}]
D'après la Proposition~\ref{p:H0} et le Lemme~\ref{l:E1}, on a:
$$
	\C[\SS] \simeq H^{0} (\n, \gr_{\F} Q) \simeq \gr_{\FF} H^{0} (\n, Q)
		= \gr_{\F}  Q^\n =\gr_{\F} H . 
$$
Par construction, le morphisme $\nu$ est un isomorphisme d'algèbres graduées.
Il reste à montrer que c'est un morphisme 
d'algèbres de Poisson.
Cela résulte de la Remarque \ref{rk:Poiss}. 
\end{proof} 
\begin{remark}  \label{rk:kQ}
Pour tout $i>0$, on a $ H^{i} (\n, Q)=0.$
En effet, 
soient $k,l \in \Z$ tels que $k+l>0$. 
D'après le Lemme \ref{l:E1} (démonstration de (ii)), 
on a $$E_{1}^{k,l} = H^{k+l}(\n, \gr_{\F} Q)_{k}=0.$$
D'après \cite[Ch.\,XV, Proposition 5.1]{CaE}, 
on déduit que $$0 = E_{\infty}^{k,l} = \gr_{\FF,k} H^{k+l}(\n,Q).$$  
Ainsi, $\gr_{\FF} H^{i} (\n, Q)=0$ pour tout $i>0$.
Par conséquent, $\FF^{k} H^{i}(\n, Q)=\FF^{k+1}H^{i} (\n, Q)$ pour tout $k$ 
et tout $i > 0$.  
Pour $k$ assez grand, $\FF^{k} H^{i}(\n, Q)=0$.
On déduit que $\FF^{k} H^{i}(\n, Q)=0$ pour tout $k$.
La remarque s'ensuit.
\end{remark}

\smallskip

Soit $\m$ une sous-algèbre de $\g$ admissible pour $e$.
Soient $\CC$ la catégorie abélienne des $U(\g)$-modules 
à gauche finiment engendrés sur lesquels,  
pour tout $m \in \m$, l'élément $m - \chi(m)$ 
de $U(\g)$ agit localement comme 
un endomorphisme nilpotent, 
et $\CC'$ la catégorie des $H$-modules à gauche 
finiment engendrés. 

\smallskip

%
On note d'une part $Q \otimes_H - : \CC' \rightarrow \CC$ le foncteur
défini par 
$$
  (Q \otimes_H -)(V) = Q \otimes_H V \ \text{ et } \ (Q \otimes_H -)(\varphi)(q\otimes v) = q \otimes \varphi(v),
$$
pour tous $V,W$ deux objets de $\CC'$, $\varphi \in \Hom_{\CC'}(V,W)$, $q \in Q$ et $v \in V$.
D'autre part, on désigne par $\Wh: \CC \rightarrow \CC'$ le foncteur
$$
  \Wh(E) = \{ x \in E \; ; \; mx = \chi(m) x \,  \text{ pour tout } \,  m \in \m \} \ \text{ et } \ 
      \Wh(\Phi)(x) = \Phi(x)
$$
pour tous $E,F$ deux objets de $\CC$, $\Phi \in \Hom_{\CC}(E,F)$ et $x \in \Wh(E)$.
On remarque que $\Wh(\Phi)$ est bien défini car $\Phi(\Wh(E)) \subset \Wh(F)$.
\begin{theorem}   \label{t:Skr}
Le foncteur $Q \otimes_H -$ établit une équivalence 
de catégories entre les catégories $\CC'$ et $\CC$. 
L'inverse est donné par le foncteur $\Wh$.
\end{theorem} 
Les arguments de \cite[Section 6]{GG} 
s'appliquent à notre situation
compte tenu de la démonstration précédente. On omet ici la démonstration.
\section{\'{E}quivalence et problème d'isomorphisme} \label{ch:iso}
Compte tenu de la généralisation faite à la section précédente,
il est naturel de se demander si l'algèbre $H(\m,\n)$
dépend du choix de la paire admissible $(\m,\n)$.
Cette question sera traitée dans cette section. On conserve
les notations des sections précédentes.
En particulier, $\g$ est une algèbre de Lie simple complexe de dimension finie 
et $e$ est un élément nilpotent de $\g$.
%
\begin{definition} \label{d:comp}
Soient $\Ga \in \GA(e)$ et $(\m,\n), (\m',\n') \in \PA(e, \Ga)$.
On note 
$$
  (\m',\n') \peq_\Gamma (\m,\n)
$$
si $\m \subseteq \m' \subseteq \n' \subseteq \n$.

Soient $(\m,\n), (\m',\n') \in \PA(e)$.
On dit que $(\m,\n)$ et $ (\m',\n')$
sont {\rm \bf comparables} s'il existe $\Gamma \in \GA(e)$
telle que $(\m,\n), (\m',\n') \in \PA(e, \Ga)$ et telle que
$$
  (\m',\n') \peq_\Gamma (\m,\n) \ \text{ ou } \  (\m,\n) \peq_\Gamma (\m',\n').
$$
\end{definition}
\begin{example} \label{ex:comparable}
Soit $\Gamma: \g = \bigoplus_{j \in \Z} \g_j$ une bonne graduation
pour $e$. La paire
$(\g_{\le -2},\g_{<0})$ est admissible pour $e$. De plus,
pour tout $(\m,\n) \in \PA(e, \Ga)$, on a
$$
	(\m,\n ) \peq_\Gamma (\g_{\le -2},\g_{<0}).
$$

Plus généralement, si $\Gamma: \g = \bigoplus_{j \in \Z} \g_j$ est une graduation de $\g$
telle que $e\in \g_a$ pour $a > 1$ et $\g^e \cap \g_{<0} =\{0\}$
alors tout 
$(\m,\n) \in \PA(e, \Ga)$  vérifie
$$
	(\m,\n) \peq_\Gamma (\g_{\le -a},\g_{<0}),
$$
où $(\g_{\le -a},\g_{<0})$ est l'unique paire optimale d'après la Proposition \ref{p:supp}.
\end{example}
\begin{remark} \label{rk:optmax}
 Si $\Ga \in \GA(e)$ et si $(\m,\n)\in \PA(e, \Ga)$ est optimale,
alors $(\m,\n)$ est maximale pour l'ordre partiel $\peq_\Ga$ sur
les éléments de $\PA(e,\Ga)$.
\end{remark}
\begin{example} \label{ex:compsansmax}
 On suppose que $\g=\lsl_8(\C)$ et que
$e := \sum \limits_{\underset{i \not= 4}{1 \le i \le 7}} E_{i,i+1}$. 
On considère la graduation $\Gamma: \g = \bigoplus \limits_{j \in \Z} \g_j$ 
définie par l'élément semisimple
$$\frac{1}{2}\diag(7,1,-5,-11,11,5,-1,-7).$$ 
On a $\dim \g^e = 15$ et $e\in \g_3$.
Posons
$$
    \m := \g_{\le -3} \oplus \g_{-1}, \quad 
    \m':= \g_{\le -3} \oplus \C E_{6,1}, \quad \n':= \m' \oplus \C E_{7,2}
		    \oplus \C E_{8,3}\oplus \C E_{3,7}\oplus \C E_{4,8}.
$$
On vérifie par un calcul que les paires $(\m,\m)$ et $(\m',\n')$ sont admissibles pour $e$.
Elles sont comparables. Précisément, on a
$
    (\m,\m) \peq_\Gamma (\m',\n').
$
\end{example}
%
%
\begin{example} \label{ex:noncomp}
Deux paires admissibles pour $e$ ne sont pas toujours comparables.
En effet, reprenons l'exemple \ref{ex:admsl}. 
On a vu que
$$
  (\m_1,\n_1) := (\g_{\le -2} , \g_{\le -2} \oplus \C E_{2,1} \oplus \C E_{3,2}) \in \PA(e,\Ga).
$$
De la même manière on peut montrer que la paire
$$
  (\m_2,\n_2) := (\g_{\le -2} , \g_{\le -2} \oplus \C E_{3,2} \oplus \C E_{4,3})
$$
appartient à $\PA(e,\Gamma)$. On remarque que les paires $(\m_1,\n_1)$
et $(\m_2,\n_2)$ ne sont pas comparables.
\end{example}
\begin{proposition} \label{p:comp}
Si $(\m_1,\n_1), (\m_2,\n_2) \in \PA(e)$
sont comparables, alors les algèbres $H(\m_1,\n_1)$ et $H(\m_2,\n_2)$
sont isomorphes.
\end{proposition}
\begin{proof}
Sans perte de généralité, on peut supposer qu'il existe 
une $\Z$-graduation $e$-admissible $\Gamma: \g = \bigoplus_{j \in \Z} \g_j$ de $\g$
telle que $e \in \g_a$ pour $a>1$ et $(\m_2,\n_2) \peq_\Gamma (\m_1,\n_1)$. 
En particulier,
$$
  \g_{\le -a} \subseteq \m_1 \subseteq \m_2 \subseteq \n_2 \subseteq \n_1 \subseteq \g_{<0}.
$$
Comme $I(\m_1) \subset I(\m_2)$, on a la suite exacte courte
$$
  0 \rightarrow \ker \Phi \rightarrow Q(\m_1) \xrightarrow[]{\Phi} Q(\m_2) \rightarrow 0, 
$$
où $\Phi$ est le morphisme de $Q(\m_1)$ sur $Q(\m_2)$ induit par 
le morphisme quotient (surjectif) $U(\g) \rightarrow Q(\m_2)$.
Soit $\bar{\Phi}$ la restriction de $\Phi$ à
$H(\m_1,\n_1)$.
Comme $I(\m_1) \subset I(\m_2)$ et $\n_2 \subset \n_1$,
$$
	\bar{\Phi}(H(\m_1,\n_1)) \subseteq H(\m_2,\n_2).
$$
On rappelle que la filtration de Kazhdan généralisée
$\F$ induit une filtration croissante sur $H(\m_1,\n_1)$
et $H(\m_2,\n_2)$.
On note $\gr \Phi$ et $\gr \bar{\Phi}$
les morphismes gradués associés à $\Phi$ et $\bar{\Phi}$ respectivement:
$$
    \gr \Phi: \gr_{\F} Q(\m_1) \rightarrow \gr_{\F} Q(\m_2),
\  \gr \bar{\Phi}: \gr_{\F} H(\m_1,\n_1) \rightarrow \gr_{\F} H(\m_2,\n_2).
$$
Soit $\ss$ un sous-espace gradué de $\g$
supplémentaire de $[\n_2,e]$ dans $\m_2^\perp$.
D'après le Lemme \ref{l:supps}, on a $\ss \oplus [\g,e] = \g$.
Il s'ensuit que $\ss \cap [\n_1,e] =\{0\}$. 
D'après la condition (A6),
on en déduit que $\ss \oplus [\n_1,e] = \m_1^\perp$.
On pose alors 
$$
	\SS:= \kappa(e + \ss).
$$
D'après le Théorème \ref{t:gr},
on a un isomorphisme
$$
	\nu_i: \gr_\F H(\m_i,\n_i) \rightarrow \C[\SS],
$$
pour $i \in \{1,2\}$.
Montrons que $\nu_2 \circ \gr \bar{\Phi} = \nu_1$.
Soient $\mu_i$ et $\mu$ les comorphismes correspondants aux
inclusions $\SS \hookrightarrow \kappa(e + \m_i^\perp)$ ($i=1,2$), 
et $\kappa(e + \m_2^\perp) \hookrightarrow \kappa(e + \m_1^\perp) $.
Le diagramme suivant commute alors
$$
\xymatrix{
\C[\kappa(e + \m_2^\perp)] \ar[rr]^{\mu_2}  &&\C[\SS] \\
\C[\kappa(e + \m_1^\perp)]. \ar[u]^{\mu} \ar[urr]_{\mu_1}&&}
$$
Soit $\iota_i: \gr_\F H(\m_i,\n_i) \hookrightarrow \gr_\F Q(\m_i)$
le morphisme injectif donné par la Proposition \ref{p:iota} pour $i=1,2$.
On a $\nu_i = \mu_i \circ \iota_i$.
De plus, on a le diagramme commutatif suivant
$$
\xymatrix{
\gr_\F H(\m_2,\n_2) \ar@{^{(}->}[rr]^{\, \iota_2} && \gr_\F Q(\m_2)\\
\gr_\F H(\m_1,\n_1) \ar[u]^{\gr \bar{\Phi}} \ar@{^{(}->}[rr]_{\, \iota_1} && \gr_\F Q(\m_1). \ar[u]_{\gr \Phi}}
$$
Comme $\gr_\F Q(\m_i) \simeq \C[\kappa(e + \m_i^\perp)]$ pour $i \in \{1,2\}$
d'après la Proposition \ref{p:grQ}, on en déduit que le
diagramme suivant commute:
$$
\xymatrix{
\gr_\F H(\m_2,\n_2) \ar@{^{(}->}[r] & \gr_\F Q(\m_2) \simeq \C[\kappa(e + \m_2^\perp)] \ar[r]  &\C[\SS]\\
\gr_\F H(\m_1,\n_1) \ar[u]^{\gr \bar{\Phi}} \ar@{^{(}->}[r] & \gr_\F Q(\m_1) 
\simeq \C[\kappa(e + \m_1^\perp)] \ar@<5ex>[u]^{\gr \Phi} \ar@<-1ex>[u]_{\mu} \ar[ur] & }
$$
d'où $\nu_2 \circ \gr \bar{\Phi} = \nu_1$.
On en déduit que $\gr \bar{\Phi}$ est un
isomorphisme d'algèbres graduées.

Comme $\F_k Q =\{0\}$ pour tout $k < 0$ et d'après \cite[Propositions 7.5.7 and 7.5.8]{TY},
$\bar{\Phi}$ est donc un isomorphisme d'algèbres
ce qui démontre la proposition.
\end{proof}
\begin{remark}
 L'isomorphisme établi dans la démonstration de la Proposition \ref{p:comp}
n'est pas canonique car il dépend du choix du sous-espace gradué $\ss$ de $\g$.
\end{remark}

\begin{definition} \label{d:equi}
 Soient $(\m,\n), (\m',\n') \in \PA(e)$.
On dit que $(\m,\n)$ et $(\m',\n')$ sont {\rm \bf équivalentes},
et on note $(\m,\n) \sim (\m',\n')$, s'il existe une famille
finie de paires $e$-admissibles $\{(\m_i,\n_i)\}_{i \in \{1, \cdots, s\}}$
telle que 
\begin{enumerate}
 \item[{\rm (1)}] $(\m_1,\n_1) = (\m,\n)$;
 \item[{\rm (2)}] les paires $(\m_i,\n_i)$ et $(\m_{i+1},\n_{i+1})$ sont comparables
pour tout $i \in \{1,\cdots, s-1\}$;
 \item[{\rm (3)}] $(\m_s,\n_s) = (\m',\n')$.
\end{enumerate}
\end{definition}
La relation $\sim$ définit une relation d'équivalence
sur $\PA(e)$.
\begin{examples} \label{ex:equi}
{\rm (1) }On a vu dans l'Exemple \ref{ex:noncomp}
que les paires
$$
  (\m_1,\n_1) := (\g_{\le -2} , \g_{\le -2} \oplus \C E_{2,1} \oplus \C E_{3,2}) \quad \text{et} \quad
  (\m_2,\n_2) := (\g_{\le -2} , \g_{\le -2} \oplus \C E_{3,2} \oplus \C E_{4,3})
$$
ne sont pas comparables.
Cependant, on peut montrer que la sous-algèbre 
$$
  \m_3:= \g_{\le -2} \oplus \C E_{3,2}
$$
est admissible pour $e$ relativement à la même graduation $\Gamma$, et que
$$
  ( \m_3, \m_3) \peq_\Gamma (\m_1,\n_1) \text{ et } ( \m_3, \m_3) \peq_\Gamma (\m_2,\n_2).
$$
Il s'ensuit que les paires $(\m_1,\n_1)$ et $(\m_2,\n_2)$ sont équivalentes.

\medskip

{\rm (2)} Donnons un exemple de paires $e$-admissibles équivalentes
qui ne sont pas issues de la même graduation $e$-admissible.
On reprend l'Exemple
\ref{ex:bonnondyn}. On rappelle que la paire $(\g_{-2},\g_{-2})$
est la seule paire $e$-admissible relativement à $\Gamma$.
On considère ensuite la graduation $\Gamma': \g = \bigoplus_{j \in \Z} \g'_j$
définie par l'élément semisimple
 $\frac{1}{3}\diag(4,-2,-2)$. Les degrés
des matrices élémentaires $E_{i,j}$ sont donnés par la matrice suivante:
$$\left(
\begin{array}{ccc}
0 & 2 & 2 \\
-2 & 0 & 0 \\
-2 & 0 & 0 
\end{array} \right).
$$
On vérifie alors que $(\g'_{-2},\g'_{-2})$
est la seule paire $e$-admissible relativement à $\Gamma'$.
On considère enfin la graduation de Dynkin $\Ga'': \g = \bigoplus_{j \in \Z} \g''_j$
dont les degrés des matrices élémentaires $E_{i,j}$ sont donnés par la matrice suivante:
$$\left(
\begin{array}{ccc}
0 & 1& 2 \\
-1 & 0 & 1 \\
-2 & -1 & 0 
\end{array} \right).
$$
En particulier, la paire $(\g''_{-2},\g''_{-2} \oplus \g''_{-1})$ est $e$-admissible
et optimale d'après l'Exemple \ref{ex:dyn}. De plus, 
$(\g_{-2},\g_{-2})$ et $(\g'_{-2},\g'_{-2})$ appartiennent à $\PA(e,\Ga'')$
et vérifient
$$
  (\g_{-2},\g_{-2}) \peq_{\Ga''}  (\g''_{-2},\g''_{-2} \oplus \g''_{-1}) 
    \quad \text{et} \quad (\g'_{-2},\g'_{-2}) \peq_{\Ga''}  (\g''_{-2},\g''_{-2} \oplus \g''_{-1})).
$$
Il s'ensuit que les paires $(\g_{-2},\g_{-2})$ et $(\g'_{-2},\g'_{-2})$
sont équivalentes.
\end{examples}
\begin{remark} \label{rk:uniqopt}
Soit $\Gamma: \g = \bigoplus_{j \in \Z} \g_j$ une graduation de $\g$
telle que $e\in \g_a$ pour $a > 1$. Si $(\g_{\le -a},\g_{<0})$
est une paire admissible pour $e$, toutes les paires admissibles pour $e$ relativement
à $\Gamma$ sont équivalentes d'après l'Exemple \ref{ex:comparable} {\rm (2)}.
En particulier, les paires admissibles pour $e$ relativement à une graduation de Dynkin
(et plus généralement à une bonne graduation) sont équivalentes.
\end{remark}
\begin{theorem} \label{t:equi}
 Soient $(\m,\n)$ et $(\m',\n')$ deux paires $e$-admissibles équivalentes.
Alors les algèbres $H(\m,\n)$ et $H(\m',\n')$ sont isomorphes.
\end{theorem}
\begin{proof}
 Comme $(\m,\n) \sim (\m',\n')$, il existe d'après la Définition \ref{d:equi} une famille
finie $\{(\m_i,\n_i)\}_{i \in \{1, \cdots, s\}}$ de paires admissibles pour $e$ 
telle que $(\m_1,\n_1) = (\m,\n)$, $(\m_s,\n_s) = (\m',\n')$ et telle que
les paires $(\m_i,\n_i)$ et $(\m_{i+1},\n_{i+1})$ soient comparables
pour tout $i \in \{1,\cdots, s-1\}$. Il découle de la Proposition 
\ref{p:comp} que 
$$
  H(\m_i,\n_i) \simeq H(\m_{i+1},\n_{i+1}) \text{ pour tout } i \in \{1,\cdots, s-1\}.
$$ 
Alors 
$$
  H(\m_i,\n_i) \simeq H(\m_{j},\n_{j}) \text{ pour tout } i,j \in \{1,\cdots, s\}.
$$ 
En particulier, 
$$
   H(\m,\n) = H(\m_1,\n_1) \simeq H(\m_{s},\n_{s}) = H(\m',\n').
$$
\end{proof}
\begin{remark} \label{rk:distiso}
Si $e$ est distingué alors, d'après la Proposition \ref{p:dist} 
et la Remarque \ref{rk:uniqopt}, les paires $e$-admissibles sont équivalentes
et donc, d'après le Théorème \ref{t:equi}, les $W$-algèbres associées sont isomorphes.
\end{remark}
\section{Connexité des graduations admissibles} \label{ch:conn}
On conserve
les notations des sections précédentes.
La notion de paires admissibles pour $e$
relativement à des $\Z$-graduations de $\g$ peut être étendue aux $\Q$-graduations de $\g$
comme suit: une paire $(\m,\n)$ de sous-algèbres ${\rm ad}$-nilpotentes de $\g$
est dite {\bf {\em admissible pour $e$}} s'il existe une $\Q$-graduation
$\Gamma: \g = \displaystyle \bigoplus_{j\in \Q} \g_j$ de $\g$
et un entier $a > 1$ tels que les conditions (A1) à (A6)
soient vérifiées. De façon analogue, on dira qu'une $\Q$-graduation
de $\g$ est {\bf {\em admissible pour $e$}} s'il existe un entier $a > 1$ tel que $e \in \g_a$ 
et s'il existe une paire admissible pour $e$ relativement à cette graduation.
On note dans la suite $\PA_\Q(e)$ l'ensemble des paires admissibles pour $e$
relativement à des $\Q$-graduations de $\g$ et $\GA_\Q(e)$ l'ensemble des
$\Q$-graduations admissibles pour $e$. Pour $\Gamma \in \GA_\Q(e)$,
on désigne par $\PA_\Q(e,\Gamma)$ l'ensemble des paires admissibles pour $e$
relativement à la $\Q$-graduation $\Gamma$.
Pour une $\Q$-graduation $\Gamma$ de $\g$ et $\lambda \in \Q^*_+$,
on rappelle que $\l \Ga$ est la $\Q$-graduation de $\g$ définie par
l'élément semisimple $\l h_\Ga$ où $h_\Ga$ est l'élément semisimple
qui définit la graduation $\Ga$.
\begin{proposition} \label{p:lambdagamQ}
 \begin{enumerate}
  \item[{\rm (i)}] Pour tout $\Gamma \in \GA_\Q(e)$, il existe $\lambda \in \Q^*_+$
tel que $\lambda \Gamma \in \GA(e)$.
  \item[{\rm (ii)}] Pour tous $\Gamma \in \GA_\Q(e)$ et $\lambda \in \Q^*_+$, on a 
$\PA_\Q(e,\Gamma) = \PA_\Q(e, \lambda \Gamma)$.
 \end{enumerate}
\end{proposition}
\begin{definition} \label{d:adj}
 Deux $\Q$-graduations $\Ga, \Ga'$ admissibles pour $e$ sont dites
{\rm \bf adjacentes} si elles ont une paire $e$-admissible en commun, i.e.,
$\PA_\Q(e,\Ga) \cap \PA_\Q(e,\Ga') \neq \varnothing$.
\end{definition}
\begin{definition} \label{d:conn}
Deux graduations $\Gamma, \Gamma' \in \GA_\Q(e)$ sont dites {\rm \bf connexes} 
s'il existe une suite $(\Gamma_i)_{i \in \{1,\cdots,s\}}$
de $\Q$-graduations admissibles pour $e$ telle que
\begin{enumerate}
 \item[{\rm (1)}] $\Gamma = \Gamma_1$;
 \item[{\rm (2)}] les graduations $\Gamma_i$ et $\Gamma_{i+1}$ sont adjacentes pour tout $1 \le i \le s-1$;
 \item[{\rm (3)}] $\Gamma' = \Gamma_s$.
\end{enumerate}
\end{definition}
\begin{examples} \label{ex:adj}
{\rm (1)} Soit $\Gamma \in \GA_\Q(e)$. Pour tout $\lambda \in \Q^*_+$, les
graduations $\lambda \Gamma$ et $\Gamma$ sont adjacentes d'après la Proposition
\ref{p:lambdagamQ} {\rm (ii)}.

\noindent
{\rm (2)} On reprend l'Exemple \ref{ex:equi}(1) et on rappelle que $(\m_3,\m_3) \in \PA(e,\Gamma)$.
On considère la graduation de Dynkin $\g = \bigoplus_{j \in \Z} \g'_j$ de $\g$
associée à $\diag(1,1,-1,-1)$.
Les degrés des matrices élémentaires $E_{i,j}$ sont donnés sur la matrice suivante
$$
\left(
\begin{array}{cccc}
0 & 0 & 2 & 2 \\
0 & 0 & 2 & 2 \\
-2 & -2 & 0 & 0 \\
-2 & -2 & 0 & 0
\end{array}
\right).
$$ 
On remarque que $ \m_3= \g'_{-2}$.
Par conséquent, $(\m_3,\m_3)$ est une paire admissible pour $e$ commune
à la graduation $\Gamma$ et la graduation de Dynkin
qui sont donc adjacentes.
\end{examples}
\begin{example} \label{ex:nonadj}
 Deux $\Q$-graduations admissibles pour $e$ ne sont pas toujours
adjacentes. En effet, dans l'Exemple \ref{ex:equi}(2) 
les graduations $\Gamma$ et $\Gamma'$
sont connexes via $\Ga''$ mais ne sont pas adjacentes.
\end{example}

\begin{proposition}
Deux éléments $\Gamma, \Gamma' \in \GA_\Q(e)$ sont connexes
si et seulement si $\Ga$ et $\Ga'$ sont connexes via une suite d'éléments de $\GA(e)$.
\end{proposition}
\begin{proof}
L'implication réciproque est évidente.
Montrons l'implication directe. D'après l'hypothèse, il existe une suite $(\Gamma_i)_{i \in \{1,\cdots,s\}}$
de $\Q$-graduations admissibles pour $e$ telle que $\Gamma = \Gamma_1$ et $\Gamma' = \Gamma_s$
et telle que les graduations $\Gamma_i$ et $\Gamma_{i+1}$ soient adjacentes pour tout $1 \le i \le s-1$.
D'après la Proposition \ref{p:lambdagamQ}, pour tout $1 \le i \le s$, il existe
$\lambda_i \in \Q^*_+$ tel que $\lambda_i \Gamma_i \in \GA(e)$ et
$\PA_\Q(e,\Gamma_i) = \PA_\Q(e, \lambda_i \Gamma_i)$.
On en déduit que les graduations $\lambda_1 \Gamma$ et $\lambda_s \Gamma'$
sont connexes via une suite de graduations $(\l_i \Gamma_i)_{i \in \{1,\cdots,s\}}$.
Or, $\PA_\Q(e,\Gamma) = \PA_\Q(e, \lambda_1 \Gamma)$ et $\PA_\Q(e,\Gamma') = \PA_\Q(e, \lambda_s \Gamma')$.
La proposition s'ensuit.
\end{proof}
%
La suite du paragraphe est consacrée à la démonstration du théorème suivant:
\begin{theorem} \label{t:CONN}
Les $\Q$-graduations admissibles pour $e$
sont connexes entre elles.
\end{theorem}
On fixe désormais
une $\Z$-graduation $e$-admissible $\Gamma: \g=\bigoplus_{j \in \Z} \g_j$ telle que 
$e \in \g_a$ où $a >1$ et un $\lsl_2$-triplet $(e,h,f)$ de $\g$
tel que $h \in \g_0$ et $f \in \g_{-a}$ (cf. \cite[Proposition 32.1.7]{TY}).
Soient $h_\Gamma$ l'élément semisimple de $\g$
définissant $\Gamma$ et $\ss := \Vect(e,h,f)$.
On pose
$$
	t:= h_\Gamma -\frac{a}{2} h.
$$
L'élément $t$ de $\g$ est semisimple
et appartient à $\g^e \cap \g^h = \g^\ss$.

On pose $[0,1]_\Q = [0,1] \cap \Q$.
Soit $\veps \in [0,1]_\Q$.
On considère l'élément semisimple
$$
  h_\Gamma^{(\veps)} := \frac{a}{2} h + \veps t.
$$
Pour $j \in \Q$, on note 
$$
    \g_j^{(\veps)}:= \{x \in \g \, ; \; (\ad h_\Gamma^{(\veps)}) (x) = jx\}.
$$
Soit alors
$\Gamma^{(\veps)}: \g = \bigoplus_{j \in \Q} \g_j^{(\veps)}$
la $\Q$-graduation de $\g$ définie par l'élément semisimple $h_\Gamma^{(\veps)}$.
En particulier, $h_\Gamma^{(1)} = h_\Gamma$, $\Gamma^{(1)} = \Gamma$ et
$\Gamma^{(0)} =\frac{a}{2} \GaD$.
Observons aussi que $e \in \g_a^{(\veps)}$.
\begin{lemma} \label{l:epsadm}
La graduation $\Gamma^{(\veps)}$ appartient à $\GA_\Q(e)$.
\end{lemma}
\begin{proof}
D'après le Théorème \ref{t:carac}, 
comme la graduation $\Gamma$ est admissible, $\g_{\le -a} \cap \g^e =\{0\}$ .
Par suite, avec les notations de la formule (\ref{eq:condlambda}) de la Section \ref{ch:padm},
on a $|\lambda| < \frac{a}{2}(d_i+1)$.
Ainsi, pour $\veps \in [0,1]_\Q$,
$$
  |\veps \lambda| \le |\lambda| < \frac{a}{2}(d_i+1),
$$
ceci implique que $\g_{\le -a}^{(\veps)} \cap \g^e =\{0\} $,
de nouveau d'après la formule (\ref{eq:condlambda}).
Le lemme s'ensuit d'après le Théorème \ref{t:carac}.
\end{proof}
La démonstration du Théorème \ref{t:CONN} étant assez technique,
on commence par en expliquer les idées principales.
Remarquons tout d'abord qu'il suffit de démontrer que $\Ga$ est connexe
à $\Ga^{(0)}$ qui est adjacente à la graduation de Dynkin.
On va construire une suite de rationnels $0=\veps_0 < \veps_1 < \cdots < \veps_s =1$
telle que pour tout $i \in \{0, \ldots, s\}$, $\Ga^{(\veps_i)}$ et $\Ga^{(\veps_{i+1})}$
soient adjacentes. Pour cela, on reprend dans les grandes lignes 
la démonstration du Théorème \ref{t:carac}.

On rappelle quelques notations de la démonstration du Théorème \ref{t:carac}.
On a les décompositions 
$$
    \g = \bigoplus_{i=1}^r \E_i\, , \quad \E_i = \bigoplus_{\l \in \Q} \E_{i,\l}\, , \quad 
	\E_{i,\l} = \bigoplus_{l= 0}^{d_i-1} \E_{i,\l}^l \, ,
$$
où la première décomposition est la décomposition orthogonale
en composantes isotypiques de $\ss$-modules simples,
la deuxième est la décomposition en sous-espaces propres de $\ad t$
et la troisième est la décomposition en sous-espaces propres de $\ad h_\Ga$
avec $d_i$ la dimension d'un $\ss$-module simple de $\E_i$.
Rappelons aussi que pour $\l \ge 0$, 
$$
  \V_{i,\l} = \E_{i,\l} +\E_{i,-\l} \quad \text{ et } \quad  \g= \bigoplus_{i=1}^r \bigoplus_{\l \in \Q^+} \V_{i,\l}
$$
est une décomposition orthogonale relative à la forme de Killing.
Enfin pour tout $l$, 
\begin{equation} \label{eq:dimvileps}
 m_{i,\l} = \frac{\dim \E_{i,\l}}{d_i} = \dim \E_{i,\l}^l = \dim (\E_{i,\l} \cap \g^e)
\end{equation}
car $\E_{i,\l} \cap \g^e = \E_{i,\l}^{d_i -1}$.
Pour tous $\veps \in [0,1]_\Q$ et $\l \in \Q$, l'espace $\E_{i,\l}$ est stable
par $\ad h_\Ga^{(\veps)}$. De plus, pour tout $l \in \{0, \cdots, d_i-1\}$,
$\E_{i,\l}^l$ est un sous-espace propre de $\ad h_\Ga^{(\veps)}$ de valeur propre $\rho_{i,\l}^{(\veps)}+ la$
où 
$$
\rho_{i,\l}^{(\veps)} := - \frac{a}{2}(d_i-1) +\veps \l.
$$
L'ensemble des valeurs propres de $\ad h_\Gamma^{(\veps)}$ sur $\E_{i,\lambda}$
est donné par:
$$
    \Xi_{i,\l}^{(\veps)} := \{\rho_{i,\lambda}^{(\veps)} + l a \, ; \  l =0, 1, \ldots, d_i-1\}.
$$
\begin{remark} \label{rk:vpentre-a0}
On a $\Xi_{i,-\l}^{(\veps)}= - \Xi_{i,\l}^{(\veps)}$.
Il en résulte que
$$
  {\rm Card \,} \bigl( \, ]-a,0[ \  \cap \  (\Xi_{i,\l}^{(\veps)} \cup \Xi_{i,-\l}^{(\veps)}) \bigr) \le 2.
$$
De plus, si ce cardinal vaut exactement $2$, alors il existe $b \in ]0, \frac{a}{2}[$ tel que 
$$
   ]-a,0[ \  \cap \  (\Xi_{i,\l}^{(\veps)} \cup \Xi_{i,-\l}^{(\veps)})= \{-b, b-a\}.
$$
\end{remark}

Comme dans la démonstration du Théorème \ref{t:carac},
on va s'intéresser à la \og position du zéro\fg\ dans $\Xi_{i,\l}^{(\veps)}$,
d'où la définition suivante:
\begin{definition} \label{d:plac0}
On définit l'entier positif $p_{i, \lambda}^{(\veps)}$ par 

{\rm a)} $p_{i,\lambda}^{(\veps)} = 1$ si $0 < \rho_{i, \lambda}^{(\veps)}$;

{\rm b)} $p_{i,\lambda}^{(\veps)} = 2(s+1)+1$ si $\rho_{i,\lambda}^{(\veps)}+sa < 0 < \rho_{i,\lambda}^{(\veps)}+(s+1)a$
pour $s \in \{0, \cdots, d_i -2\}$;

{\rm c)} $p_{i,\lambda}^{(\veps)} = 2d_i+1$ si $\rho_{i,\lambda}^{(\veps)} + (d_i-1)a < 0$;

{\rm d)} $p_{i,\lambda}^{(\veps)} = 2(s+1)$ si $\rho_{i,\lambda}^{(\veps)} + sa =0$
pour $s \in \{0, \cdots, d_i -1\}$.

\end{definition}

On représente sur la Figure \ref{f:plac0} la position
du zéro dans $\Xi_{i,\lambda}^{(\veps)}$ selon les différents cas
de la Définition \ref{d:plac0}.
Sur cette figure, les $\bullet$ représentent les éléments
de $\Xi_{i,\lambda}^{(\veps)}$ et les $|$ représentent les positions du zéro.

\smallskip

\begin{figure}[h!]
$$
\begin{array}{lccl} 
\hbox{Cas a) :} & 
\begin{picture}(275,-10)
\put(0,5){\line(1,0){120}}
\put(120,5){\line(1,0){3}}
\put(130,5){\line(1,0){3}}
\put(140,5){\line(1,0){3}}
\put(150,5){\line(1,0){3}}
\put(160,5){\line(1,0){3}}
\put(170,5){\line(1,0){3}}
\put(180,5){\line(1,0){3}}
\put(190,5){\line(1,0){3}}
\put(200,5){\line(1,0){75}}
\put(12,2){\line(0,1){6}}
\put(10,-5){\scriptsize$0$}
\put(55,5){\circle*{4}}
\put(46,12){\Rho{}}
\put(100,5){\circle*{4}}
\put(87,12){\Rho{+a}}
\put(230,5){\circle*{4}}
\put(210,12){\Rho{+(d_{i}-1)a}}
\end{picture} 
& \hskip0em &
p_{i,\lambda}^{(\varepsilon)} = 1
\\ [2em]
\hbox{Cas b) :} & 
\begin{picture}(275,-10)
\put(0,5){\line(1,0){40}}
\put(40,5){\line(1,0){3}}
\put(50,5){\line(1,0){3}}
\put(60,5){\line(1,0){3}}
\put(70,5){\line(1,0){3}}
\put(80,5){\line(1,0){3}}
\put(90,5){\line(1,0){75}}
\put(160,5){\line(1,0){3}}
\put(175,5){\line(1,0){3}}
\put(185,5){\line(1,0){3}}
\put(195,5){\line(1,0){3}}
\put(205,5){\line(1,0){3}}
\put(215,5){\line(1,0){60}}
\put(25,5){\circle*{4}}
\put(16,12){\Rho{}}
\put(105,5){\circle*{4}}
\put(85,12){\Rho{+sa}}
\put(118,2){\line(0,1){6}}
\put(116,-5){\scriptsize$0$}
\put(150,5){\circle*{4}}
\put(130,12){\Rho{+(s+1)a}}
\put(240,5){\circle*{4}}
\put(220,12){\Rho{+(d_{i}-1)a}}
\end{picture} 
& & p_{i,\lambda}^{(\varepsilon)} = 2(s+1)+1
\\ [2em]
\hbox{Cas c) :} & 
\begin{picture}(275,-10)
\put(0,5){\line(1,0){85}}
\put(85,5){\line(1,0){3}}
\put(95,5){\line(1,0){3}}
\put(105,5){\line(1,0){3}}
\put(115,5){\line(1,0){3}}
\put(125,5){\line(1,0){3}}
\put(135,5){\line(1,0){3}}
\put(145,5){\line(1,0){3}}
\put(155,5){\line(1,0){3}}
\put(165,5){\line(1,0){110}}
\put(20,5){\circle*{4}}
\put(16,12){\Rho{}}
\put(65,5){\circle*{4}}
\put(45,12){\Rho{+a}}
\put(205,5){\circle*{4}}
\put(175,12){\Rho{+(d_{i}-1)a}}
\put(250,2){\line(0,1){6}}
\put(248,-5){\scriptsize$0$}
\end{picture} 
& & p_{i,\lambda}^{(\varepsilon)} = 2d_{i}+1
\\ [2em]
\hbox{Cas d) :} & 
\begin{picture}(275,-10)
\put(0,5){\line(1,0){40}}
\put(40,5){\line(1,0){3}}
\put(50,5){\line(1,0){3}}
\put(60,5){\line(1,0){3}}
\put(70,5){\line(1,0){3}}
\put(80,5){\line(1,0){3}}
\put(90,5){\line(1,0){75}}
\put(160,5){\line(1,0){3}}
\put(175,5){\line(1,0){3}}
\put(185,5){\line(1,0){3}}
\put(195,5){\line(1,0){3}}
\put(205,5){\line(1,0){3}}
\put(215,5){\line(1,0){60}}
\put(25,5){\circle*{4}}
\put(16,12){\Rho{}}
\put(105,5){\circle*{4}}
\put(85,12){\Rho{+sa}}
\put(105,1){\line(0,1){8}}
\put(103,-7){\scriptsize$0$}
\put(150,5){\circle*{4}}
\put(130,12){\Rho{+(s+1)a}}
\put(240,5){\circle*{4}}
\put(220,12){\Rho{+(d_{i}-1)a}}
\end{picture} 
& & p_{i,\lambda}^{(\varepsilon)} = 2(s+1)
\end{array} 
$$\caption{Position du zéro dans $\Xi_{i,\l}^{(\veps)}$ } \label{f:plac0}
\end{figure}
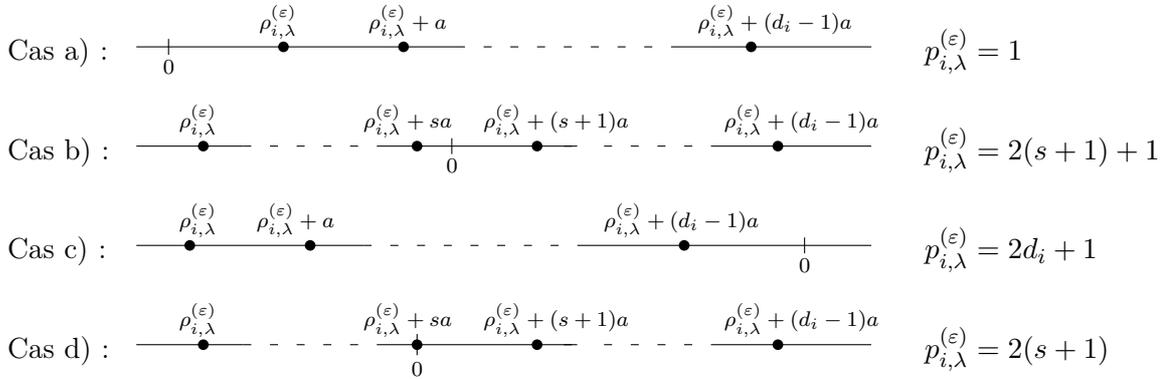 

Comme $\Xi_{i,-\lambda}^{(\veps)} = - \Xi_{i,-\lambda}^{(\veps)}$, on a
\begin{equation} \label{eq:sumplac0}
    p_{i,\lambda}^{(\veps)} + p_{i, -\lambda}^{(\veps)} = 2d_i +2.
\end{equation}
\begin{example} \label{ex:plac0}
On reprend l'Exemple \ref{ex:nonexisopt}
où l'on considère $$h := \diag(5,3,1,-1,-3,-5,2,0,-2,1,-1).$$
Il s'ensuit que les valeurs propres de $\ad t$ sont
$$
  -4, \ - \frac{7}{2}, \  -\frac{1}{2}, \  0, \  \frac{1}{2}, \  \frac{7}{2}, \  4.
$$
On note $\E_{i, \l}$ les sous-espaces propres avec $d_i = i$.
Les couples $(i,\l)$ pour $\l \neq 0$ sont alors
$$
  \big( 8, \pm \frac{1}{2}\big), \ \big( 6, \pm \frac{1}{2}\big) , \ \big( 4, \pm \frac{1}{2}\big),
  \ \big( 4, \pm \frac{7}{2}\big), \ \big( 2, \pm \frac{7}{2}\big), \ \big( 7, \pm 4\big), \ \big( 5, \pm 4\big).
$$
On obtient donc
$$
  p_{8, -\frac{1}{2}}^{(1)} = 9, \ p_{6, -\frac{1}{2}}^{(1)} =7, \ p_{4, -\frac{1}{2}}^{(1)} = 5, \ 
  p_{4, -\frac{7}{2}}^{(1)} = 7, \ p_{2, -\frac{7}{2}}^{(1)} = 5, \ p_{7, -4}^{(1)} = 11 , \ p_{5, -4}^{(1)} = 9.
$$
Les autres valeurs s'obtiennent grâce à l'égalité (\ref{eq:sumplac0}).
\end{example}

%
%
%
\begin{proposition} \label{p:ADJACENTE}
Soient $\veps, \veps' \in [0,1]_\Q$ tels que
$|p_{i,\lambda}^{(\veps)}-p_{i,\lambda}^{(\veps')}| \le 1$ pour tous $i$ et $\lambda$. 
Alors les graduations $\Gamma^{(\veps)}$ et $\Gamma^{(\veps')}$
sont adjacentes.
\end{proposition}
\begin{proof} 
On cherche une paire $e$-admissible $(\m,\n)$ commune à 
$\Gamma^{(\veps)}$ et $\Gamma^{(\veps')}$ de la forme suivante:
$$
  \m = \bigoplus_{i,\lambda} \m_{i,\lambda} \quad \text{ et } \quad 
		\n = \bigoplus_{i,\lambda} \n_{i,\lambda}
$$
telle que pour tous $i \in \{1, \cdots, r\}$ et $\lambda \in \Q^+$, les sous-espaces
$\m_{i,\lambda}$ et $\n_{i,\lambda} $ sont $\Ga^{(\veps)}$-gradués 
et $\Ga^{(\veps')}$-gradués dans $\V_{i,\lambda}$ et vérifient les conditions (C1), (C2), (C3) et (C4)
du Lemme \ref{l:stratgen} appliqué à la décomposition $\g= \bigoplus_{i=1}^r \bigoplus_{\l \in \Q^+} \V_{i,\l}$
et aux deux graduations $\Ga^{(\veps)}$ et $\Ga^{(\veps')}$.

Soient $i \in \{1, \cdots, r\}$ et $\lambda \in \Q^+$.
Pour la construction de $\m_{i,\lambda}$ et $\n_{i,\lambda} $, on distingue deux cas :
le cas où $p_{i,\lambda}^{(\veps)} = p_{i,\lambda}^{(\veps')}$ et le cas où
$|p_{i,\lambda}^{(\veps)} - p_{i,\lambda}^{(\veps')}|=1$.

\paragraph{Cas I: $p_{i,\lambda}^{(\veps)} = p_{i,\lambda}^{(\veps')}$.} Dans ce cas,
$\V_{i,\lambda} \cap \g_{\le -a}^{(\veps)} = \V_{i,\lambda} \cap \g_{\le -a}^{(\veps')}$ et
$\V_{i,\lambda} \cap \g_{<0 }^{(\veps)} = \V_{i,\lambda} \cap \g_{<0 }^{(\veps')}.$
On distingue trois sous-cas.
\begin{enumerate}
\item[{\bf (a)}] Si $\lambda = 0$, on pose
$$
     \m_{i,0} := \V_{i,0} \cap \g_{\le -a}^{(\veps)}, \quad
	  \n_{i,0} := \V_{i,0} \cap \g_{<0}^{(\veps)} = \V_{i,0} \cap \g_{\le -\frac{a}{2}}^{(\veps)}.
$$
\item[{\bf (b)}] Si $\lambda \neq 0$ et si $p_{i,\lambda}^{(\veps)}$ est pair ou appartient à $\{1, 2d_i+1\}$,
on pose
$$
    \m_{i,\lambda} = \n_{i,\lambda} := \V_{i,\lambda} \cap \g_{\le -a}^{(\veps)}.
$$
\item[{\bf (c)}] Supposons $\lambda \neq 0$ et $p_{i,\lambda}^{(\veps)}$ 
impair de la forme $2k+1$ où $k \in \{1,\ldots, d_i-1\}$.

Si $\rho_{i,\lambda}^{(\veps)} + (k-1)a \le \rho_{i,-\lambda}^{(\veps)} + (d_i-k-1)a$, on pose
$$
    \m_{i,\lambda} = \n_{i,\lambda} := \E_{i,\lambda} \cap \g_{<0}^{(\veps)} +
					   \E_{i,-\lambda} \cap \g_{\le -a}^{(\veps)}.
$$

Si $\rho_{i,\lambda}^{(\veps)} + (k-1)a > \rho_{i,-\lambda}^{(\veps)} + (d_i-k-1)a$, on pose
$$
    \m_{i,\lambda} = \n_{i,\lambda} := \E_{i,\lambda} \cap \g_{\le -a}^{(\veps)} +
					   \E_{i,-\lambda} \cap \g_{<0}^{(\veps)}.
$$
\end{enumerate}

Dans chacun des sous-cas, les conditions (C1), (C2) et (C3) sont vérifiées
par construction. De plus, la condition (C4) est satisfaite.
En effet, si $\lambda = 0$,
$$
  \dim  \m_{i,0} + \dim \n_{i,0} = (d_i-1) m_{i,0} = \dim \V_{i,0}- \dim (\V_{i,0} \cap \g^e)
$$ 
d'après (\ref{eq:dimvileps}). Si $\lambda \neq 0$,
$$
  \dim  \m_{i,\lambda} + \dim \n_{i,\lambda} = 2  [(d_i-1) m_{i,\lambda}] = 
\dim \V_{i,\lambda}- \dim (\V_{i,\lambda} \cap \g^e)
$$
d'après (\ref{eq:dimvileps}).
On constate aussi que dans tous les sous-cas, on a par construction
\begin{equation} \label{eq:prssalgI}
  \m_{i,\lambda}, \n_{i,\lambda} \subseteq \V_{i,\lambda}\cap \g_{\le -\frac{a}{2}}^{(\veps)} 
	      = \V_{i,\lambda}\cap \g_{\le -\frac{a}{2}}^{(\veps')}.
\end{equation} 
%
\paragraph{Cas II: $|p_{i,\lambda}^{(\veps)} - p_{i,\lambda}^{(\veps')}|=1$.} 
Il existe $k \in\{1, \ldots, d_i\}$ tel que l'une des conditions suivantes soit vérifiée:

\smallskip

{\bf (i)} $(p_{i,\lambda}^{(\veps)} , p_{i,\lambda}^{(\veps')})= (2k, 2k-1);$ 

{\bf (ii)} $(p_{i,\lambda}^{(\veps)} , p_{i,\lambda}^{(\veps')})= (2k, 2k+1);$ 

{\bf (iii)} $(p_{i,\lambda}^{(\veps)} , p_{i,\lambda}^{(\veps')})= (2k+1, 2k);$ 

{\bf (iv)} $(p_{i,\lambda}^{(\veps)} , p_{i,\lambda}^{(\veps')})= (2k-1, 2k).$

\medskip

\noindent
On représente sur la Figure \ref{f:plac0demo} la position du zéro
dans chacun des sous-cas (i), (ii), (iii) et (iv).
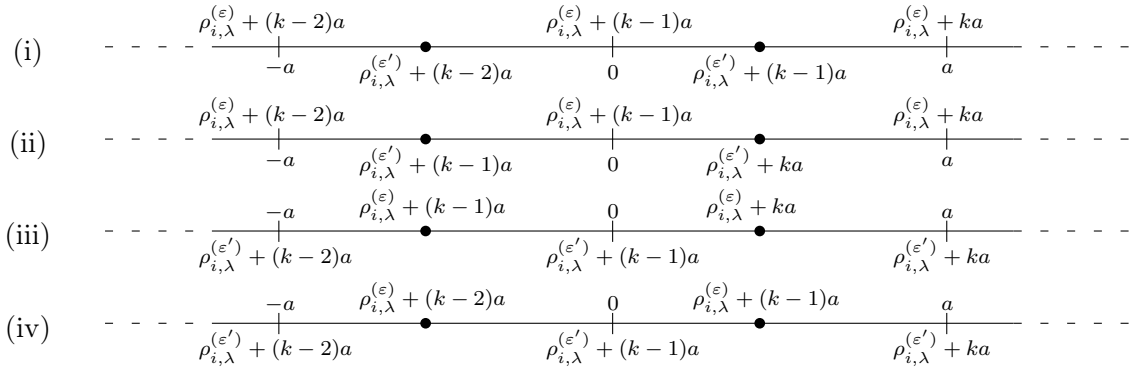
\begin{figure}[h!]
$$
\begin{array}{ccc}
\hbox{(i)} & 
\begin{picture}(400,-10)
\put(10,5){\line(1,0){3}}
\put(20,5){\line(1,0){3}}
\put(30,5){\line(1,0){3}}
\put(40,5){\line(1,0){3}}
\put(50,5){\line(1,0){300}}
\put(350,5){\line(1,0){3}}
\put(360,5){\line(1,0){3}}
\put(370,5){\line(1,0){3}}
\put(380,5){\line(1,0){3}}
\put(390,5){\line(1,0){3}}
\put(75,1){\line(0,1){8}}
\put(45,12){\Rho{+(k-2)a}}
\put(70,-5){\scriptsize$-a$}
\put(200,1){\line(0,1){8}}
\put(175,12){\Rho{+(k-1)a}}
\put(198,-7){\scriptsize$0$}
\put(325,1){\line(0,1){8}}
\put(305,12){\Rho{+k a}}
\put(323,-5){\scriptsize$a$}
\put(130,5){\circle*{4}}
\put(105,-7){\Rhop{+(k-2)a}}
\put(255,5){\circle*{4}}
\put(230,-7){\Rhop{+(k-1)a}}
\end{picture} \\ [2em]
\hbox{(ii)} &  
\begin{picture}(400,-10)
\put(10,5){\line(1,0){3}}
\put(20,5){\line(1,0){3}}
\put(30,5){\line(1,0){3}}
\put(40,5){\line(1,0){3}}
\put(50,5){\line(1,0){300}}
\put(350,5){\line(1,0){3}}
\put(360,5){\line(1,0){3}}
\put(370,5){\line(1,0){3}}
\put(380,5){\line(1,0){3}}
\put(390,5){\line(1,0){3}}
\put(75,1){\line(0,1){8}}
\put(45,12){\Rho{+(k-2)a}}
\put(70,-5){\scriptsize$-a$}
\put(200,1){\line(0,1){8}}
\put(175,12){\Rho{+(k-1)a}}
\put(198,-7){\scriptsize$0$}
\put(325,1){\line(0,1){8}}
\put(305,12){\Rho{+k a}}
\put(323,-5){\scriptsize$a$}
\put(130,5){\circle*{4}}
\put(105,-7){\Rhop{+(k-1)a}}
\put(255,5){\circle*{4}}
\put(235,-7){\Rhop{+ka}}
\end{picture} \\ [2em]
\hbox{(iii)} &  
\begin{picture}(400,-10)
\put(10,5){\line(1,0){3}}
\put(20,5){\line(1,0){3}}
\put(30,5){\line(1,0){3}}
\put(40,5){\line(1,0){3}}
\put(50,5){\line(1,0){300}}
\put(350,5){\line(1,0){3}}
\put(360,5){\line(1,0){3}}
\put(370,5){\line(1,0){3}}
\put(380,5){\line(1,0){3}}
\put(390,5){\line(1,0){3}}
\put(75,1){\line(0,1){8}}
\put(45,-7){\Rhop{+(k-2)a}}
\put(70,10){\scriptsize$-a$}
\put(200,1){\line(0,1){8}}
\put(175,-7){\Rhop{+(k-1)a}}
\put(198,10){\scriptsize$0$}
\put(325,1){\line(0,1){8}}
\put(305,-7){\Rhop{+k a}}
\put(323,10){\scriptsize$a$}
\put(130,5){\circle*{4}}
\put(105,12){\Rho{+(k-1)a}}
\put(255,5){\circle*{4}}
\put(235,12){\Rho{+ka}}
\end{picture} \\ [2em]
\hbox{(iv)} & 
\begin{picture}(400,-10)
\put(10,5){\line(1,0){3}}
\put(20,5){\line(1,0){3}}
\put(30,5){\line(1,0){3}}
\put(40,5){\line(1,0){3}}
\put(50,5){\line(1,0){300}}
\put(350,5){\line(1,0){3}}
\put(360,5){\line(1,0){3}}
\put(370,5){\line(1,0){3}}
\put(380,5){\line(1,0){3}}
\put(390,5){\line(1,0){3}}
\put(75,1){\line(0,1){8}}
\put(45,-7){\Rhop{+(k-2)a}}
\put(70,10){\scriptsize$-a$}
\put(200,1){\line(0,1){8}}
\put(175,-7){\Rhop{+(k-1)a}}
\put(198,10){\scriptsize$0$}
\put(325,1){\line(0,1){8}}
\put(305,-7){\Rhop{+k a}}
\put(323,10){\scriptsize$a$}
\put(130,5){\circle*{4}}
\put(105,12){\Rho{+(k-2)a}}
\put(255,5){\circle*{4}}
\put(230,12){\Rho{+(k-1)a}}
\end{picture} 
\end{array}
$$
\caption{Position du zéro dans les sous-cas (i), (ii), (iii) et (iv)}\label{f:plac0demo}
\end{figure}

\smallskip

\noindent On pose
\begin{equation} \label{eq:mncas2}
 \m_{i,\lambda}=\n_{i,\lambda} := \bigoplus \limits_{l =0}^{k-2}\E_{i,\lambda}^{l}
			+\bigoplus \limits_{l =0}^{d_i-k-1}\E_{i, -\lambda}^{l}.
\end{equation}
La condition (C2) est vérifiée car
$$
    \m_{i,\lambda}^\perp \cap [ \V_{i,\lambda}, e] = \bigoplus \limits_{l =1}^{k-1}\E_{i,\lambda}^{l}
			+\bigoplus \limits_{l =1}^{d_i-k}\E_{i, -\lambda}^{l} 
			  = [ \n_{i,\lambda}, e].
$$
Par construction, la condition (C3) est satisfaite.
De plus, la condition (C4) l'est aussi car
$$
  \dim  \m_{i,\lambda} + \dim \n_{i,\lambda} = 2  [(d_i-1) m_{i,\lambda}] = 
\dim V_{i,\lambda}- \dim (V_{i,\lambda} \cap \g^e)
$$
d'après (\ref{eq:dimvileps}).
Pour chacun des sous-cas (i), (ii), (iii) et (iv), on représente dans la Table \ref{f:conncasII}
les sous-espaces $\E_{i,\lambda} \cap \g_{<0}^{(\veps)}$, 
$\E_{i,- \lambda} \cap \g_{<0}^{(\veps)}$,
 $\E_{i,\lambda} \cap \g_{<0}^{(\veps')}$ et $\E_{i,- \lambda} \cap \g_{<0}^{(\veps')}$.
Ceci nous permet de conclure que la condition (C1) est vérifiée.
\begin{table}[h]
\centering
{\renewcommand{\arraystretch}{2}
\setlength{\tabcolsep}{0.5cm}
\begin{tabular}{|c|c|c|c|c|}
  \hline
  {\bf Cas II.} & $\E_{i,\lambda} \cap \g_{<0}^{(\veps)}$ 
	    & $\E_{i,- \lambda} \cap \g_{<0}^{(\veps)}$ & $\E_{i,\lambda} \cap \g_{<0}^{(\veps')}$ &
	       $\E_{i,- \lambda} \cap \g_{<0}^{(\veps')}$  \\
  \hline
  {\bf (i)} & $\bigoplus \limits_{l =0}^{k-2}\E_{i,\lambda}^{l}$  & 
	    $\bigoplus \limits_{l =0}^{d_i-k-1}\E_{i, -\lambda}^{l}$  &
		  $\bigoplus \limits_{l =0}^{k-2}\E_{i, \lambda}^{l}$ &
						$\bigoplus \limits_{l =0}^{d_i-k}\E_{i, -\lambda}^{l}$ \\[0.2cm]
  \hline
  {\bf (ii)} & $\bigoplus \limits_{l =0}^{k-2}\E_{i,\lambda}^{l}$  & 
	    $\bigoplus \limits_{l =0}^{d_i-k-1}\E_{i, -\lambda}^{l}$  &
		  $\bigoplus \limits_{l =0}^{k-1}\E_{i, \lambda}^{l}$&
						$\bigoplus \limits_{l =0}^{d_i-k-1}\E_{i, -\lambda}^{l}$ \\[0.2cm]
  \hline
  {\bf (iii)} & $\bigoplus \limits_{l =0}^{k-1}\E_{i,\lambda}^{l}$  & 
	    $\bigoplus \limits_{l =0}^{d_i-k-1}\E_{i, -\lambda}^{l}$  &
		  $\bigoplus \limits_{l =0}^{k-2}\E_{i, \lambda}^{l}$&
						$\bigoplus \limits_{l =0}^{d_i-k-1}\E_{i, -\lambda}^{l}$ \\[0.2cm]
  \hline
  {\bf (iv)} & $\bigoplus \limits_{l =0}^{k-2}\E_{i,\lambda}^{l}$  & 
	    $\bigoplus \limits_{l =0}^{d_i-k}\E_{i, -\lambda}^{l}$  &
		  $\bigoplus \limits_{l =0}^{k-2}\E_{i, \lambda}^{l}$&
						$\bigoplus \limits_{l =0}^{d_i-k-1}\E_{i, -\lambda}^{l}$ \\[0.2cm]
    \hline
\end{tabular}}

\smallskip

\caption{{\bf Cas II.}} \label{f:conncasII}
\end{table} 
On constate enfin que dans les sous-cas (i) et (ii), on a
\begin{equation} \label{eq:prssalgIIab}
 \m_{i,\lambda} = \n_{i,\lambda} = \V_{i,\lambda}\cap \g_{<0}^{(\veps)}
      = \V_{i,\lambda}\cap \g_{\le -a}^{(\veps)}.
\end{equation}
Dans les sous-cas (iii) et (iv), on a
\begin{equation} \label{eq:prssalgIIcd}
 \m_{i,\lambda} = \n_{i,\lambda} = \V_{i,\lambda}\cap \g_{<0}^{(\veps')}
	  = \V_{i,\lambda}\cap \g_{\le -a}^{(\veps')}.
\end{equation}
%
%
%
%
%
%

\paragraph{Conclusion.} On pose 
$$
  \m = \bigoplus_{i,\lambda} \m_{i,\lambda} \ \text{ et } \ 
		\n = \bigoplus_{i,\lambda} \n_{i,\lambda}.
$$
D'après le Lemme \ref{l:stratgen}, la paire $(\m,\n)$
vérifie les conditions (A2), (A3), (A4) et (A6). 
D'après (\ref{eq:prssalgI}), (\ref{eq:prssalgIIab})
et (\ref{eq:prssalgIIcd}), $\m$ et $\n$ sont deux sous-algèbres de $\g$
et vérifient (A5).
Ceci nous permet de conclure que $(\m,\n)$ est une paire admissible pour $e$ commune aux
graduations $\Gamma^{(\veps)}$ et $\Gamma^{(\veps')}$.
\end{proof}
%
\begin{proposition}\label{p:CHAINE}
 Il existe une chaîne de rationnels $\veps_0,\veps_1,\cdots,\veps_{s}=1$
telle que 
$$
    0 = \veps_0 < \veps_1 < \cdots < \veps_{s-1} < \veps_{s}=1
$$
avec pour tout $k \in \{0, \cdots, s-1\}$ et pour tout $(i,\lambda)$ on a 
\begin{enumerate}
 \item[{\rm (1)}] $p_{i,\lambda}^{(\veps)} = p_{i,\lambda}^{(\veps')}$ pour tous
$\veps,\veps' \in ]\veps_k,\veps_{k+1}[$;
 \item[{\rm (2)}] $|p_{i,\lambda}^{(\veps)} - p_{i,\lambda}^{(\veps_{k})}| \le 1$ pour tout
$\veps \in ]\veps_k,\veps_{k+1}[$;
 \item[{\rm (3)}] $|p_{i,\lambda}^{(\veps)} - p_{i,\lambda}^{(\veps_{k+1})}| \le 1$ pour tout
$\veps \in ]\veps_k,\veps_{k+1}[$.
\end{enumerate}
\end{proposition}
\begin{proof}
Observons que $p_{i,\lambda}^{(0)} = d_i+1 $. De plus, on remarque
que $p_{i,0}^{(\veps)} = d_i+1$ pour tout $\veps$. D'autre part, lorsque
$\l \neq 0$, on a
$$
    p_{i,\lambda}^{(\veps)} = 2k \Leftrightarrow 0 = \rho_{i,\lambda}^{(\veps)} + (k-1)a \Leftrightarrow
	  \veps = \frac{1}{\l} \big( \frac{a}{2}(d_i-1) - (k-1)a \big).
$$
Il en résulte que lorsque $\l \neq 0$, $p_{i,\lambda}^{(\veps)}$ n'est pair que pour un nombre 
fini de valeurs de $\veps$.
L'application $\veps \mapsto \bigl( p_{i,\l}^{(\veps)} \bigr)_{i,\l}$
converge vers $\bigl (d_i+1 \bigr )_{i,\l}$ lorsque $\veps$ tend vers $0$.
De plus, l'ensemble des couples $(i,\l)$ est fini
et, pour $(i,\l)$ fixé, l'ensemble des valeurs des $p_{i,\l}^{(\veps)}$
est contenu dans $\{1, \ldots, 2d_i+1\}$. On en déduit qu'il existe 
$0 < \veps_1 < \cdots < \veps_{s-1} < \veps_{s}=1$ tels que

\smallskip

$\ast$ pour tout $k \in \{1,\cdots, s-1\}$, il existe $(i,\l)$ tel que $\l \neq 0$
et $p_{i,\l}^{(\veps_k)}$ soit pair;

\smallskip

$\ast$ pour $\veps \in ]0,1[$ différent de $\veps_1,\veps_2, \cdots, \veps_s$, 
$p_{i,\l}^{(\veps)}$ est impair pour tout $\l \neq 0$.

\smallskip

L'assertion est alors claire.
\end{proof}
\begin{example}
On reprend l'Exemple \ref{ex:plac0} et on conserve les mêmes notations.
Par des calculs directs, on vérifie sans peine que $p_{i,\lambda}^{(\veps)} = i+1$
pour $\l \in \{-\frac{1}{2}, 0\}$ et pour tout $\veps \in [0,1]$.
De plus,
$$
  p_{i,-4}^{(\veps)} = \left\{
	    \begin{array}{ll}
	      i+1 & \text{ si } \ \veps = 0; \\
	      i+2 & \text{ si } \ 0 < \veps < \frac{3}{4}; \\
	      i+3 & \text{ si } \ \veps = \frac{3}{4}; \\
	      i+4 & \text{ si } \ \frac{3}{4} < \veps \le 1,
	    \end{array}
	\right.
\ \text{ et } \ 
 p_{i,-\frac{7}{2}}^{(\veps)} = \left\{
	    \begin{array}{ll}
	      i+1 & \text{ si } \ 0 \le \veps < \frac{3}{7}; \\
	      i+2 & \text{ si } \ \veps = \frac{3}{7}; \\
	      i+3 & \text{ si } \ \frac{3}{7} < \veps \le 1.
	    \end{array}
	\right.
$$
On en déduit que la suite $0 < \frac{3}{7} < \frac{3}{4} < 1 $
vérifie les propriétés de la Proposition \ref{p:CHAINE}.
\end{example}

\begin{theorem} \label{t:DYN}
Toute $\Q$-graduation admissible pour $e$ est connexe 
à la graduation de Dynkin.
\end{theorem}
\begin{proof}
Soit $\Gamma$ une $\Q$-graduation admissible pour $e$.
D'après la Proposition \ref{p:lambdagamQ}, on peut supposer que $\Ga$
est une $\Z$-graduation admissible pour $e$ et il suffit de montrer
que $\Ga$ est connexe à $\Ga^{(0)}$.
Considérons les rationnels $\veps_k$ donnés par la Proposition \ref{p:CHAINE}
et la suite
$$
    0=\veps_0 < \veps'_0 < \veps_1 < \veps'_1 < \cdots < \veps_{s-1} < \veps'_{s-1} < \veps_s =1.
$$
Pour tous $(i,\lambda)$ et $k \in \{0, \cdots, s-1\}$, on a
$$
    |p_{i,\lambda}^{(\veps_k)}-p_{i,\lambda}^{(\veps'_{k})}| \le 1 \ \text{ et } \ 
|p_{i,\lambda}^{(\veps'_k)}-p_{i,\lambda}^{(\veps_{k+1})}| \le 1.
$$
D'après la Proposition \ref{p:ADJACENTE}, d'une part les graduations
$\Gamma^{(\veps_k)}$ et $\Gamma^{(\veps'_{k})}$
sont adjacentes pour tout $k \in \{0, \cdots, s-1\}$
et d'autre part, les graduations $\Gamma^{(\veps'_{k})}$ et $\Gamma^{(\veps_{k+1})}$
le sont aussi.
Le théorème s'ensuit.
\end{proof}
\begin{proof}[Démonstration du Théorème \ref{t:CONN}]
Si $\Gamma, \Gamma' \in \GA_\Q(e)$ alors d'après le théorème
précédent, elles sont toutes les deux connexes à la graduation
de Dynkin. En particulier, elles sont connexes entre elles.
\end{proof}
\begin{remark}\label{rk:pbreduit}
D'après le Théorème \ref{t:CONN}, le problème d'isomorphisme des $W$-algèbres
se réduit à l'étude de la relation d'équivalence sur l'ensemble des paires $e$-admissibles
pour une graduation $e$-admissible donnée.
\end{remark}
\section{Problème d'isomorphisme pour les graduations optimales} \label{ch:opt}
On montre dans cette section que les $W$-algèbres associées aux paires $e$-admissibles
de certaines graduations sont isomorphes.
On conserve les notations des sections précédentes.
Soit $\Ga: \g = \bigoplus_{j \in \Q} \g_j$ une graduation de $\g$
telle que $e \in \g_a$ avec $a \in \N$, $a \ge 2$.
On considère la forme bilinéaire antisymétrique
$$ \Phi_e: \g_{-b} \times \g_{b-a} \rightarrow \C, \quad (x,y) \mapsto \la e, [x,y] \ra,$$
où $b \in \Q$. On rappelle que $\g_k^e$ désigne l'intersection
de $\g_k$ avec $\g^e$ pour $k\in \Q$.
Comme $\g_{b-a}$ et $\g_{a-b}$ sont en couplage
par rapport à la forme de Killing, on montre aisément
le lemme suivant.
\begin{lemma} \label{l:phie}
Soit $V$ (resp. $W$) un supplémentaire de $\g_{-b}^e$ dans $\g_{-b}$ 
(resp. de $\g_{b-a}^e$ dans $\g_{b-a}$).
Alors la restriction de $\Phi_e$ à $V \times W$ est non dégénérée. En particulier,
$\dim V = \dim W$.
\end{lemma}
\begin{lemma} \label{l:dimuorth}
Soit $U$ un sous-espace de $\g_{-b}$ tel que $U \cap \g^e = \{0\}$.
Alors $\dim (U^\perp \cap [e,\g_{b-a}]) = \dim \g_{b-a} - \dim \g_{b-a}^e - \dim U$.
\end{lemma}
\begin{proof}
Soient $V$ un supplémentaire de $\g_{-b}^e$ dans $\g_{-b}$ contenant $U$
et $W$ un supplémentaire de $\g_{b-a}^e$ dans $\g_{b-a}$.
D'après le Lemme \ref{l:phie}, la restriction de la forme de Killing à $[e,W] \times V$
est non dégénérée.
Il en résulte que
$$\dim (U^\perp \cap [e,\g_{b-a}]) = \dim \g - \dim (U + [e,\g_{b-a}]^\perp) = \dim [e,\g_{b-a}] - \dim U. $$
Le lemme s'ensuit.
\end{proof}
Supposons désormais que la graduation $\Ga$ soit admissible pour $e$.
\begin{proposition}\label{p:croisees}
Soient $(\m,\n)$ une paire $e$-admissible relativement à $\Ga$ et
$b \in ]0,\frac{a}{2}]$.  
On pose $U' = \m \cap \g_{b-a}$, $V' = \m \cap \g_{-b}$, $U = \n \cap \g_{b-a}$ et $V = \n \cap \g_{-b}$.
Alors $\dim U + \dim V'= \dim \g_{-b} - \dim \g_{-b}^e = \dim U' + \dim V $. 
\end{proposition}
\begin{proof}
Montrons tout d'abord que $\dim U + \dim V' = \dim \g_{-b} - \dim \g_{-b}^e$.
Comme la paire $(\m,\n)$ est $e$-admissible, on a, d'après la condition (A3)
de la Définition \ref{d:adm}, 
$$V'^\perp \cap [e, \g_{b-a}] = [e,U] .$$
Par suite, $\dim (V'^\perp \cap [e, \g_{b-a}]) = \dim U$ car $U \cap \g^e = \{0\}$.
D'après le Lemme \ref{l:dimuorth}, on a $\dim (V'^\perp \cap [e, \g_{b-a}]) = \dim \g_{b-a} - \dim \g_{b-a}^e - \dim V'$, 
d'où l'égalité car $\dim \g_{-b} - \dim \g_{-b}^e = \dim \g_{b-a} - \dim \g_{b-a}^e$
d'après le Lemme \ref{l:phie}.
Des arguments analogues s'appliquent pour montrer la deuxième égalité.
\end{proof}
\begin{definition} \label{d:bopt}
Soit $b >0$.
Une $\Q$-graduation $\Ga: \g = \bigoplus_{j \in \Q} \g_j$ est dite {\bf $b$-optimale pour $e$}
si $\g_{< - b} \cap \g^e = \{0\}$ et s'il existe $a \in \N$, avec $a \ge 2$,
tel que $e\in \g_a$ et $a \ge 2b$.
\end{definition}
\begin{examples}

{\rm (1)} Une graduation de Dynkin est $1$-optimale pour $e$.

{\rm (2)} Une graduation $2d$-bonne est $d$-optimale pour $e$.

{\rm (3)} La graduation $\Ga$ de l'Exemple \ref{ex:admsl} est $1$-optimale pour $e$.

{\rm (4)} Il n'existe pas toujours des graduations $b$-optimales pour $e$. C'est
le cas de la graduation $\Ga$ de l'Exemple \ref{ex:nonexisopt}.
\end{examples}

\begin{remark}
{\rm (1)} Soient $b > b' >0$ et $\Ga:\g = \bigoplus_{j \in \Q} \g_j$ une graduation $b'$-optimale pour $e$
telle que $e \in \g_a$ et $a \ge 2b$. Alors $\Ga$ est $b$-optimale pour $e$.

{\rm (2)} Si $\Ga$ est $b$-optimale et $\l \in \N^*$, alors $\l \Ga$ est $\l b$-optimale.

{\rm (3)} Une graduation $b$-optimale pour $e$
est $e$-admissible. Cela provient du Théorème \ref{t:carac}.
\end{remark}
\begin{theorem}\label{t:equiopt}
Si $\Ga$ est une graduation $b$-optimale pour $e$, alors les paires
$e$-admissibles relativement à $\Ga$ sont équivalentes entre elles.
\end{theorem}
\begin{proof}
Soit $\Ga: \g = \bigoplus_{j \in \Q} \g_j$ une graduation $b$-optimale pour $e$
telle que $a \in \N$, avec $a \ge 2$, $e\in \g_a$ et $a \ge 2b$.

On désigne par $h_\Ga$ l'élément semisimple de $\g$ définissant $\Ga$.
On note $-b_1 , -b_2 , \cdots , -b_r$ les valeurs propres de $\ad h_\Ga$
appartenant à $]-\frac{a}{2},0[$ telles que $b_i < b_{i+1}$ pour $i \in \{1,\cdots, r-1\}$.
En particulier, 
$$\g_{<0} = \g_{\le -a} \oplus \bigoplus_{i=1}^{r} (\g_{b_i-a} \oplus \g_{-b_i}) \oplus \g_{-\frac{a}{2}} .$$
Soit $(\m,\n)$ une paire $e$-admissible relativement à $\Ga$ telle que
$$\m = \g_{\le -a} \oplus \bigoplus_{i=1}^{r} (\m_{b_i-a} \oplus \m_{-b_i}) \oplus \m_{-\frac{a}{2}}, \quad 
\n = \g_{\le -a} \oplus \bigoplus_{i=1}^{r} (\n_{b_i-a} \oplus \n_{-b_i}) \oplus \n_{-\frac{a}{2}},$$
où $\m_{j} \subset \n_{j} \subset \g_j$ pour tout $j$.
Comme $\g_{< -\frac{a}{2}} \cap \g^e = \{0\}$, on a $\g_{b_i-a}^e =\{0\}$ pour tout $i$.
D'après la Proposition \ref{p:croisees}, on a alors pour tout $1 \le i \le r$
\begin{equation} \label{f:dimdemo}
\dim \m_{b_i-a} + \dim \n_{-b_i} = \dim \n_{b_i-a} + \dim \m_{-b_i} 
= \dim \g_{-b_i} - \dim \g_{-b_i}^e = \dim \g_{b_i-a}.
\end{equation}
De plus,
\begin{equation} \label{f:dimdemo1}
 \dim \m_{-\frac{a}{2}} + \dim \n_{-\frac{a}{2}} = \dim \g_{-\frac{a}{2}} - \dim \g_{-\frac{a}{2}}^e.
\end{equation}

Soit $U$ un sous-espace de $\g_{-\frac{a}{2}}$ supplémentaire de $\g_{-\frac{a}{2}}^e$
contenant $\n_{-\frac{a}{2}}$.
On pose
$$\m' := \g_{\le -a} \oplus \bigoplus_{i=1}^{r} \m_{b_i-a}, 
\quad \n' := \g_{< -\frac{a}{2}} \oplus \bigoplus_{i=1}^{r}  \n_{-b_i} \oplus U,$$
$$\m'' := \g_{< -\frac{a}{2}} , \ \text{ et } \ 
\n'' := \g_{< -\frac{a}{2}} \oplus U .$$
Alors $(\m', \n')$ et $(\m'',\n'')$ appartiennent à $\PA(e,\Ga)$.
En effet, les propriétés (A1), (A2) et (A4) sont vérifiées par construction.
Comme par construction $[e, \g_{-\frac{a}{2}}] = [e, U]$, on a (A3).
Les sous-espaces $\m'$, $\n'$, $\m''$ et $\n''$ sont des sous-algèbres de $\g$
qui vérifient (A5) car $b_i < \frac{a}{2}$ et $(\m,\n) \in \PA(e, \Ga)$.
La condition (A6) est satisfaite d'après (\ref{f:dimdemo}) et (\ref{f:dimdemo1}).
De plus, on a
$$(\m,\n) \peq_\Ga (\m',\n') \seq_\Ga (\m'',\n'').$$
Le théorème s'ensuit grâce au lemme suivant:
\begin{lemma} \label{l:recur}
Soient $U$ et $V$ sont deux sous-espaces de $\g_{-\frac{a}{2}}$ supplémentaires de $\g_{-\frac{a}{2}}^e$.
Les paires $\P_U$ et $\P_V$ sont équivalentes entre elles, où
pour $W \in \{U,V\}$, $\P_W$ désigne la paire
$(\g_{< -\frac{a}{2}} , \g_{< -\frac{a}{2}}  \oplus W)$.
\end{lemma}
\begin{proof}
Si $-\frac{a}{2}$ n'est pas une valeur propre de $\ad h_\Ga$ (respectivement si $\g_{-\frac{a}{2}}^e = \{0\}$),
le lemme est évident car $U = V = \{0\}$ (resp. car $U = V = \g_{-\frac{a}{2}}$).
Supposons donc que $-\frac{a}{2}$ est une valeur propre de $\ad h_\Ga$ et que $\g_{-\frac{a}{2}}^e \not= \{0\}$.
Montrons le lemme par récurrence sur $n := \codim_U U\cap V$.

Supposons tout d'abord que $n =1$.
D'après le Lemme \ref{l:phie}, la restriction de $\Phi_e$ à $U \times U$ est non dégénérée.
Par suite, la dimension de $U$ est un entier pair, donc la dimension de $U \cap V$ est un entier impair
car $\dim U \cap V = \dim U -1$. Il s'ensuit que la restriction de $\Phi_e$ à $U\cap V \times U \cap V$
est dégénérée. Il existe donc un élément $x$ non nul de $U \cap V$ tel que $\Phi_e(x,y) = 0$
pour tout $y \in U \cap V$.
On pose
$$D = \C x \subset U \cap V.$$
On montre alors que
$(\m,\n) := (\g_{< -\frac{a}{2}} \oplus D, 
\g_{< -\frac{a}{2}} \oplus  U\cap V) \in \PA(e,\Ga)$.
En effet, les propriétés (A1), (A2) et (A4) de la Définition \ref{d:adm} sont vérifiées par construction.
Comme $\m \subset \n \subset \g_{\le -\frac{a}{2}}$, les sous-espaces $\m$ et $\n$ sont
des sous-algèbres de $\g$ vérifiant (A5) de la même définition. 
De plus, (A6) est satisfaite car $\dim D + \dim U\cap V = \dim \g_{-\frac{a}{2}} - \dim \g_{-\frac{a}{2}}^e$.
Il nous reste à vérifier (A3). Il suffit de montrer que 
$$D^\perp \cap [e, \g_{-\frac{a}{2}}] = [e, U \cap V] .$$
Par construction, on a $[e, U \cap V] \subset D^\perp \cap [e, \g_{-\frac{a}{2}}]$.
De plus, $\dim [e, U \cap V] = \dim U -1$ et 
$\dim D^\perp \cap [e, \g_{-\frac{a}{2}}] = \dim D^\perp \cap [e,U] = \dim [e,U] -1$,
d'où l'égalité voulue.
Pour conclure, il suffit de remarquer que 
$$ \P_U \seq_\Ga (\m,\n) \peq_\Ga \P_V.$$

Supposons à présent que le lemme soit vrai pour $n-1$
et montrons-le pour $n$.
Soient $U$ et $V$ deux sous-espaces de $\g_{-\frac{a}{2}}$ supplémentaires de 
$\g_{-\frac{a}{2}}^e$
tels que $\codim_U U\cap V = n$ et
$$U = \Vect(\underline{w}, u_1, \cdots , u_n) \ \text{ et } \  V = \Vect(\underline{w}, v_1, \cdots , v_n),$$
où $\underline{w} = (w_1,\cdots, w_r)$, avec $w_i, u_i, v_i \in \g_{-\frac{a}{2}}$ pour tout $i$. 
En particulier, $U \cap V = \Vect (\underline{w})$.
Soient $\a$ et $\b$ deux réels non nuls tels que les sous-espaces
$$ U_{\a,\b} = \Vect(\underline{w}, u_1, \cdots , u_{n-1}, \a u_n + \b v_n) \ 
\text{ et } \ V_{\a,\b}=\Vect(\underline{w}, v_1, \cdots , v_{n-1},\a u_n + \b v_n)$$
de $\g_{-\frac{a}{2}}$ soient des supplémentaires de $\g_{-\frac{a}{2}}^e$.
Pour $W \in \{U_{\a,\b}, V_{\a,\b}\}$ on a $(\g_{<-\frac{a}{2}}, \g_{<-\frac{a}{2}} \oplus W) \in \PA(e,\Ga)$.
De plus, $\codim_{U_{\a,\b}} U_{\a,\b} \cap V_{\a,\b} = n-1$ 
car $U_{\a,\b} \cap V_{\a,\b} = \Vect(\underline{w},\a u_n + \b v_n)$.
D'après l'hypothèse de récurrence, les paires $(\g_{<-\frac{a}{2}}, \g_{<-\frac{a}{2}} \oplus U_{\a,\b}) $
et $(\g_{<-\frac{a}{2}}, \g_{<-\frac{a}{2}}\oplus V_{\a,\b}) $ sont équivalentes entre elles.
De plus, $\codim_U U\cap U_{\a,\b} =1$ et $\codim_V V\cap V_{\a,\b} =1$.
Il s'ensuit que les paires $(\g_{<-\frac{a}{2}}, \g_{<-\frac{a}{2}} \oplus U) $ 
et $(\g_{<-\frac{a}{2}}, \g_{<-\frac{a}{2}} \oplus U_{\a,\b})$ sont équivalentes entre elles. 
Il en est de même pour les paires $(\g_{<-\frac{a}{2}}, \g_{<-\frac{a}{2}} \oplus V) $ 
et $(\g_{<-\frac{a}{2}}, \g_{<-\frac{a}{2}}\oplus V_{\a,\b})$, d'où le lemme.
\end{proof}
\end{proof}
\begin{theorem} \label{c:equidyn}
Les paires $e$-admissibles relativement à une graduation
$b$-optimale pour $e$ sont équivalentes à la paire $e$-admissible
optimale d'une graduation de Dynkin. En particulier, les $W$-algèbres associées sont isomorphes.
\end{theorem}
\begin{proof}
On utilise dans cette démonstration les notions de la section précédente.
Considérons les rationnels $\veps_k$ donnés par la Proposition \ref{p:CHAINE}
et la suite
$$
    0=\veps_0 < \veps'_0 < \veps_1 < \veps'_1 < \cdots < \veps_{s-1} < \veps'_{s-1} < \veps_s =1.
$$
Comme dans la démonstration du Théorème \ref{t:DYN}, on applique la Proposition \ref{p:ADJACENTE}. Pour tout $k \in \{0, \cdots, s-1\}$, les graduations
$\Gamma^{(\veps_k)}$ et $\Gamma^{(\veps'_{k})}$
sont adjacentes. Il en est de même
des graduations $\Gamma^{(\veps'_{k})}$ et $\Gamma^{(\veps_{k+1})}$.
Comme $\Ga = \Ga^{(1)}$ est $b$-optimale pour $e$, on a $\g_{< -b}^{(1)} \cap \g^e =\{0\}$.
Pour tout $\lambda$, on a alors
$|\lambda| < \frac{a}{2}(d_i-1) + b$. Or, pour tout $\veps \in [0,1]_\Q$ on a
$|\veps \lambda| \le |\lambda|$. On en déduit que 
$|\veps \lambda| < \frac{a}{2}(d_i-1) + b$. En particulier, $\g_{< -b}^{(\veps)} \cap \g^e =\{0\}$.
La graduation $\Gamma^{(\veps)}$ est donc $b$-optimale pour $e$ pour tout $\veps \in [0,1]_\Q$.
Le Théorème \ref{t:equiopt} suffit alors pour conclure.
\end{proof}

\begin{remark}
À partir du corollaire précédent, on retrouve comme cas particulier que les paires
$e$-admissibles issues de bonnes graduations construites par Brundan et Goodwin dans \cite{BG}
sont équivalentes à la paire optimale d'une graduation de Dynkin.
En particulier, les $W$-algèbres associées sont isomorphes, \cite[Theorem 1]{BG}.
\end{remark}
\section{Résultats dans quelques cas particuliers} \label{ch:rang1}
Les notations des sections précédentes sont conservées.
On montre tout d'abord que les paires $e$-admissibles
relatives à une graduation vérifiant une certaine propriété sont
équivalentes entre elles. Ce résultat sera appliqué à plusieurs reprises dans la suite.
\begin{theorem} \label{t:2vp}
Si $\g_{<0} = \g_{\le -a} \oplus (\g_{b-a} +  \g_{-b})$,
où $b \in ]0,\frac{a}{2}]$, alors les paires $e$-admissibles
relativement à la graduation $\Ga$ sont équivalentes entre elles.
\end{theorem}
\begin{proof}
Compte tenu de la Remarque \ref{rk:uniqopt}, on suppose que
$\g_{<0} \cap \g^e \not= \{0\}$.
Si $b = \frac{a}{2}$, alors $\g_{< -\frac{a}{2} } \cap \g^e = \{0\}$, la graduation $\Ga$ est donc $\frac{a}{2}$-optimale. Le théorème s'ensuit d'après le Théorème \ref{t:equiopt}.

Considérons le cas où $b \not= \frac{a}{2}$.
Si $\g_{b-a}^e =\{0\}$, la graduation $\Ga$ est $b$-optimale. On conclut à partir du Théorème \ref{t:equiopt}. Supposons alors que $\g_{b-a}^e \not=\{0\}$.

Une paire $e$-admissible est de la forme $(\m'',\n'') :=(\g_{\le -a} \oplus U'' \oplus V'' , \g_{\le -a} \oplus U' \oplus V')$
avec $U'' \subset U' \subset \g_{b-a}$ et $V'' \subset V' \subset \g_{-b}$.
Soit $U$ un sous-espace de $\g_{b-a}$
supplémentaire de $\g_{b-a}^e$ contenant $U'$.
On peut alors montrer que $ (\m',\n') := (\g_{\le -a} \oplus U'', \g_{\le -a} \oplus U \oplus V') \in \PA(e,\Ga)$.
En effet, les propriétés (A1), (A2), (A4) et (A5) sont vérifiées par construction.
La propriété (A6) est satisfaite car d'après la Proposition \ref{p:croisees} et le Lemme \ref{l:phie},
on a $\dim U'' + \dim V' = \dim \g_{-b} - \dim \g_{-b}^e = \dim U$.
Il reste à vérifier que $\m'$ et $\n'$ sont des sous-algèbres de $\g$ qui vérifient
la condition (A3). Comme $(\m'',\n'') \in \PA(e, \Ga)$, on a 
$U''^\perp \cap [e,\g_{-b}] = [e,V']$. La condition (A3) est donc vraie
car $\dim [e,\g_{b-a}] = \dim U = \dim [e,U]$.
Par construction, $\m'$ est une sous-algèbre de $\g$.
Comme $\n'' \subset \n'$ et $[V',V']\subset \n''$ car $(\m'',\n'') \in \PA(e, \Ga)$,
on déduit que $\n'$ est une sous-algèbre de $\g$.
De plus,
$$(\m'',\n'') \peq_\Ga (\m',\n').$$
On pose $\m := \g_{\le -a} \oplus U$.
On peut alors vérifier aisément que la paire $(\m,\m)$ est $e$-admissible
relativement à la graduation $\Ga$ telle que
$$(\m,\m) \peq_\Ga (\m',\n').$$
Il s'ensuit que les paires $(\m'',\n'')$ et $(\m,\m)$ sont équivalentes entre elles.
Pour conclure, on considère $U$ et $V$ deux sous-espaces de $\g_{b-a}$ supplémentaires de $\g_{b-a}^e$.
Il suffit alors de montrer que les paires $e$-admissibles $\P_U$ et $\P_V$ 
sont équivalentes entre elles, où pour $W \in \{U,V\}$,  $\P_W$ désigne la paire
$(\g_{\le -a} \oplus W, \g_{\le -a} \oplus W)$.
En effet, on utilise un raisonnement par récurrence sur $n := \codim_U U\cap V$
et on montre le résultat pour $n=1$. Pour achever la récurrence, on utilise des arguments
analogues à ceux de la démonstration du Lemme \ref{l:recur}.
On a
$$ 
\dim \g_{-b} \cap [e,U \cap V]^\perp = \dim \g_{-b} - \dim [e,U \cap V] = \dim \g_{-b}^e + 1.
$$
De plus, on a $\g_{-b}^e \subset \g_{-b} \cap [e,U \cap V]^\perp$.
Il existe alors $D \subseteq \g_{-b} \cap [e,U \cap V]^\perp$ tel que $\dim D =1$
et $D \cap \g^e = \{0\}$.
On montre par suite que pour $W \in \{U,V\}$, on a
$\mathbf{A}_W :=(\g_{\le -a} \oplus U \cap V, \g_{\le -a} \oplus W \oplus D) \in \PA(e,\Ga)$.
En effet, les conditions (A1), (A2) et (A4) sont vérifiées par construction.
Les sous-espaces $\m$ et $\n$ sont des sous-algèbres de $\g$ vérifiant (A5) car $[D,D] = \{0\}$.
On a (A6) car $\dim U \cap V + \dim D = \dim W$.
Il reste à vérifier (A3). Il suffit donc de montrer que
$$(U \cap V)^\perp \cap [e,\g_{-b}] = [e,D] .$$
Par construction, on a $[e,D] \subset (U \cap V)^\perp \cap [e,\g_{-b}]$, avec,
$$\dim (U \cap V)^\perp \cap [e,\g_{-b}] = \dim [e,\g_{-b}] - \dim U \cap V = 1 = \dim [e,D],$$
d'où l'égalité voulue.
Par ailleurs, on pose $\m = \g_{\le -a} \oplus U\cap V \oplus D$.
On vérifie alors aisément que $(\m,\m) \in \PA(e,\Ga)$.
Comme
$$\P_U \peq_\Ga \mathbf{A}_U \seq_\Ga (\m,\m) \peq_\Ga \mathbf{A}_V \seq_\Ga \P_V,$$
le résultat s'ensuit.
\end{proof}

On rappelle que $\ss$ est la sous-algèbre de $\g$ engendrée par un $\lsl_2$-triplet
de $\g$ contenant $e$. Lorsque $e$ est distingué, i.e., $\rk \gs = 0$,
les paires $e$-admissibles sont équivalentes entre elles (cf. Remarque \ref{rk:distiso}).
On considère dans cette section le cas où $\rk \gs = 1$.
On tente de répondre à la question suivante:
\begin{question}\label{q:equirang1}
Sous l'hypothèse que $\rk \gs =1$, les paires admissibles pour $e$ sont-elles équivalentes
entre elles?
\end{question}
Compte tenu de la Remarque \ref{rk:pbreduit}, on s'intéresse désormais à la question suivante.
\begin{question} \label{q:equirang1gradfix}
Sous l'hypothèse que $\rk \gs =1$,
les paires admissibles pour $e$ relativement à une graduation $e$-admissible donnée
sont-elles équivalentes entre elles?
\end{question}
\subsection{Description des éléments nilpotents pour lesquels $\rk \gs =1$} \label{S:eltnil}
\subsubsection{Cas classiques}
Nous supposons dans ce paragraphe que $\g$ soit l'une des algèbres de Lie
simples de type classique suivantes: $\lsl(V)$, $\so(V)$, $\sp(V)$, où 
$V$ est un $\C$-espace vetoriel de dimension finie $n \in \N^*$.

Notre référence pour la théorie des orbites nilpotentes de $\g$ est \cite{Ja}.
Rappelons que les orbites nilpotentes de $\g$ sont paramétrées par certaines partitions de $n$.
Soit $(d_1,\ldots, d_m)$ la partition de $n$ associée à l'orbite nilpotente de $e$, où
$m \in \{1,\ldots, n\}$ et $d_1, \ldots, d_m \in \N^*$. On décrit dans ce 
paragraphe les conditions sur la partition $(d_1,\ldots, d_m)$ pour que $\rk \gs =1$.

\underline{Cas où $\g = \lsl(V)$:}
Supposons que $\g$ soit l'algèbre de Lie $\lsl(V)$ formée
des endomorphismes de $V$ de trace nulle.
D'après \cite[\S3.7, Proposition 1]{Ja}, on a:
\begin{lemma} \label{l:esl}
 Le rang de $\gs$ est égal à $1$ si et seulement si la partition
associée à $e$ est de la forme $(d_1,d_2)$. 
De plus, $\gs \simeq \lsl_2(\C)$, si $d_1 = d_2$ et $\gs \simeq \C$ sinon.
\end{lemma}
\underline{Cas où $\g=\so(V)$ ou $\sp(V)$:}
On désigne par $\Phi$ une forme bilinéaire de $V$ non dégénérée,
symétrique ou alternée.
Supposons que $\g$ soit l'algèbre de Lie simple formée des
endomorphismes $x$ de $V$ qui vérifient pour tous $v,w \in V$:
$$
  \Phi(xv,w) + \Phi(v,xw) = 0.
$$
On a alors $\g = \so(V)$ si $\Phi$ est symétrique et $\g = \sp(V)$
si $\Phi$ est alternée.
On note $G$ le groupe algébrique d'algèbre de Lie $\g$ formé des automorphismes $g$ de $V$
vérifiant pour tous $v,w \in V$:
$$\Phi(gv,gw) = \Phi(v,w). $$
Rappelons que d'après \cite[Theorem 1.6]{Ja}, si $\Phi$ est symétrique,
l'ensemble des $G$-orbites nilpotentes
dans $\so(V)$ est en bijection avec l'ensemble des
partitions de $n$ dont les parties paires ont un
nombre d'occurrences pair.
Si $\Phi$ est alternée,
l'ensemble des $G$-orbites nilpotentes
dans $\sp(V)$ est en bijection avec l'ensemble des
partitions de $n$ dont les parties impaires ont un
nombre d'occurrences pair.
On pose, pour tout $s\in \N^*$, $r_s := {\rm Card} (j \, ; \; d_j = s)$.
Si $\Phi$ est symétrique, d'après la remarque qui suit \cite[\S3.7, Proposition 2]{Ja}, 
on a l'isomorphisme
$$
	\gs \simeq \displaystyle \prod_{s\ge 1; \, s \text{ impair}} \so_{r_s}(\C) 
			\times \prod_{s\ge 1; \, s \text{ pair}} \sp_{r_s}(\C).
$$
Le lemme suivant s'ensuit:
\begin{lemma} \label{l:eso}
On suppose que $\g = \so(V)$.
Le rang de $\gs$ est égal à $1$ si et seulement si l'une des conditions suivantes
est vérifiée:
\begin{enumerate}
\item[{\rm (a)}] $m \ge 2$ et il existe $i \in \{1, \ldots, m-1\}$ tel que $d_i$ soit pair et
$$
  d_1 > d_2 > \cdots > d_i = d_{i+1} > \cdots > d_m.
$$
Dans ce cas, $\gs \simeq \sp_2(\C) \simeq \lsl_2(\C)$ et $d_j$ est impair pour tout $j \notin \{i, i+1\}$.
\item[{\rm (b)}] $m \ge 2$ et il existe $i \in \{1, \ldots, m-1\}$ tel que $d_i$ soit impair et
$$
  d_1 > d_2 > \cdots > d_i = d_{i+1} > \cdots > d_m.
$$
Dans ce cas, $\gs \simeq \C$ et $d_j$ est impair pour tout $j \in \{1,\cdots, m\}$.
\item[{\rm (c)}] $m \ge 3$ et il existe $i \in \{2, \ldots, m-1\}$ tel que $d_i$ soit impair et
$$
  d_1 > d_2 > \cdots > d_{i-1} = d_i = d_{i+1} > \cdots > d_m.
$$
Dans ce cas, $\gs \simeq \so_3(\C) \simeq \lsl_2(\C)$ et $d_j$ est impair pour tout $j \in \{1,\cdots, m\}$.
\end{enumerate}
\end{lemma}
Si $\Phi$ est alternée, toujours d'après la remarque qui suit \cite[\S3.7, Proposition 2]{Ja}, 
on a l'isomorphisme
$$
	\gs \simeq \displaystyle \prod_{s\ge 1; \, s \text{ impair}} \sp_{r_s}(\C) 
			\times \prod_{s\ge 1; \, s \text{ pair}} \so_{r_s}(\C).
$$
Le lemme suivant s'ensuit:
\begin{lemma} \label{l:esp}
On suppose que $\g = \sp(V)$.
Le rang de $\gs$ est égal à $1$ si et seulement si l'une des conditions suivantes
est vérifiée:
\begin{enumerate}
\item[{\rm (a)}] $m \ge 2$ et il existe $i \in \{1, \ldots, m-1\}$ tel que $d_i$ soit impair et
$$
  d_1 > d_2 > \cdots > d_i = d_{i+1} > \cdots > d_m.
$$
Dans ce cas, $\gs \simeq \sp_2(\C) \simeq \lsl_2(\C)$ et $d_j$ est pair pour tout $j \notin \{i, i+1\}$.
\item[{\rm (b)}] $m \ge 2$ et il existe $i \in \{1, \ldots, m-1\}$ tel que $d_i$ soit pair et
$$
  d_1 > d_2 > \cdots > d_i = d_{i+1} > \cdots > d_m.
$$
Dans ce cas, $\gs \simeq \C$ et $d_j$ est pair pour tout $j \in \{1,\cdots, m\}$.
\item[{\rm (c)}] $m \ge 2$ et il existe $i \in \{2, \ldots, m-1\}$ tel que $d_i$ soit pair et
$$
  d_1 > d_2 > \cdots > d_{i-1} = d_i = d_{i+1} > \cdots > d_m.
$$
Dans ce cas, $\gs \simeq \so_3(\C) \simeq \lsl_2(\C)$
et $d_j$ est pair pour tout $j \in \{1,\cdots, m\}$.
\end{enumerate}
\end{lemma}
\subsubsection{Cas exceptionnels}
Supposons dans ce paragraphe que $\g$ soit l'une des algèbres de Lie exceptionnelles
$\mathbf{G}_2$, $\mathbf{F}_4$ ou $\mathbf{E}_6$.
À l'aide de \cite[Chapter 13]{Ca}, on détermine les éléments nilpotents $e$ de $\g$
tels que $\rk \gs = 1$. On liste ces éléments dans les Tables 
\ref{fig:g2}, \ref{fig:f4} et \ref{fig:e6}, correspondantes à 
$\mathbf{G}_2$, $\mathbf{F}_4$ et $\mathbf{E}_6$ respectivement.
Dans chacune d'elles, la première colonne donne le label de l'orbite de $e$
dans la classification de Bala-Carter, la deuxième son diagramme de Dynkin pondéré,
et la troisième sa dimension.
\begin{table}[h!]
\centering
\small
\begin{tabular}{|c|c|c|}
\hline
Label & $\underset{\begin{Dynkin}
\Dbloc{\Dcirc \Deast \Ddoubleeast \Drightarrow}
\Dbloc{\Dwest \Ddoublewest \Dcirc}
\end{Dynkin}}{\text{Diagramme}}$& $\dim G.e$ \\
\hline
$A_1$ & $\begin{Dynkin}
        \Dbloc{\Dtext{}{1}}
        \Dbloc{\Dtext{}{0}}
        \end{Dynkin}$& $6$ \\
\hline
$\tilde{A}_1$ & $\begin{Dynkin}
        \Dbloc{\Dtext{}{0}}
        \Dbloc{\Dtext{}{1}}
        \end{Dynkin}$& $8$ \\
\hline
\end{tabular}

\smallskip

\caption{Type $\mathbf{G}_2$.}\label{fig:g2}
\end{table}

\begin{table}[h!]
\centering
\small
\begin{tabular}{|c|c|c|}
\hline
Label & $\underset{\begin{Dynkin}
\Dbloc{\Dcirc \Deast}
\Dbloc{\Dwest \Dcirc \Ddoubleeast \Drightarrow}
\Dbloc{\Ddoublewest \Dcirc \Deast}
\Dbloc{\Dwest \Dcirc}
\end{Dynkin}}{\text{Diagramme}}$& $\dim G.e$ \\
\hline
$A_2 + \tilde{A}_1$ & $\begin{Dynkin}
        \Dbloc{\Dtext{}{0}}
        \Dbloc{\Dtext{}{0}}
        \Dbloc{\Dtext{}{1}}
        \Dbloc{\Dtext{}{0}}
        \end{Dynkin}$& $34$ \\
\hline
$\tilde{A}_2+A_1$ & $\begin{Dynkin}
        \Dbloc{\Dtext{}{0}}
        \Dbloc{\Dtext{}{1}}
        \Dbloc{\Dtext{}{0}}
        \Dbloc{\Dtext{}{1}}
        \end{Dynkin}$& $36$ \\
\hline
$C_3(a_1)$ & $\begin{Dynkin}
        \Dbloc{\Dtext{}{1}}
        \Dbloc{\Dtext{}{0}}
        \Dbloc{\Dtext{}{1}}
        \Dbloc{\Dtext{}{0}}
        \end{Dynkin}$& $38$ \\
\hline
$B_3$ & $\begin{Dynkin}
        \Dbloc{\Dtext{}{2}}
        \Dbloc{\Dtext{}{2}}
        \Dbloc{\Dtext{}{0}}
        \Dbloc{\Dtext{}{0}}
        \end{Dynkin}$& $42$ \\
\hline
$C_3$ & $\begin{Dynkin}
        \Dbloc{\Dtext{}{1}}
        \Dbloc{\Dtext{}{0}}
        \Dbloc{\Dtext{}{1}}
        \Dbloc{\Dtext{}{2}}
        \end{Dynkin}$& $36$ \\
\hline
\end{tabular}

\smallskip

\caption{Type $\mathbf{F}_4$.}\label{fig:f4}
\end{table}

\begin{table}[h!]
\centering
\small
\begin{tabular}{|c|c|c|}
\hline
Label & $\underset{\begin{Dynkin}
\Dbloc{\Dcirc \Deast}
\Dbloc{\Dwest \Dcirc \Deast }
\Dbloc{\Dwest \Dcirc \Dsouth \Deast }
\Dbloc{\Dwest \Dcirc \Deast}
\Dbloc{\Dwest \Dcirc }\Dskip
\Dbloc{}\Dbloc{}
\Dbloc{\Dcirc \Dnorth}
\end{Dynkin}}{\text{Diagramme}}$& $\dim G.e$ \\
\hline
$2A_2 + A_1$ & $\begin{Dynkin}
        \Dbloc{\Dtext{}{1}}
        \Dbloc{\Dtext{}{0}}
        \Dbloc{\Dtext{}{1}}
        \Dbloc{\Dtext{}{0}}
        \Dbloc{\Dtext{}{1}} \Dskip
        \Dbloc{}\Dbloc{}
        \Dbloc{\Dtext{}{0}}
        \end{Dynkin}$& $54$ \\
\hline
$A_4 + A_1$ & $\begin{Dynkin}
        \Dbloc{\Dtext{}{1}}
        \Dbloc{\Dtext{}{1}}
        \Dbloc{\Dtext{}{0}}
        \Dbloc{\Dtext{}{1}}
        \Dbloc{\Dtext{}{1}} \Dskip
        \Dbloc{}\Dbloc{}
        \Dbloc{\Dtext{}{1}}
        \end{Dynkin}$& $62$ \\
\hline
$A_5$ & $\begin{Dynkin}
        \Dbloc{\Dtext{}{2}}
        \Dbloc{\Dtext{}{1}}
        \Dbloc{\Dtext{}{0}}
        \Dbloc{\Dtext{}{1}}
        \Dbloc{\Dtext{}{2}} \Dskip
        \Dbloc{}\Dbloc{}
        \Dbloc{\Dtext{}{1}}
        \end{Dynkin}$& $64$ \\
\hline
$D_5(a_1)$ & $\begin{Dynkin}
        \Dbloc{\Dtext{}{1}}
        \Dbloc{\Dtext{}{1}}
        \Dbloc{\Dtext{}{0}}
        \Dbloc{\Dtext{}{1}}
        \Dbloc{\Dtext{}{1}} \Dskip
        \Dbloc{}\Dbloc{}
        \Dbloc{\Dtext{}{2}}
        \end{Dynkin}$& $64$ \\
\hline
$D_5$ & $\begin{Dynkin}
        \Dbloc{\Dtext{}{2}}
        \Dbloc{\Dtext{}{0}}
        \Dbloc{\Dtext{}{2}}
        \Dbloc{\Dtext{}{0}}
        \Dbloc{\Dtext{}{2}} \Dskip
        \Dbloc{}\Dbloc{}
        \Dbloc{\Dtext{}{2}}
        \end{Dynkin}$& $68$ \\
\hline
\end{tabular}

\smallskip

\caption{Type $\mathbf{E}_6$.}\label{fig:e6}
\end{table}
Désormais, $\Ga: \g = \bigoplus_{j \in \Z} \g_j$ est une $\Z$-graduation admissible pour $e$
définie par un élément semisimple $h_\Ga$ de $\g$ et $(e,h,f)$ est un $\lsl_2$-triplet de $\g$ tel que
$h \in \g_0$ et $f \in \g_{-a}$ (cf. \cite[Proposition 32.1.7]{TY}). On pose $\ss := \Vect(e,h,f)$
et $t:= h_\Ga -\frac{a}{2} h$.
On désigne par $\gl(V)$ l'algèbre de Lie des endomorphismes de $V$.
On suppose enfin que $\rk \gs =1$.
\subsection{Cas où $\g = \lsl(V)$} \label{S:sl}
Supposons que $\g$ soit
l'algèbre de Lie $\lsl(V)$ où $V$ est un $\C$-espace vectoriel de dimension
finie $n \in \N^*$.
D'après le Lemme \ref{l:esl}, comme $\rk \gs =1$, la partition associée à $e$ est $(d_1,d_2)$.
L'élément $t$ appartient à $\gs$ et il est semisimple.
En tant que $\ss$-module, $V$ se décompose en composantes isotypiques.
Comme $t$ commute avec $\ss$, d'après le lemme de Schur, $t$ laisse
stable chaque composante isotypique. 
Puisque $t$ est semisimple, sa restriction à chaque composante isotypique
est semisimple.
De plus, comme $t \in \gs$ on a la décomposition
\begin{equation} \label{eq:decmodsim}
 V = V_1 \oplus V_2
\end{equation}
où $V_i$ est un $\ss$-module simple de dimension $d_i$
et $V_i$ est stable par $\ad t$ pour tout $i \in \{1,2\}$.
Pour tout $i \in \{1,2\}$, il existe $\a_i \in \C$ tel que $t|_{V_i} = \a_i \id_{V_i}$
et on a $\a_1 d_1 + \a_2 d_2  = 0$.
Rappelons l'identification en tant que $\gl(V)$-modules
$$
   \gl(V) = V^* \otimes V
$$
où pour $(\phi,v) \in V^* \times V$ l'endomorphisme $\phi \otimes v$ de $V$ est défini par
$\phi \otimes v(x) = \phi(x)v$ pour tout $x \in V.$ 
Par suite,
$$
    \gl(V) = \bigoplus_{1 \le i,j \le 2} \V_{i,j} \quad \text{ où } \quad \V_{i,j} := V_i^* \otimes V_j.
$$
Alors $\V_{i,j}$ est un $\ss$-module via l'opération adjointe dans $\gl(V)$
et il est $\ad t$-stable, avec
$$
    \ad t|_{\V_{i,j}} = (\a_j-\a_i) \id_{\V_{i,j}}.
$$
De plus on a
\begin{equation} \label{eq:coupijji}
	\la.,.\ra |_{\V_{i,j}\times \V_{k,l}} = 0 \  \text{ si } \  (i,j) \neq (l,k)  \  \text{ et } \ 
	\la.,.\ra |_{\V_{i,j}\times \V_{j,i}} \ 
	 \text{ est non dégénérée.}
\end{equation}
Autrement dit, $\V_{i,j}$ et $\V_{j,i}$ sont en couplage par rapport à la forme de Killing.

On peut donc faire une étude directe des valeurs propres de $\ad h_\Ga$ sur les $\ss$-modules $\V_{i,j}$.
On en déduit que le nombre de valeurs propres de $\ad h_\Ga$
appartenant à $]-a,0[$ est au plus égal à $2$. 
D'après le Théorème \ref{t:CONN} et le Théorème \ref{t:2vp}, on peut conclure.
\begin{theorem} \label{t:equisl}
Lorsque $\g = \lsl(V)$
et $\rk \gs = 1$, les paires $e$-admissibles
sont équivalentes entre elles. En particulier, les $W$-algèbres associées sont isomorphes.
\end{theorem}
\subsection{Cas où $\g = \so(V)$ ou $\sp(V)$} \label{S:so}
On traite dans ce paragraphe le cas où $\g = \so(V)$, et on montre le résultat
de façon analogue si $\g = \sp(V)$.
Supposons donc que $\g$ soit l'algèbre de Lie $\so(V)$.
On conserve les notations utilisées dans la Section \ref{S:eltnil}.
La forme bilinéaire non dégénérée $\Phi$ induit un isomorphisme $\b: V \rightarrow V^*$
qui associe à un élément $v$ de $V$ la forme linéaire $\Phi(v, \cdot)$. Les isomorphismes $\b$ et $\b^{-1}$ sont
des isomorphismes de $\so(V)$-modules. On a l'identification de $\so(V)$-modules
$$
  \gl(V) = V^* \otimes V \simeq V \otimes V.
$$
En effet, la structure de $\so(V)$-module sur $\gl(V)$ est donnée
par l'opération adjointe dans $\gl(V)$. De plus, pour tous $x,v,w \in V$,
l'identification est donnée par
$$
  (v \otimes w)(x) = \Phi(v,x)w.
$$
Pour $v,w \in V$, on note $v \wedge w = v \otimes w - w \otimes v$ et
$
  \bigwedge \nolimits^2 V = \Vect(v \wedge w \, ; \; v,w \in V)
$
la composante de degré $2$ de la graduation naturelle
de l'algèbre extérieure de $V$.
On a $v \wedge w \in \so(V)$ pour tout $v,w \in V$.
On en déduit pour des raisons de dimension 
$$
    \so(V) = \bigwedge \nolimits^2 V.
$$
{\bf Cas (a) ou (b).}
Supposons tout d'abord que la partition associée à $e$ vérifie l'une des conditions (a) ou (b) du Lemme \ref{l:eso}.
On peut supposer qu'elle est de la forme $(d_1,d_2, d_3, \ldots, d_m)$, avec $d_1 = d_2$, où
on ne suppose plus que les $d_i$ soient croissants.

Comme $t$ est un élément semisimple de $\gs$, on a la décomposition
$$
    V = V_1 \oplus V_2 \oplus \cdots \oplus V_m
$$
où pour tout $i \in \{1, \ldots , m\}$, $V_i$ est un $\ss$-module
simple de dimension $d_i$, avec $t|_{V_1} = \a_1 \id_{V_1}$,
$t|_{V_2} = \a_2 \id_{V_2}$ et $t|_{V_i} = 0$ pour tout $i \ge 3$.
Comme la trace de $t$ est nulle, $\a_1 + \a_2 =0$. On pose donc $\a := \a_1 = - \a_2$.
Ici, $V_1 \oplus V_2, V_3, \cdots, V_m$ sont les composantes isotypiques
du $\ss$-module $V$, et définissent une décomposition orthogonale par rapport à $\Phi$.
Observons que $V_1$ et $V_2$ sont totalement isotropes par rapport à $\Phi$.
On a 
$$
    \bigwedge \nolimits^2 (V_1 + V_2) =  \W_0 \oplus \W_+ \oplus \W_-
$$
où $$\W_0 = \Vect(v \wedge w \, ; \; v \in V_1, \, w \in V_2) \simeq \gl(V_1),$$
$$\W_+ \simeq \bigwedge \nolimits^2 V_1 = \Vect(v \wedge w \, ; \; v,w \in V_1) \ 
\text{ et } \  \W_- \simeq \bigwedge \nolimits^2 V_2 = \Vect(v \wedge w \, ; \; v,w \in V_2).$$
On note 
$$
	\W_1 = \W_+ \oplus \W_-.
$$
On a donc la décomposition orthogonale
\begin{equation} \label{eq:decompso}
    \so(V) = \bigwedge \nolimits^2 V = \W_0 \oplus \W_1 \oplus 
	  \bigoplus_{3 \le i \le m} \V_{1,i} \oplus \bigoplus_{3 \le i \le m} \V_{2,i}
	      \oplus \bigoplus_{3 \le i \le j \le m} \V_{i,j},
\end{equation}
avec $\V_{i,i} \simeq \bigwedge \nolimits^2 V_i$ et $\V_{i,j} \simeq V_i \otimes V_j$ pour $i < j$,
des isomorphismes en tant que $\ss$-modules.
De plus, $\W_+$ et $\W_-$ sont en couplage par rapport à la forme de Killing. De plus, les crochets
suivants s'annulent: $[\W_{\pm},\W_{\pm} ]$, $[\W_{0},\W_{\pm} ]$, $[\W_{\pm},\V_{1,i} ]$
et $[\W_{\pm},\V_{2,i} ]$ pour $i \ge 3$.

Notre objectif est d'obtenir l'équivalence des paires 
$e$-admissibles. Compte tenu de la Remarque \ref{rk:uniqopt}, on va s'intéresser
au cas où $\g_{< 0} \cap \g^e \neq \{0\}$.  D'après ce qui précède, on peut faire une étude
directe des valeurs propres de $\ad h_\Ga$ sur les sous-espaces $\W_0$, $\W_1$ et $\V_{i,j}$ pour
tous $i,j$.
On résume dans la Table \ref{fig:vp-a0} l'ensemble des valeurs propres de $\ad h_\Ga$
appartenant à $]-a, 0[$ réparties dans chacun des sous-espaces $\W_+,\W_-,\V_{1,i}$ et $\V_{2,i}$.
On encadre dans cette table la valeur propre $k$ telle que $\g_k \cap \g^e \not= \{0\}$.
\begin{table}[h]
\centering
{\renewcommand{\arraystretch}{1.5}
\begin{tabular}{|c|c|c|c|c|c|}
 \hline
 $d_1$ & $\a$ & $\W_+$ & $\W_-$ & $\V_{1,i}$ & $\V_{2,i}$ \\
 \hline
 pair & $]-\frac{a}{2}, 0[$ & $\boxed{2\a}$ & $-2\a-a$ & $-\frac{a}{2} + \a$ & $-\frac{a}{2} - \a$ \\
 \hline
 pair & $]0, \frac{a}{2} [$ & $2 \a -a$ & $\boxed{-2\a}$ & $-\frac{a}{2} + \a$ & $-\frac{a}{2} - \a$ \\
 \hline
 impair & $]-a, -\frac{a}{2} [$ & $\boxed{2 \a +a}$ & $-2\a -2a$ & $\a$ & $-a-\a$ \\
 \hline
 impair & $] \frac{a}{2},a [$ & $2\a - 2a$ & $\boxed{-2\a +a}$ & $-a+\a$ & $-\a$ \\
\hline 
\end{tabular}}

\smallskip

 \caption{Cas (a) ou (b): valeurs propres de $\ad h_\Ga$ dans $]-a, 0[$.} \label{fig:vp-a0}
\end{table}

Afin de montrer l'équivalence des paires $e$-admissibles, on distingue
deux cas: le cas où le nombre de valeurs propres de $\ad h_\Ga$ appartenant à $]-a,0[$
est égal à $4$, et le cas où il est strictement inférieur à $4$.

{\bf \underline{Cas I:}} le nombre de valeurs propres de $\ad h_\Ga$ appartenant à $]-a,0[$
est égal à $4$. On traite ici uniquement le cas où $d_1$ est pair et $ \a < 0$, les autres cas
se traitent de façon analogue.
Soit $(\m,\n)$ une paire $e$-admissible de $\g$ relativement à $\Ga$. Alors $(\m,\n)$
peut s'écrire de la forme suivante:
$$\m =  \g_{\le -a} \oplus U' \oplus V'  \oplus W' \oplus X' \ \text{ et } \ 
\n =  \g_{\le -a} \oplus U \oplus V  \oplus W \oplus X,$$
où $U' \subset U \subset \g_{2\a}$, $V' \subset V \subset \g_{-2\a-a}$,
$W' \subset W \subset  \g_{-\frac{a}{2} + \a}$
et $X' \subset X \subset  \g_{-\frac{a}{2} -\a}$.

\noindent On pose
$$ \m' :=   \g_{\le -a} \oplus V'  \oplus W' , \  
\n' := \g_{\le -a} \oplus U \oplus \g_{-2\a-a}  \oplus  \g_{-\frac{a}{2} + \a} \oplus X$$
$$\text{ et } \  \m'' :=  \g_{\le -a} \oplus \g_{-2\a-a}  \oplus  \g_{-\frac{a}{2} + \a}.$$
On montre alors de façon directe que les paires $(\m',\n')$ et $(\m'',\m'')$ sont
$e$-admissibles relativement à la graduation $\Ga$. 
On peut donc conclure car
$$ (\m,\n) \peq_\Ga (\m',\n') \seq_\Ga (\m'',\m'').$$

{\bf \underline{Cas II:}} le nombre de valeurs propres de $\ad h_\Ga$ appartenant à $]-a,0[$
est strictement inférieur à $4$.
On représente dans la Table \ref{fig:poss}
les différentes possibilités.
\begin{table}[h]
\centering
\renewcommand{\arraystretch}{1.5}
\begin{tabular}{|c|c|c|c|c|c|}
\hline
 $d_1$ & $\a$ & $\W_+$ & $\W_-$ & $\V_{1,i}$ & $\V_{2,i}$ \\
\hline
pair & $-\frac{a}{4}$ & $-\frac{a}{2}$ & $-\frac{a}{2}$ & $-\frac{3a}{4}$ & $-\frac{a}{4}$ \\
\hline
pair & $-\frac{a}{6}$ & $-\frac{a}{3}$ & $-\frac{2a}{3}$ & $-\frac{2a}{3}$ & $-\frac{a}{3}$ \\
\hline
pair & $\frac{a}{4}$ & $-\frac{a}{2}$ & $-\frac{a}{2}$ & $-\frac{a}{4}$ & $-\frac{3a}{4}$ \\
\hline
pair & $\frac{a}{6}$ & $-\frac{2a}{3}$ & $-\frac{a}{3}$ & $-\frac{a}{3}$ & $-\frac{2a}{3}$ \\
\hline
impair & $-\frac{3a}{4}$ & $-\frac{a}{2}$ & $-\frac{a}{2}$ & $-\frac{3a}{4}$ & $-\frac{a}{4}$ \\
\hline
impair & $-\frac{2a}{3}$ & $-\frac{a}{3}$ & $-\frac{2a}{3}$ & $-\frac{2a}{3}$ & $-\frac{a}{3}$ \\
\hline
impair & $\frac{3a}{4}$ & $-\frac{a}{2}$ & $-\frac{a}{2}$ & $-\frac{a}{4}$ & $-\frac{3a}{4}$ \\
\hline
impair & $\frac{2a}{3}$ & $-\frac{2a}{3}$ & $-\frac{a}{3}$ & $-\frac{a}{3}$ & $-\frac{2a}{3}$ \\
\hline
\end{tabular}

\smallskip

\caption{Cas (a) ou (b): possibilités du cas II.}\label{fig:poss}
\end{table}

Lorsque $\a \in \{\pm \frac{a}{6}, \pm \frac{2a}{3}\}$, on a
$
	\g_{<0} = \g_{\le -a} \oplus \g_{- \frac{2a}{3}} \oplus \g_{ -\frac{a}{3} } .
$
On conclut grâce au Théorème \ref{t:2vp}.

Lorsque $\a \in \{\pm \frac{a}{4}, \pm \frac{3a}{4}\}$, on remarque que la
graduation $\Ga$ est $\frac{a}{2}$-optimale. On conclut alors grâce au Théorème \ref{t:equiopt}.

{\bf Cas (c).}
Supposons que la partition associée à $e$ vérifie la condition (c) du Lemme \ref{l:eso}.
On peut supposer qu'elle est de la forme $(d_0,d_0,d_0,d_1, \ldots, d_m)$, où
on ne suppose plus que les $d_i$ soient croissants.

Comme $t$ est un élément semisimple de $\gs$, on a la décomposition en composantes isotypiques
du $\ss$-module $V$
$$
    V = V_0 \oplus V_1 \oplus \cdots \oplus V_m.
$$
Pour tout $i \in \{1, \ldots , m\}$, $V_i$ est un $\ss$-module
simple de dimension $d_i$, avec $t|_{V_i} = 0$.
De plus, $V_0$ admet une décomposition 
$$ V_0 = \W_0 \oplus \W_+ \oplus \W_-$$
où $\W_0$ est un sous-espace régulier relativement à $\Phi$
et $\W_+$ et $\W_-$ sont deux sous-espaces totalement isotropes
par rapport à $\Phi$. En particulier, $t|_{\W_0} = 0$, $t|_{\W_+} = \a_1 \id_{\W_+}$ et
$t|_{\W_-} = \a_2 \id_{\W_-}$.
Comme la trace de $t$ est nulle, $\a_1 + \a_2 =0$. On pose donc $\a := \a_1 = - \a_2$.
On a donc la décomposition
$$ \so(V) =  \W_{0,0} \oplus \W_{0,+} \oplus \W_{0,-} \oplus  
\W_{+,+} \oplus \W_{-,-} \oplus \W_{+,-} \oplus \bigoplus_{i \ge 1} (\W_{0,i} \oplus 
\W_{+,i} \oplus \W_{-,i}) \oplus \bigoplus_{1 \le i \le j \le m} \V_{i,j},$$
avec $\W_{0,0}  \simeq \bigwedge \nolimits^2  \W_0$, $\W_{\pm,\pm}  \simeq \bigwedge \nolimits^2  \W_\pm$, 
$\W_{0,\pm} \simeq \W_0 \otimes \W_\pm$, $\W_{+,-} \simeq \W_+ \otimes \W_-$, 
$\V_{i,i}  \simeq \bigwedge \nolimits^2  V_i$ pour tout $i \ge 1$,
$\W_{0,i} \simeq \W_0 \otimes V_i$, $\W_{\pm,i} \simeq \W_\pm \otimes V_i$
et $\V_{i,j} \simeq V_i \otimes V_j$ pour tous $1 \le i <j$.
D'après ce qui précède, les crochets suivants s'annulent: $[\W_{\pm,\pm},\W_{\pm,\pm}]$,
$[\W_{\pm,i},\W_{\pm,\pm}]$, $[\W_{\pm,i},\W_{0,\pm}]$ et $[\W_{0,\pm},\W_{\pm,\pm}]$.

On peut donc faire une étude directe des valeurs propres de $\ad h_\Ga$ sur ces différents sous-espaces.
Comme $\g_{\le -a} \cap \g^e = \{0\}$, on a $-a < \a < a$.
Pour montrer l'équivalence des paires $e$-admissibles relativement à $\Ga$, on considère le cas où $\a > 0$. L'autre cas se traite de façon analogue.
On résume dans la Table \ref{fig:2emecasvp-a0} l'ensemble des valeurs propres de $\ad h_\Ga$
appartenant à $]-a, 0[$ dans les différents sous-espaces de la décomposition précédente.
On encadre dans cette table la valeur propre $k$ telle que $\g_k \cap \g^e \not= \{0\}$.
\begin{table}[h]
\centering
{\renewcommand{\arraystretch}{1.5}
\begin{tabular}{|c|c|c|c|c|c|c|}
 \hline
 $\a$ & $\W_{+,i}$ & $\W_{-,i}$ & $\W_{0,+}$ & $\W_{0,-}$ & $\W_{+,+}$ & $\W_{-,-}$ \\
 \hline
$]\frac{a}{2}, a[$ & $\a-a$ & $-\a$ & $\a-a$ & $\boxed{-\a}$ & $2\a - 2a$ & $\boxed{-2\a +a}$ \\
 \hline
 $]0, \frac{a}{2} [$ & $\a-a$ & $-\a$ & $\a-a$ & $\boxed{-\a}$ & $2\a - a$ & $-2\a $ \\
 \hline
 $\frac{a}{2}$ & $-\frac{a}{2}$ & $-\frac{a}{2}$& $-\frac{a}{2}$& $\boxed{-\frac{a}{2}}$ && \\
 \hline
\end{tabular}}

\smallskip

 \caption{Cas (c): valeurs propres de $\ad h_\Ga$ dans $]-a, 0[$.} \label{fig:2emecasvp-a0}
\end{table}

Lorsque $\a \in ]0, \frac{a}{2} [$, on peut montrer l'équivalence des paires $e$-admissibles de
façon analogue à la démarche utilisée dans le paragraphe précédent.
Lorsque $\a =\frac{a}{2}$, la graduation est $\frac{a}{2}$-optimale, on peut donc conclure grâce au
Théorème \ref{t:equiopt}.
Supposons donc que $\a \in ]\frac{a}{2},a [$.
Afin de montrer l'équivalence des paires $e$-admissibles, on distingue
deux cas: le cas où le nombre de valeurs propres de $\ad h_\Ga$ appartenant à $]-a,0[$
est égal à $4$, et le cas où il est strictement inférieur à $4$.

{\bf \underline{Cas I:}} le nombre de valeurs propres de $\ad h_\Ga$ appartenant à $]-a,0[$
est égal à $4$. 
Soit $(\m,\n)$ une paire $e$-admissible de $\g$ relativement à $\Ga$. Alors $(\m,\n)$
peut s'écrire de la forme suivante:
$$\m =  \g_{\le -a} \oplus U' \oplus V'  \oplus W' \oplus X' \ \text{ et } \ 
\n =  \g_{\le -a} \oplus U \oplus V  \oplus W \oplus X,$$
où $U' \subset U \subset \g_{\a-a}$, $V' \subset V \subset \g_{-\a}$,
$W' \subset W \subset  \g_{2 \a-2a}$
et $X' \subset X \subset  \g_{ -2\a+a}$.
Soit $ \tilde{V}$ un sous-espace de $\g_{-\a}$ supplémentaire de $\g_{-\a} \cap \g^e$
contenant $V$.
On pose
$$ \m' :=   \g_{\le -a} \oplus V'  \oplus W' , \  
\n' := \g_{\le -a} \oplus U \oplus \tilde{V}  \oplus  \g_{2 \a-2a} \oplus X
\text{ et } \  \m'' :=  \g_{\le -a} \oplus \tilde{V}  \oplus  \g_{2 \a-2a}.$$
On montre par calculs directs que les paires $(\m',\n')$ et $(\m'',\m'')$ sont
$e$-admissibles relativement à la graduation $\Ga$. De plus, elles vérifient
$$ (\m,\n) \peq_\Ga (\m',\n') \seq_\Ga (\m'',\m'').$$
On peut donc conclure grâce au lemme suivant:
\begin{lemma}
Soient $U$ et $ V$ deux sous-espaces de $\g_{-\a}$ supplémentaires de $\g_{-\a} \cap \g^e$.
Les paires $\P_U$ et $\P_V$ sont équivalentes entre elles, où  
pour $W \in \{U,V\}$, $\P_W$ désigne la paire $e$-admissible  
$(\g_{\le -a} \oplus W \oplus  \g_{2 \a-2a},\g_{\le -a} \oplus W  \oplus  \g_{2 \a-2a})$.
\end{lemma}
\begin{proof}
Montrons le lemme par récurrence sur $n := \codim_U U\cap V$.
On se contente de le montrer dans le cas où $n =1$, le reste de la récurrence étant similaire
à celui de la démonstration du Lemme \ref{l:recur}.
On pose $$D := \g_{\a-a} \cap [e,U \cap V]^\perp \ \text{ , } \ \m:= \g_{\le -a}  \oplus D \oplus U \cap V  \oplus \g_{2 \a-2a},$$
$$\text{ et } \mathbf{A}_W := (\g_{\le -a}  \oplus U \cap V  \oplus \g_{2 \a-2a}, \g_{\le -a}  \oplus D \oplus W  \oplus \g_{2 \a-2a} ),$$
pour $W \in \{U,V\}$.
On vérifie sans peine que les paires $\mathbf{A}_U$, $\mathbf{A}_V$ et $(\m,\m)$ sont $e$-admissibles
relativement à la graduation $\Ga$. Le lemme s'ensuit car
$$\P_U \peq_\Ga \mathbf{A}_U  \seq_\Ga (\m,\m) \peq_\Ga \mathbf{A}_V  \seq_\Ga \P_V.$$ 
\end{proof}

{\bf \underline{Cas II:}} le nombre de valeurs propres de $\ad h_\Ga$ appartenant à $]-a,0[$
est strictement inférieur à $4$.
On représente dans la Table \ref{fig:poss2} les deux possiblités.
\begin{table}[h]
\centering
\renewcommand{\arraystretch}{1.5}
\begin{tabular}{|c|c|c|c|c|}
\hline
 $\a$ & $\W_{+,i} + \W_{0,+}$ & $\W_{-,i} + \W_{0,-}$ & $\W_{+,+}$ & $\W_{-,-}$ \\
\hline
 $\frac{2a}{3}$ & $-\frac{a}{3}$ & $-\frac{2a}{3}$ & $-\frac{2a}{3}$ & $-\frac{a}{3}$ \\
\hline
 $\frac{3a}{4}$ & $-\frac{a}{4}$ & $-\frac{3a}{4}$ & $-\frac{a}{2}$ & $-\frac{a}{2}$ \\
\hline
\end{tabular}

\smallskip

\caption{Cas (c): possibilités du cas II.}\label{fig:poss2}
\end{table}

Lorsque $\a = \frac{2a}{3}$, on conclut d'après le Théorème \ref{t:2vp}.
Supposons désormais que $\a = \frac{3a}{4}$.
Soit $(\m_1,\n_1)$ une paire $e$-admissible de $\g$ relativement à $\Ga$. Alors $(\m_1,\n_1)$
peut s'écrire de la forme suivante:
$$\m_1 =  \g_{\le -a} \oplus U' \oplus V'  \oplus W' \ \text{ et } \ 
\n_1 =  \g_{\le -a} \oplus U \oplus V  \oplus W,$$
où $U' \subset U \subset \g_{-\frac{3a}{4}}$, $V' \subset V \subset \g_{-\frac{a}{2}}$ et
$W' \subset W \subset  \g_{-\frac{a}{4}}$.
Soient $\tilde{U}$ un sous-espace de $ \g_{-\frac{3a}{4}}$ supplémentaire de $ \g_{-\frac{3a}{4}}^e$
contenant $U$ et $\tilde{V}$ un sous-espace de $ \g_{-\frac{a}{2}}$ supplémentaire de $ \g_{-\frac{a}{2}}^e$ contenant $\g_{-\frac{a}{2}} \cap \W_{+,+}$.
On pose
$$\m_2 :=  \g_{\le -a} \oplus U' \oplus V' \ \text{ , } \  \n_2:= \g_{\le -a} \oplus \tilde{U} \oplus V  \oplus W,$$
$$\m_3 :=   \g_{\le -a} \oplus \tilde{U} \oplus V'  \  \text{ , } \ \n_3:= \g_{\le -a} \oplus \tilde{U} \oplus V,$$
$$\m_4 :=  \g_{\le -a} \oplus \tilde{U} \ \text{ , } \ \n_4:= \g_{\le -a} \oplus \tilde{U} \oplus \tilde{V},$$
$$ \text{ et } \  \m_5 :=  \g_{\le -a} \oplus \tilde{U} \oplus \g_{-\frac{a}{2}} \cap \W_{+,+}.$$
On vérifie de façon directe que les paires $(\m_2,\n_2)$, $(\m_3,\n_3)$, $(\m_4,\n_4)$ et $(\m_5,\m_5)$
sont $e$-admissibles relativement à la graduation $\Ga$. De plus, on a
$$(\m_1,\n_1) \peq_\Ga (\m_2,\n_2) \seq_\Ga (\m_3,\n_3) \peq_\Ga (\m_4,\n_4) \seq_\Ga (\m_5,\m_5).$$
On déduit alors l'équivalence des paires $e$-admissibles dans ce cas à partir du lemme suivant.
\begin{lemma}
Soient $U$ et $ V$ deux sous-espaces de  $ \g_{-\frac{3a}{4}}$ supplémentaires de $ \g_{-\frac{3a}{4}}^e$.
Les paires $\P_U$ et $\P_V$ sont équivalentes entre elles, où  
pour $X \in \{U,V\}$, $\P_X$ désigne la paire $e$-admissible  
$(\g_{\le -a} \oplus X \oplus \g_{-\frac{a}{2}} \cap \W_{+,+},\g_{\le -a} \oplus X \oplus \g_{-\frac{a}{2}} \cap \W_{+,+})$.
\end{lemma}
\begin{proof}
Montrons le lemme par récurrence sur $n := \codim_U U\cap V$.
On se contente de le montrer dans le cas où $n =1$, le reste de la récurrence étant similaire
à celui de la démonstration du Lemme \ref{l:recur}.
Pour simplifier, on note $\g_{-\frac{a}{2}}^+ $ le sous-espace $\g_{-\frac{a}{2}} \cap \W_{+,+}$.
On pose $$D := \g_{-\frac{a}{4}} \cap [e,U \cap V]^\perp \ \text{ , } \ \m:= \g_{\le -a}  \oplus U \cap V  \oplus \g_{-\frac{a}{2}}^+\oplus D,$$
$$\text{ et } \mathbf{A}_X := (\g_{\le -a}  \oplus U \cap V  \oplus \g_{-\frac{a}{2}}^+, \g_{\le -a}  \oplus X \oplus\g_{-\frac{a}{2}}^+ \oplus D ),$$
pour $X \in \{U,V\}$.
On vérifie sans peine que les paires $\mathbf{A}_U$, $\mathbf{A}_V$ et $(\m,\m)$ sont $e$-admissibles
relativement à la graduation $\Ga$. Le lemme s'ensuit car
$$\P_U \peq_\Ga \mathbf{A}_U  \seq_\Ga (\m,\m) \peq_\Ga \mathbf{A}_V  \seq_\Ga \P_V.$$ 
\end{proof}
On a donc montré le théorème suivant.
\begin{theorem} \label{t:equiso}
Lorsque $\g = \so(V)$
et $\rk \gs = 1$, les paires $e$-admissibles
sont équivalentes entre elles. En particulier, les $W$-algèbres associées sont isomorphes.
\end{theorem}
De façon analogue, on montre un résultat similaire lorsque $\g = \sp(V)$.
\begin{theorem} \label{t:equisp}
Lorsque $\g = \sp(V)$
et $\rk \gs = 1$, les paires $e$-admissibles
sont équivalentes entre elles. En particulier, les $W$-algèbres associées sont isomorphes.
\end{theorem}
\subsection{Cas où $\g = \G_2$, $\mathbf{F}_4$ ou $\E_6$}\label{S:excep}
Supposons que $\g = \mathbf{G}_2$, $\mathbf{F}_4$ ou $\mathbf{E}_6$.
Comme $\g$ est un $\ss$-module, $\g$ se décompose en composantes isotypiques
de $\ss$-modules simples. L'élément $t$ est semisimple et appartient à $\gs$.
Il laisse donc stable chaque composante isotypique. 
Puisque $t$ est semisimple, sa restriction à chaque composante
isotypique est semisimple. De nouveau comme $t \in \gs$, on a une décomposition
de $\g$ en $\ss$-modules simples stables par $\ad t$.
En conclusion, $\g$ se décompose en $\ad t$-espaces propres
$$
    \g =\bigoplus_{\l} \W_\l.
$$
Pour $\l \neq 0$, $\W_\l$ et $\W_{-\l}$ sont en couplage
par rapport à la forme de Killing.
De plus, chaque $\ad t$-espace propre se décompose en $\ss$-modules simples.
Afin d'étudier l'équivalence des paires $e$-admissibles, on traite
au cas par cas les orbites nilpotentes des Tables \ref{fig:g2},
\ref{fig:f4} et \ref{fig:e6}. Pour chacune d'entre elles, on peut expliciter
les décompositions ci-dessus à l'aide du logiciel GAP4. 
On étudie ensuite les valeurs propres de $\ad h_\Ga$ sur les espaces propres $\W_\l$.
En s'inspirant des constructions faites dans les sections précédentes, on montre que lorsque $\g = \mathbf{G}_2$, $\mathbf{F}_4$ ou $\mathbf{E}_6$
et $\rk \gs = 1$, les paires $e$-admissibles
sont équivalentes entre elles. En particulier, les $W$-algèbres associées sont isomorphes.

\end{document}